%% file: kenig-pcmi-with-solutions.tex
\documentclass[12pt,lecture]{pcms-l}

\input pcmslmod.tex   

\numberwithin{section}{chapter}
\numberwithin{equation}{chapter}
\theoremstyle{plain}
\newtheorem{theorem}[equation]{Theorem}
\newtheorem{lemma}[equation]{Lemma}

\newtheorem{corollary}[equation]{Corollary}

\theoremstyle{definition}
\newtheorem{definition}[equation]{Definition}
\newtheorem{example}[equation]{Example}
\newtheorem{fact}[equation]{Fact}

\theoremstyle{remark}
\newtheorem{claim}[equation]{Claim}
\newtheorem{remark}[equation]{Remark}

\usepackage{amsmath}
\usepackage{amssymb}
\usepackage{amsthm}
\usepackage{amsxtra}
\usepackage{smartref}

\newcommand{\tv}{| \! | \! |}
\newcommand{\nequiv}{ \equiv \mspace{-16mu} / \mspace{11mu}}
\newcounter{exnum}
\newenvironment{problems}
   {\setcounter{exnum}{0}
    \begin{list}{\textbf{Problem \arabic{chapter}.\arabic{exnum}.}}
            {\setlength{\leftmargin}{0in} 
             \setlength{\rightmargin}{0in}
             \setlength{\labelwidth}{0in} 
             \setlength{\labelsep}{.1in}
             \setlength{\itemindent}{.1in}
             \setlength{\itemsep}{0.5ex}}
     \usecounter{exnum}}
   {\end{list}}
\addtoreflist{chapter}

\begin{document}


\part*{The Cauchy Problem for the Quasilinear Schr\"{o}dinger Equation \\ (Following Kenig-Ponce-Vega)}
\pauth{Carlos E. Kenig}
\tableofcontents

\mainmatter
\setcounter{page}{1}

\LogoOn

\lectureseries[Quasilinear Schr\"{o}dinger Equation]{The Cauchy Problem for the Quasilinear Schr\"{o}dinger Equation \\ (Following Kenig-Ponce-Vega)}

\auth[C.~E. Kenig]{Carlos E. Kenig}

\address{Department of Mathematics, University of Chicago, Chicago, Illinois, USA 60637}
\email{cek@math.uchicago.edu}


\thanks{The author was supported in part by the National Science Foundation}

\setaddress

\lecture{The energy method}

In these lectures we will discuss the local in time well-posedness of the Cauchy problem for quasi-linear Schr\"{o}dinger equations, i.e.\ equations of the form
\begin{equation}\label{E:QLCP}
\text{(QLCP)} \left\{ 
\begin{aligned}
\partial_t u &= ia_{lk}(x,t;u,\bar{u},\nabla_xu,\nabla_x\bar{u}) \partial_{x_lx_k}^2u \\
& \quad + ib_{lk}(x,t;u,\bar{u},\nabla_xu,\nabla_x\bar{u}) \partial_{x_lx_k}^2\bar{u} \\
& \quad + \vec{b}_1(x,t;u,\bar{u},\nabla_xu,\nabla_x\bar{u})\cdot \nabla_xu \\
&\quad + \vec{b}_2(x,t;u,\bar{u},\nabla_xu,\nabla_x\bar{u})\cdot \nabla_x\bar{u} \\
& \quad + c_1(x,t;u,\bar{u})u+c_2(x,t;u,\bar{u})\bar{u} \\
& \quad + f(x,t) \\
u\big|_{t=0} &= u_0 
\end{aligned}
\right. \qquad  x\in \mathbb{R}^n, \; t\in [0,T]
\end{equation}
We will be assigning suitable ``ellipticity'' hypotheses on $\{ a_{lk}, b_{lk} \}$, smoothness on all the coefficients, ``asymptotic-flatness'' on the coefficients, and as we shall see a (necessary) ``non-trapping'' condition on a Hamiltonian flow obtained from the coefficients and the data $u_0$.  By ``local well-posedness'' in a space $B$, we mean that, given $u_0\in B$, $f\in X$, there exists $T=T(u_0,f)$, and a unique $u\in C([0,T]; B)$, such that $u$ solves the equation (in a suitable sense), $u(0,-)=u_0$, and the mapping $(u_0,f)\in B\times X\mapsto u\in C([0,T];B)$ is continuous.  In general, the space $B$ will be a Sobolev space, like 
$$H^s(\mathbb{R}^n)= \left\{ f\in \mathcal{S}' : \int (1+|\xi|^2)^s |\hat{f}(\xi)|^2 \, d\xi < \infty \right\}$$ 
or a ``weighted Sobolev space'' of the type $H^s(\mathbb{R}^n)\cap L^2( |x|^Ndx)$, whose presence will be explained later on.  It turns out that the classical theory of pseudo-differential operators ($\Psi$DO) is an appropriate and useful tool in this task, and we will review it and utilize it.

Equations of the type described above appear in several fields of physics, such as plasma fluids, classical and quantum ferromagnetism, laser theory, etc., and also in complex geometry, where, for example, in K\"{a}hler geometry they model ``Schr\"odinger flows''.  These equations are also analogous to corresponding ones for hyperbolic equations, where the corresponding results were obtained much earlier, in the 70's, by Kato and his collaborators \cite{Kato} \cite{HKM}.

The problem was extensively studied in the 90's, in the case of constant coefficients (i.e.\ $(x,t)$ independent) and semilinear non-linearity, i.e.\
\begin{equation}\label{E:SLCP}
\text{(SLCP)} \left\{ 
\begin{aligned}
\partial_t u &= i\Delta u + F(u,\bar{u},\nabla_xu,\nabla_x\bar{u}) \\
u\big|_{t=0} &= u_0 
\end{aligned}
\right. \qquad
x\in \mathbb{R}^n, \;
t\in [0,T]
\end{equation}
Let us first discuss SLCP \eqref{E:SLCP} in the case when there are no derivatives in the non-linearity, i.e.\ when $F(u,\bar{u},\nabla_xu,\nabla_x\bar{u})=G(u,\bar{u})$ with $G(0,0)=0$.  In this case, the energy method applies, and gives local-wellposedness in $H^s(\mathbb{R}^n)$ for $s>n/2$.  Since the energy method will be important to us in the sequel, let us work out this result.  

Thus, we assume $G(0,0)=0$, $G\in C^\infty(\mathbb{C}\times \mathbb{C})$, and we wish to show the local well-posedness of the Cauchy problem
\begin{equation}\label{E:3}
\left\{
\begin{aligned}
\partial_t u &= i\Delta u + G(u,\bar{u}) \\
u\big|_{t=0} &= u_0
\end{aligned} 
\right.
\end{equation}
in the Sobolev space $H^s(\mathbb{R}^n)$ for $s>n/2$.  To simplify the exposition, let us assume that $G$ is a polynomial, so that 
$$G(u,\bar{u})= \sum_{ \substack{ 0\leq j \leq M \\ 0 \leq k \leq N \\ (j,k) \neq (0,0) } } c_{jk}u^j\bar{u}^k$$
We will recall a few facts about Sobolev spaces.
\begin{fact}\label{F:1}
$\|u\|_{L^\infty(\mathbb{R}^n)} \leq C \|u\|_{H^s(\mathbb{R}^n)}$ for $s>n/2$.
\end{fact}
\begin{fact}\label{F:2}
For $s>n/2$, $H^s(\mathbb{R}^n)$ is an algebra under pointwise multiplication, i.e.\ 
$$\| f\cdot g \|_{H^s} \leq C \|f\|_{H^s}\|g\|_{H^s}$$
This is a consequence of Fact~\ref{F:1}. 
\end{fact}
\begin{fact} \label{F:3}
For $s>n/2$, if $G(0,0)=0$, $G$ is smooth, $s>n/2$, then 
$$\|G(u,\bar{u})\|_{H^s} \leq R( \| u \|_{H^s} )$$
where $R$ is an increasing function that depends on $G$, $s$, with $R(0)=0$.
  For instance, in our polynomial case, we have 
$$\|G(u,\bar{u})\|_{H^s} \leq C \{ \|u\|_{H^s}^{M+N}+\|u\|_{H^s} \}$$
\end{fact}

\underline{Step 1.  \textit{A priori} estimates}.  Assume that we have a sufficiently regular solution $u$ to \eqref{E:3}.  Let $\widehat{J^su}(\xi,t)=\hat{u}(\xi,t)(1+|\xi|^2)^{s/2}$.  We take \eqref{E:3} and rewrite it as
\begin{align}
\partial_t u &= i\Delta u + G(u,\bar{u}) \label{E:4}\\
\partial_t \bar{u} &= -i\Delta \bar{u} + \overline{G(u,\bar{u})} \label{E:5}
\end{align}
We now apply $J^s$ to both equations, multiply \eqref{E:4} by $\overline{J^su}=J^s\bar{u}$, multiply \eqref{E:5} by $J^su$, integrate both equations in $x$, and add.  We then have (with $v=J^su$)
$$
\partial_t \int |v|^2 = i \int [ \Delta v \bar{v} - \Delta \bar{v} v ] + \int J^sG(u,\bar{u}) \overline{J^su} + \int J^s\overline{G(u,\bar{u})}J^s u
$$
Since $i\int [\Delta v \bar{v} - \Delta \bar{v} v ] = 0$, this term drops out.  Using Fact \ref{F:3}, we obtain, with $f(t) = \|u(-,t)\|_{H^s}^2= \|v(-,t)\|_{L^2}^2$, 
\begin{align*}
 \left| \frac{d}{dt} f(t) \right| & \leq 2\| J^sG(u,\bar{u})\|_{L^2}\|v\|_{L^2} \\
&\leq C \{ \|u\|_{H^s}^{M+N} + \|u\|_{H^s} \}\|u\|_{H^s} \\
&\leq C \{ f(t) + f(t)^{(M+N+1)/2} \}
\end{align*}
We now define $f_1(t) = \sup_{0<r<t} \|u(-,r)\|_{H^s}^2$.  Then $\exists \; r_0$, $0\leq r_0 \leq t$ such that 
$$f_1(t) = f(r_0) = \int_0^{r_0} f'(r)\, dr + f(0) \leq \|u_0\|_{H^s}^2+Ctf_1(t)+Ctf_1(t)^\alpha$$
where $\alpha=\frac{M+N+1}{2}>1$.  For any $t\leq \frac{1}{2C}$, we obtain 
$$f_1(t) \leq 2 \|u_0\|_{H^s}^2 + 2Ctf_1(t)^\alpha$$
Let now $T_0= \text{first }t \leq \frac{1}{2C}$ for which $f_1(T_0) \geq 4 \|u_0\|_{H^s}^2$.  Since $f_1(t)$ is continuous, 
$$4\|u_0\|_{H^s}^2 = f_1(T_0)\leq 2 \|u_0\|_{H^s}^2 + 2CT_04^\alpha \|u_0\|_{H^s}^{2\alpha}$$
and so $T_0\geq \frac{1}{C4^\alpha\|u_0\|_{H^s}^{2\alpha-2}}$.  In other words, if $T_0= \min\left\{ \frac{1}{2C}, \frac{1}{C4^\alpha \|u_0\|_{H^s}^{2\alpha-2}} \right\}$, then for $0\leq t\leq T_0$, we have $\|u(-,t)\|_{H^s}^2 \leq 4 \|u_0\|_{H^s}^2$, which is our \textit{a priori} inequality. \\

\begin{remark} \label{R:1}
Suppose we considered solutions to 
\begin{equation} \label{E:6}
\left\{
\begin{aligned}
\partial_t u &= -\epsilon\Delta^2 u + i\Delta u + G(u,\bar{u}) \\
u\big|_{t=0} &= u_0
\end{aligned}
\right. \qquad \epsilon>0
\end{equation}
Then, the same conclusion holds, with $C$ independent of $\epsilon$.  In fact, we only need to understand $\epsilon \int [ \Delta^2v \cdot \bar{v} + \Delta^2 \bar{v} \cdot v ] \, dx = 2 \epsilon \int |\Delta v |^2 \geq 0$.  But then
\begin{align*}
\partial_t \int |v|^2 & = -2\epsilon \int |\Delta v|^2 + \int J^sG(u,\bar{u})\overline{J^s u} + \int \overline{J^sG(u,\bar{u})} J^s u \\
& \leq \int J^sG(u,\bar{u})\overline{J^su} + \int \overline{J^sG(u,\bar{u})} J^su 
\end{align*}
and we proceed as before.
\end{remark}

\underline{Step 2.  Existence of solutions}.  For each $\epsilon >0$, a solution $u^\epsilon$ on $[0,T_\epsilon]$ to \eqref{E:6} is obtained by ``standard parabolic theory''.  Specifically, let $s>n/2$, and define
\begin{align*}
X_{T,M_0} = \{ \quad & v: \mathbb{R}^n \times [0,T] \to \mathbb{C}, \; v\in C([0,T]; H^s), \; v(0)= u_0, \\
& \text{ and } \tv y \tv_T = \sup_{[0,T]} \|v(t)\|_{H^s} \leq M_0 \quad \}
\end{align*}
We then have: for any $u_0\in H^s$, $\|u_0\|_{H^s} \leq M_0/2$, there exists $T_\epsilon=O(\epsilon)$, depending only on $M_0$, $s$, $n$, $G$, and a unique solution $u^\epsilon$ in $X_{T_\epsilon, M_0}$ to
\begin{equation} \label{E:2}
\left\{
\begin{aligned}
\partial_t u &= -\epsilon \Delta^2 u + i \Delta u + G(u,\bar{u}) \\
u\big|_{t=0} &= u_0
\end{aligned}
\right.
\end{equation}
so that $\sup_{t\in[0,T_\epsilon]} \|u^\epsilon(t)\|_{H^s} \leq M_0$.  This is proved by converting \eqref{E:2} into the integral equation $\Gamma u^\epsilon =u^\epsilon$, where
$$\Gamma w(t) = e^{-\epsilon t \Delta^2}u_0 + \int_0^t e^{-\epsilon (t-t')\Delta^2} [ i\Delta w + G(w,\bar{w})] dt'$$
and showing that, for appropriate $T_\epsilon$, $\Gamma$ is a contraction on $X_{T_\epsilon, M_0}$.  The only estimate for the semigroup $\{ e^{-\epsilon t \Delta^2}, t\geq 0 \}$ that is needed is $\| \Delta e^{-\epsilon t \Delta^2} g \|_{L^2} \leq \frac{1}{\epsilon^{1/2}t^{1/2}} \|g\|_{L^2}$.  

Set $M_0=8\|u_0\|_{H^s}$.  Obtain, as above, a solution $u^\epsilon$ to \eqref{E:2} on $[0,T_\epsilon]$.  One then uses the \textit{a priori} estimate in Remark \ref{R:1}, to conclude that, if $T_\epsilon\leq T_0= \min \left\{ \frac{1}{2C}, \frac{1}{C4^\alpha\|u_0\|_{H^s}^{2\alpha-2}} \right\}$, one has $\sup_{[0,T_\epsilon]} \|u^\epsilon(t) \|_{H^s} \leq 4 \|u_0\|_{H^s} \leq \frac{M_0}{2}$.  We can then iterate this local existence result, in the interval $[T_\epsilon, 2T_\epsilon]$, etc., to find now a solution to \eqref{E:2} in $[0,T_0]$, $0<\epsilon < 1$, with $\sup_{[0,T_0]} \| u^\epsilon(t)\|_{H^s} \leq 4\|u_0\|_{H^s}$.  

Now consider $0<\epsilon'<\epsilon<1$, and let $u^\epsilon$, $u^{\epsilon'}$ be the corresponding solutions to \eqref{E:2}.  Set $v=u^\epsilon-u^{\epsilon'}$, so that
\begin{equation} \label{E:8}
\partial_t v = -(\epsilon-\epsilon')\Delta^2u^\epsilon - \epsilon'\Delta^2v + i \Delta v + [G(u^\epsilon, \bar{u}^\epsilon)-G(u^{\epsilon'}, \bar{u}^{\epsilon'})]
\end{equation}
Recall that $\sup_{[0,T_0]} \|u^\epsilon(t)\|_{L^\infty} \leq M_0$, and similarly for $u^{\epsilon'}$, and that $|G(u^\epsilon,\bar{u}^\epsilon)-G(u^{\epsilon'},\bar{u}^{\epsilon'})| \leq C_{M_0}|u^\epsilon-u^{\epsilon'}|$.  We then multiply \eqref{E:8} by $\bar{v}$, conjugate \eqref{E:8} and multiply by $v$, add, and integrate in $x$, to obtain
$$\partial_t \int |v|^2 \leq 2(\epsilon - \epsilon') \|\Delta^2u^\epsilon\|_{L^2}\|v\|_{L^2} + C_{M_0}\|v\|_{L^2}^2$$
so that, with $s>4$,
$$\sup_{0<t<T} \|v\|_{L^2}^2 \leq C(\epsilon-\epsilon')\|v\|_{L^2}+TC_{M_0}\sup_{0<t<T}\|v\|_{L^2}^2$$
Selecting $T\leq T_0$ such that $TC_{M_0} < \frac{1}{2}$ and using that $\|v\|_{L^2} \leq C$,  we have $v\to 0$ in $C([0,T]; L^2)$ as $\epsilon, \; \epsilon' \to 0$, giving that $u^\epsilon\to u$ in $C([0,T]; L^2)$ as $\epsilon \to 0$.  The family $u^\epsilon$ belongs to $L^\infty([0,T]; H^s)$, and thus, by weak-$\ast$ compactness, $u\in L^\infty([0,T]; H^s)$.  By the interpolation inequality 
$$\|v\|_{H^{s-1}} \leq \|v\|_{L^2}^{1/s}\|v\|_{H^s}^{(s-1)/s}$$
we have $u\in L^\infty([0,T]; H^s)\cap C([0,T]; H^{s-1})$.

\underline{Step 3.  Uniqueness}.  We argue as in Step 2, with $v=u-u'$, and $\epsilon=\epsilon'=0$, where $u$ and $u'$ are solutions.  We then obtain
$$\sup_{0<t<T} \|v\|_{L^2} \leq T C_{M_0} \sup_{0<t<T} \|v\|_{L^2}$$
which yields uniqueness, by taking $T\leq 1/(2C_{M_0})$.

\underline{Step 4.  $u\in C([0,T]; H^s)$ depends continuously on $u_0$}.  Here there is a standard argument, due to Bona-Smith \cite{BoSm}.  One solves with data $u_0^\delta=\varphi_\delta\ast u_0$, where $\varphi\in \mathcal{S}(\mathbb{R}^n)$, $\int \varphi = 1$, $\int x^\alpha\varphi(x)\, dx=0 \; \forall \, |\alpha| \neq 0$.  We then show that $u^\delta$, the solution corresponding to $u_0^\delta$, converges in $L^\infty([0,T_0]; H^s)$ to $u$ as $\delta\to 0$.  To see this, we show 
$$\sup_{[0,T_0]} \|u^\delta(t)\|_{H^{s+l}} \leq C\delta^{-l}, \quad l>0$$
and then use interpolation and the fact that 
$$\sup_{[0,T_0]} \|(u^\delta-u^{\delta'})(t)\|_{L^2} \leq C\|u_0^\delta-u_0^{\delta'} \|_{L^2} = o(\delta^s)$$  
This completes our outline of the energy method applied to SLCP. 

\begin{remark}
For power non-linearities, $G(u,\bar{u})=|u|^\alpha u$, more refined results can be obtained by means of mixed norm estimates (the so-called Strichartz estimates), by using the contraction principle on suitable mixed norm spaces \cite{CW}, or the $X_{s,b}$ spaces of J. Bourgain \cite{KPVnls} \cite{CDKS}.
\end{remark}
Now we briefly turn to the case of $F(u,\bar{u}, \nabla_x u, \nabla_x \bar{u} )$, and explain what the energy method gives in this case.  Suppose that for any $u\in H^s(\mathbb{R}^n)$, \medspace $s>\frac{n}{2}+1$, 
$$ \left| \sum_{|\alpha|\leq s} \int_{\mathbb{R}^n} \partial_x^\alpha F(u,\bar{u},\nabla_x u, \nabla_x \bar{u}) \partial_x^\alpha \bar{u} \, dx \right| \leq C (1+\|u\|_{H^s}^\rho)\|u\|_{H^s}^2$$
Then the above proof works (here $\rho=\rho(F)\in \mathbb{N}$).  Thus, for these examples, the energy method gives local well-posedness in $H^s$, $s>\frac{n}{2}+1$.

\begin{example} 
\begin{enumerate}
\item $n=1$, $F=\partial_x(|u|^k u)$.
\item $n\geq 1$, $F(u,\bar{u},\nabla_x \bar{u})$
\item $n\geq 1$, $F$ general, $\partial_{\partial_{x_j u}}F$, $\partial_{\partial_{x_j}\bar{u}} F$, $j=1, \ldots, n$ are real.
\end{enumerate}
\end{example}

These results are due to Tsutsumi-Fukuda \cite{TsFu}, Klainerman \cite{Kl}, Klainerman-Ponce \cite{KlPo}, Shatah \cite{Sh}.  The difficulty comes from trying to ``recover'' the ``extra-derivative'' in the non-linear term.  The remainder of the course will be devoted to developing the machinery necessary for this, in the most general situation.

We now need to recall some classical facts about $\Psi$DO.  \\

\subsection*{$\Psi$DO}
Recall the following basic facts about the Fourier transform:   
\begin{align*}
&\hat{f}(\xi) = \int e^{-2\pi i x\cdot \xi} f(x)\, dx, \quad f\in L^1(\mathbb{R}^n) \\
&\| \hat{f} \|_\infty \leq \|f\|_1 \\
&\text{Plancherel:} \quad \|\hat{f}\|_{L^2}=\|f\|_{L^2}
\end{align*}
We recall that $\widehat{\;}: \mathcal{S}(\mathbb{R}^n)\to \mathcal{S}(\mathbb{R}^n)$, and is invertible, with inverse $\check{f}(x)=\int e^{2\pi i x\cdot \xi} f(\xi) \, d\xi $.  If $L$ is a partial differential operator, of the form 
$$Lf(x)= \sum_{|\alpha|\leq m} a_\alpha(x) \partial_x^\alpha f(x)$$
and we let $a(x,\xi)=\sum_{|\alpha|\leq m} a_\alpha(x) (2\pi i \xi)^\alpha$, we have 
$$Lf(x)= \int_{\mathbb{R}^n} a(x,\xi) \hat{f}(\xi) e^{2\pi i x \xi} \, d\xi$$
Thus, for a function $a(x,\xi)$ (the symbol), we denote 
$$Tf(x)=\Psi_af(x)=\int a(x,\xi) \hat{f}(\xi) e^{2\pi i x \cdot \xi} \, d\xi$$
Note that this is, at this point, purely formal.  We start with the ``standard symbol class'' $S^m$: $a\in S^m$ if $a\in C^\infty(\mathbb{R}^n\times \mathbb{R}^n)$ and satisfies
$$|\partial_x^\beta \partial_\xi^\alpha a(x,\xi) | \leq A_{\alpha, \beta} (1+|\xi|)^{m-|\alpha|}$$
for all $\alpha, \beta$.  The $A_{\alpha, \beta}$ are called the seminorms of the symbol.  It is easy to see that if $a\in S^m$, $\Psi_a: \mathcal{S}\to \mathcal{S}$ and also $\Psi_a: \mathcal{S}'\to \mathcal{S}'$.
\begin{theorem} \label{T:1}
Suppose that $a\in S^0$.  Then $\Psi_a:L^2\to L^2$ (with norm depending only on finitely many seminorms of $a$, depending on the dimension).
\end{theorem}
\noindent See, for example, \cite{Ku} Ch.\ 2 Theorem 4.1 or \cite{St2} Ch.\ VI, \S 2

\subsection*{The symbolic calculus (Kohn-Nirenberg calculus)} 

\begin{theorem} \label{T:2}
Suppose that $a,b \in S^{m_1}, S^{m_2}$.  Then, there is a symbol $c\in S^{m_1+m_2}$ such that $\Psi_c=\Psi_a\circ\Psi_b$.  Moreover,
$$c \sim \sum_\alpha \frac{(2\pi i)^{-|\alpha|}}{\alpha!} (\partial_\xi^\alpha a) (\partial_x^\alpha b)$$
in the sense that $c-\sum_{|\alpha|<N}\frac{(2\pi i)^{-|\alpha|}}{\alpha!} (\partial_\xi^\alpha a)(\partial_x^\alpha b) \in S^{m_1+m_2-N}, \; \forall \, N>0$.
\end{theorem}
\noindent See, for example, \cite{Ku} Ch.\ 2, Theorem 1.7(1) or \cite{St2} Ch.\ VI, \S 3.
\begin{remark}  Note that $c-ab \in S^{m_1+m_2-1}$ and that each term $\partial_\xi^\alpha a \partial_x^\alpha b \in S^{m_1+m_2-|\alpha|}$
\end{remark}
\begin{remark}
Consider $\Psi_a \Psi_b - \Psi_b\Psi_a$, which is an operator with symbol in $S^{m_1+m_2-1}$.  Its symbol is $-(2\pi i)\sum_{j=1}^n \left\{ \frac{\partial a}{\partial \xi_j} \frac{\partial b }{\partial x_j} - \frac{\partial b}{\partial \xi_j} \frac{\partial a}{\partial x_j} \right\}$, modulo symbols of order $m_1+m_2-2$.
\end{remark}
\begin{theorem} \label{T:3}Let $a\in S^m$.  Then there is $a^\ast\in S^m$ such that $\Psi_a^\ast=\Psi_{a^\ast}$, and $\forall \, N>0$,
$$a^\ast(x,\xi)- \sum_{|\alpha|<N} \frac{(2\pi i)^{-|\alpha|}}{\alpha!} \partial_\xi^\alpha \partial_x^\alpha \bar{a}(x,\xi)\in S^{m-N}$$ 
($\Psi_a^\ast = \Psi_{\bar{a}} + \text{order }m-1$)
\end{theorem}
\noindent See for example, \cite{Ku} Ch.\ 2, Theorem 1.7(2) or \cite{St2} Ch. VI, \S 6.2.
\begin{remark}
If we are given a sequence $a_j\in S^{m_j}$, $j=0,1,\ldots$, such that $m_j\to -\infty$ and $m_0>m_1> \cdots$, then there is a symbol $a\in S^{m_0}$ with $a \sim a_0+a_1+ \ldots$, i.e.\ $a-\sum_{j=0}^k a_j\in S^{m_k}$.
\end{remark}
\noindent See \cite{Ku} Ch. 2, Lemma 3.2.
\begin{remark}
When we work with symbols of limited regularity (i.e.\ $|\alpha|+|\beta|\leq M(n)$), the above results still hold, but only for $N$ small, and finitely many seminorms.
\end{remark}

\section*{Problems for Lecture \arabic{chapter}}
\begin{problems}
\item \hspace*{0in}
\begin{enumerate}
\item Show that $\|u\|_{L^\infty} \leq C \|u\|_{H^s}$, $s>n/2$.
\item As a corollary, show that $H^s(\mathbb{R}^n)$, $s>n/2$ forms an algebra.
\item Also show that, if $f\in C^\infty(\mathbb{R})$, $u\in H^s(\mathbb{R}^n)$, $s>n/2$, $f(0)=0$, then $f(u)\in H^s(\mathbb{R}^n)$.
\end{enumerate}
\item Complete the proof of local well-posedness of
$$
\left\{ 
\begin{aligned}
&\partial_t u = i\Delta u + F(u,\bar{u}) \\
&u\big|_{t=0}=u_0 \in H^s, \; s>\tfrac{n}{2}
\end{aligned}
\right.
$$
sketched in the lecture.  In particular, use the `artificial viscosity' method, and the \textit{ a priori} estimate, to prove existence and uniqueness, and the Bona-Smith method to prove continuous dependence.

\item Give the proof of local well-posedness for
$$\left\{
\begin{aligned}
&\partial_t u = i\Delta u + F(u,\bar{u}, \nabla_x \bar{u}) \\
&u\big|_{t=0} = u_0 \in H^s(\mathbb{R}^n)
\end{aligned}
\right.$$
for $s>\frac{n}{2}+1$.
\item \hspace*{0in}
\begin{enumerate}
\item Prove the Cotlar-Stein Lemma:  If $\{ T_j \}$ is a finite collection of bounded operators on $L^2(\mathbb{R}^n)$, and there are $\gamma(j)$ such that $\gamma(j)\geq 0$, $A= \sum_{j=-\infty}^{+\infty} \gamma(j) < \infty$,
 $$\| T_i^\ast T_j \| \leq \gamma(i-j)^2$$
$$\|T_i T_j^\ast \| \leq \gamma(i-j)^2$$
Then $T= \sum T_j$ verifies $\|T\| \leq A$.
\item Use the Cotlar-Stein Lemma to show $\Psi_a: L^2\to L^2$ when $a\in S^0$.
\end{enumerate}
\item \label{P:1.5}Prove the ``composition formula'' in the Kohn-Nirenberg calculus.
\item Given a sequence of symbols $a_j\in S^{m_j}$, $m_j\to -\infty$, $m_0>m_1> \cdots$, find a symbol $a\in S^{m_0}$, with $a-\sum_{j=0}^k a_j \in S^{m_k}$.
\end{problems}

\lecture{Pseudo-differential operators and the G\r{a}rding inequality}

We start out by continuing our review of $\Psi$DO.  Recall that 
$$H^s(\mathbb{R}^n)= \{ f\in \mathcal{S}': \int (1+|\xi|^2)^s|\hat{f}(\xi)|^2 \, d\xi < \infty \}$$
and $\widehat{J^sf}=(1+|\xi|^2)^{s/2}\hat{f}(\xi)$, and hence $\sigma(J^s)\in S^s$, where $\sigma(\Psi_a)=a$.  

\begin{theorem}\label{T:4}
If $a\in S^m$, $\Psi_a: H^s\to H^{s-m}$.
\end{theorem}
\begin{proof}
$J^{s-m}\Psi_aJ^{-s}=\Psi_b$, $b\in S^0$, by Theorem \ref{T:2}.  Hence, by Theorem \ref{T:1} it is $L^2$ bounded, and the result follows.   
\end{proof}
\noindent Notice that the theorem shows that operators of negative order ``regularize''.\\

\noindent \textbf{An application:} We say that $a(x,\xi)\in S^m$ is elliptic if $|a(x,\xi)|\geq A|\xi|^m$, $|\xi|\geq R$.  Then, there exists $b\in S^{-m}$ and $e\in S^{-\infty}=\cap S^m$ such that $\Psi_b\Psi_a + \Psi_e = I$.
\begin{proof}
First, by adding $\phi(\xi)(1+|\xi|^2)^{m/2}$ to $a(x,\xi)$, with $\phi\in C_0^\infty$, we can assume, without loss of generality, that the estimate from below, by $(1+|\xi|^2)^{m/2}$, holds for all $\xi$.  Let us define the notation $a_1\circ a_2=a_3$ if $\Psi_{a_1}\circ \Psi_{a_2}= \Psi_{a_3}$.  We determine $b\sim b_0+b_1+ \cdots$ as follows: Let $b_0=a^{-1}$, so that $b_0\circ a= 1+e_0$, where $e_0\in S^{-1}$.  By induction, if $b_0, \ldots, b_{j-1}$ are chosen such that 
$$(b_0+b_1+ \cdots+b_{j-1})\circ a= 1 +e_{j-1}, \quad e_{j-1}\in S^{-j}$$
we choose $b_j=-e_{j-1}/a\in S^{-m-j}$, $e_j=b_j\circ a-b_ja\in S^{-j-1}$, so that 
$$(b_0+ \cdots +b_j) \circ a = 1 + e_j, \quad e_j\in S^{-j-1}$$
\end{proof}

\begin{corollary} $\|f \|_{H^{s-m}} \leq c \| \Psi_af \|_{H^s} + c \|f\|_{H^r}$, for any $r\leq s-m$.
\end{corollary}

\subsection*{The G\r{a}rding inequality (H\"{o}rmander \cite{Ho}, Lax-Nirenberg \cite{LaNi})}

\begin{theorem}[The sharp G\r{a}rding inequality] \label{T:5}
Let $a\in S^m$, with $\text{Re }a(x,\xi)\geq 0$ for $|\xi|\geq M$.  Then $\text{Re }\left< \Psi_af,f \right> \geq -C\|f\|_{H^\frac{m-1}{2}}^2$.
\end{theorem}
\begin{remark}
$C$ depends only on finitely many seminorms of $a$.
\end{remark}
\begin{remark} The above inequality has an improvement, due to C. Fefferman-Phong \cite{FePh}, where the error $\geq -C\|f\|_{H^\frac{m-2}{2}}^2$.  However, the above inequality still holds for Hilbert space valued symbols $a$ \cite{LaNi}, but the Fefferman-Phong improvement does not hold in that setting.
\end{remark}

\subsection*{Inverses of $\Psi$DO}

Suppose that $a\in S^0$, and recall that $\Psi_a:L^2(\mathbb{R}^n)\to L^2(\mathbb{R}^n)$ with norm dominated by finitely many seminorms $M(n)$ of $a$ in $S^0$.  Thus, assume that $$\sup_\xi \left< \xi  \right>^{|\alpha|} |\partial_x^\beta \partial_\xi^\alpha a(x,\xi)|\leq C^{-1} \text{ for }|\alpha|+|\beta|\leq M(n)$$  Then, for $C$ large, we have that $\|\Psi_a\|_{L^2}\leq \frac{1}{2}$, and thus $I-\Psi_a: L^2\to L^2$ is invertible, and its inverse $(I-\Psi_a)^{-1}= I+\Psi_a+ \cdots+\Psi_a^{k+1}+ \cdots$.

\begin{theorem}[Beals \cite{Be} \cite{Be79}] \label{T:6} 
Let $Q=(I-\Psi_a)^{-1}$, then $\exists \; b \in S^0$ such that $Q=\Psi_b$, and 
$$1+\sum_{k=0}^N \sigma( \Psi_a^{k+1}) \to b \text{ as }N\to \infty \text{ in }S^0$$
Moreover, seminorms of $b$ can be estimated by powers of seminorms of $a$.
\end{theorem}
We now turn to a couple of less well-known results for $\Psi$DO, which will be useful for our study of Schr\"{o}dinger equations.  Our first one is the study of $\Psi$DO on weighted $L^2$ spaces.  For $m\in \mathbb{R}$, let $\lambda_m(x)=\left<x\right>^{-m}=(1+|x|^2)^{-m/2}$.

\begin{theorem} \label{T:7}
Given $m\in \mathbb{R}$, $\exists \, N=N(n,|m|)>0$ such that, if $a\in S^0$, then 
$$\Psi_a: L^2(\mathbb{R}^n, \lambda_m(x)dx) \to L^2(\mathbb{R}^n, \lambda_m(x)dx)$$
with norm depending only on $|m|$ and seminorms of $a$ with $|\alpha|+|\beta|\leq N$.  
\end{theorem}
\begin{proof}
By duality and interpolation, it suffices to prove this for $m\in 4\mathbb{N}$.  Define now $Tf(x)= (1+|x|^2)^{-m} \Psi_a((1+|x|^2)^mf(x))$.  It suffices to show that $T$ is $L^2$ bounded.   We have (ignoring factors of $2\pi$ from now on) $((1+|x|^2)^mf(x))\sphat \, (\xi)=(I-\Delta_\xi)^m\hat{f}(\xi)$, so that 
$$Tf(x)= \int e^{ix\cdot \xi} \frac{a(x,\xi)}{(1+|x|^2)^m} (I-\Delta_\xi)^m\hat{f}(\xi) \, d\xi= \int \left[(I-\Delta_\xi)^m \frac{ e^{ix\cdot \xi}a(x,\xi)}{ (1+|x|^2)^{m}}\right] \hat{f}(\xi) \, d\xi$$
Once we obtain this formula, we can apply the Leibniz rule and the $L^2$ boundedness of $S^0$ operators to obtain the result.
\end{proof}

\begin{definition}
Let $\{ Q_\mu \}_{\mu\in \mathbb{Z}^n}$ be the unit cubes with integer coordinates as corners, which cover $\mathbb{R}^n$.  For $f\in L^2_\text{loc}(\mathbb{R}^n)$, we define
$$\tv f \tv = \sup_{\mu\in \mathbb{Z}^n} \|f\|_{L^2(Q_\mu)}$$
$$\tv f \tv_1 = \sup_{\substack{ \text{all }Q \\ l(Q)=1}} \|f\|_{L^2(Q)}$$
\end{definition}
\begin{remark}
We clearly have $\tv f \tv \leq \tv f \tv_1 \leq C_n \tv f \tv$, and hence $\forall \, x_0 \in \mathbb{R}^n$,  $\tv f(-+x_0) \tv \leq C_n \tv f \tv$.  Moreover, if $m>n$, then $\|f\|_{L^2(\lambda_m(x)dx)} \leq C_{m,n}\tv f \tv$.
\end{remark}
\begin{theorem} \label{T:8}
There exists $N=N(n)$ such that if $a\in S^0$, we have $\tv \Psi_a f \tv \leq C \tv f \tv$, with $C$ depending only on seminorms of $a$ with $|\alpha|+|\beta|\leq N$.
\end{theorem}
\begin{proof}
Let $Q_{\mu_0}$ be any unit size cube, and $x_{\mu_0}$ one of its corners.  Let $Q_0$ be the unit cube with the origin as one of its corners, so that $Q_{\mu_0}=x_{\mu_0}+Q_0$.  Then
$$\Psi_af(x)= \int e^{ix\xi} a(x,\xi) \hat{f}(\xi)  \, d\xi = \int e^{iy\xi} a(y+x_{\mu_0}, \xi ) e^{ix_{\mu_0}\xi} \hat{f}(\xi)\, d\xi$$
Let $g(x)= (e^{ix_{\mu_0} \xi}\hat{f}(\xi))\spcheck(x)=f(x+x_{\mu_0})$, $b(y,\xi)=a(y+x_{\mu_0},\xi)$.  Then $\tv g \tv\leq C_n \tv f \tv$ and $b\in S^0$, with bounds independent of $x_{\mu_0}$.  Moreover,
$$\| \Psi_b g \|_{L^2(Q_0)} = \|\Psi_a f \|_{L^2(Q_{\mu_0})}$$
Now, take $m>n$, and use the fact that, if $y\in Q_0$, then $\lambda_m(y)\geq C_n$, to see that 
$$\|\Psi_b g \|_{L^2(Q_0)} \leq C_n \|\Psi_b g \|_{L^2(\lambda_m dx)} \leq C \|g\|_{L^2(\lambda_m dx )} \leq C \tv g \tv \leq C \tv f \tv$$
which gives the proof.
\end{proof}

\begin{corollary} \label{C:1}
Let 
$$\tv f \tv' = \sum_{\mu\in \mathbb{Z}^n} \|f\|_{L^2(Q_\mu)}$$
Then $\tv \Psi_a f \tv' \leq C \tv f \tv'$, by duality.  Moreover, if $f: \mathbb{R}^n\times [0,T] \to \mathbb{C}$, and we define 
$$\tv f \tv_T = \sup_{\mu\in \mathbb{Z}^n} \|f\|_{L^2(Q_\mu\times [0,T])}$$
$$\tv f \tv'_T = \sum_{\mu\in \mathbb{Z}^n} \|f\|_{L^2(Q_\mu\times [0,T])}$$
then $\tv \Psi_af \tv_T\leq C\tv f \tv_T$ and $\tv \Psi_af \tv_T' \leq c \tv f \tv_T'$.
\end{corollary}
We now turn to our first application of $\Psi$DO to Schr\"{o}dinger equations.  Consider the problem
\begin{equation}
\left\{
\begin{aligned}
\partial_t u &= i\Delta u \\
u\big|_{t=0} &= u_0
\end{aligned}
\right.
\end{equation}
Is is easy to see that $u(x,t) = \int e^{ix\xi} e^{it|\xi|^2} \hat{u}_0(\xi) \, d\xi$ so that $\|u(-,t)\|_{H^s}=\|u_0\|_{H^s}$ for any $s$.  Let 
$$S(t)u_0= u(x,t) = \int e^{ix\xi}e^{it|\xi|^2}\hat{u}_0(\xi)\, d\xi$$
and note that $S(-t)S(t)u_0=u_0$, so that the equation is time reversible.  Hence, if for instance $u_0\in L^2$, $u(-,t)$, $t\neq 0$ cannot in general belong to $H^s$, $s>0$, i.e.\ there is no ``gain of regularity''.  This is in contrast with the heat equation
\begin{equation}
\left\{
\begin{aligned}
\partial_t u &= \Delta u \\
u\big|_{t=0} &= u_0
\end{aligned}
\right.
\end{equation}
where $\hat{u}(\xi,t)=e^{-t|\xi|^2}\hat{u}(\xi)$ and $u(-,t)\in H^s$, $\forall \, s$, $\forall\, t>0$, when $u_0\in L^2$.  Nevertheless, there is a ``local smoothing'' effect for the Schr\"{o}dinger equation, due to Constantin-Saut \cite{CS}, Sj\"olin \cite{Sjolin}, Vega \cite{Vega}, Kenig-Ponce-Vega \cite{KPV91a}, Doi \cite{Do1}, among others, that we are now going to describe. 

\begin{theorem} \label{T:9}
Let $u$ solve
\begin{equation}
\left\{
\begin{aligned}
\partial_t u &= i\Delta u +f \\
u\big|_{t=0} &= u_0
\end{aligned}
\right. \quad \text{in }\mathbb{R}^n\times [0,T]
\end{equation}
Assume that $u_0\in L^2(\mathbb{R}^n)$, $f\in L_t^1 L_x^2 = \int_0^{+\infty} \left( \int |f(x,t)|^2 \, dx \right)^{1/2} \, dt$.  Then: $\forall \, m>1$, 
$$\sup_{0<t<T} \|u(-,t)\|_{L^2}+ \|J^{1/2}u\|_{L^2(\lambda_m(x)dxdt)} \leq C_T \left\{ \|u_0\|_{L^2}+\|f\|_{L_t^1L_x^2} \right\}$$ 
If $J^{-1/2}f\in L_{t,x}^2(\lambda_{-m}(x)dxdt)$, the same estimate holds, i.e.
$$\sup_{0<t<T} \|u(-,t)\|_{L^2}+ \|J^{1/2}u\|_{L^2(\lambda_m(x)dxdt)} \leq C_T \left\{ \|u_0\|_{L^2}+\|J^{-1/2}f\|_{L^2(\lambda_m^{-1}(x)dxdt)} \right\}$$
\end{theorem}  
The main tool in the proof that we are going to give of this theorem is a construction, due to S. Doi \cite{Do1}.

\begin{lemma} \label{L:1}
Let $\lambda$ be radially decreasing, non-negative, with $\int_0^\infty \lambda(r)\, dr < \infty$,  $\lambda \in C^\infty$, $\lambda$ even, and $|\partial_x^\alpha \lambda(x)| \leq C_\alpha \lambda(x)$.  Then, there exists a real valued symbol $p\in S^0$, and a constant $c_0<1$, such that
$$2\xi \cdot \nabla_xp(x,\xi) \geq c_0 \lambda(|x|)|\xi| - \frac{1}{c_0}$$
for all $(x,\xi)\in \mathbb{R}^n$.
\end{lemma}
\begin{proof}[Proof of Theorem \ref{T:9}]
For our proof, we will choose $\lambda=\lambda_m$, $m>1$.  Let now $c_R(x,\xi) = \exp [-\theta_R(\xi)p(x,\xi)]$, which is another symbol in $S^0$, and we will consider the equation verified by $v=\Psi_{c_R}u$.  Here $\theta_R(\xi)=\theta(\xi/R)$, with $\theta\equiv 1$ for large $\xi$, $\theta\equiv 0$ for small $\xi$.  In order to do this, we need to calculate $i[\Psi_{c_R}\Delta-\Delta\Psi_{c_R}]$.  Recall that the symbol of $\Delta=-|\xi|^2$, and that, if $a$ is of order $2$, $c$ is of order $0$, then $i[\Psi_c\Psi_a-\Psi_a\Psi_c]=A_1+A_0$, where $A_0$ has symbol of order $0$, and $A_1$ has symbol of order $1$, equal to $\sum_{j=1}^n \left\{ \frac{\partial c}{\partial \xi_j} \frac{\partial a}{\partial x_j} - \frac{\partial a}{\partial \xi_j} \frac{\partial c}{\partial x_j} \right\}$.  In our case, this equals 
\begin{align*}
\sum_{j=1}^n  2\xi_j \frac{\partial}{\partial x_j} c_R(x,\xi) &= -\sum_{j=1}^n 2\xi_j \theta_R(\xi) \frac{\partial p}{\partial x_j}(x,\xi) c_R(x,\xi) \\
&\leq -c_0 \theta_R(\xi) \lambda_m(x)|\xi| c_R(x,\xi) + c_0^{-1}\theta_R(\xi)c_R(x,\xi)
\end{align*}
Let us denote by $a_1(x,\xi)$ the symbol of $A_1$, and $a_0(x,\xi)$ a generic symbol of order $0$.  Next, we claim that, for $R$ large, $\Psi_{c_R}$ is invertible, and its inverse, $\Psi_{c_R}^{-1} = \Psi_{c_R^+}+\Psi_s$, where $s$ is of order $-1$, and $c_R^+=\exp (+\theta_R(x)p(x,\xi))$.  Indeed, the calculus gives $\Psi_{c_R}\Psi_{c_R^+}=I+\Psi_{e_R}$, where $e_R$ has order $-1$, and as an operator of order $0$, its seminorms are bounded by negative powers of $R$. (This follows from the following precise version of the composition result.  Suppose $a,b\in S^0$, then $\Psi_a\Psi_b=\Psi_c$, where $c(x,\xi)=a(x,\xi)b(x,\xi)+\sum_{|\gamma|=1} \int_0^1 r_{\gamma,\theta}(x,\xi)\, d\theta$, $r_{\gamma,\theta}(x,\xi)=\iint e^{iy\cdot \eta} \partial_\xi^{(\gamma)}a(x,\xi+\theta\eta)\partial_x^{(\gamma)}b(x+y,\xi) \, dy \, d\eta$, and the seminorms of $r_{\gamma,\theta}$ are bounded by products of the semi-norms of $\partial_\xi^{(\gamma)}a$, $\partial_x^{(\gamma)}b$, uniformly in $\theta\in [0,1]$.  See \cite{Ku})  Therefore, by Theorem \ref{T:6}, $I+\Psi_{e_R}$ is invertible, and its inverse $=\Psi_q$, $q\in S^0$, for $R$ large.  Thus, $\Psi_{c_R}\Psi_{c_R^+}\Psi_q=I$.  Note also that, since $e_R$ is of order $-1$, the symbol $q=1+$ (symbol of order $-1$).  Clearly we have $\Psi_{c_R}^{-1}=\Psi_{c_R^+}\Psi_q$.   We then have
\begin{equation}
\left\{
\begin{aligned}
\partial_tv &= i\Delta v + A_1u+A_0u+F \\
v\big|_{t=0} &= v_0
\end{aligned}
\right.
\end{equation}
where $v_0=\Psi_{c_R}u_0$, $F=\Psi_{c_R}f$.  We rewrite $A_1u = A_1\Psi_{c_R}^{-1}\Psi_{c_R}u=\tilde{A}_1\Psi_{c_R}u+A_0\Psi_{c_R}u$, where $\tilde{A}_1$ has symbol $a_1(x,\xi)c_R^+(x,\xi)$, and $A_0$ is of order $0$.  We also rewrite $A_0u=A_0\Psi_{c_R}^{-1}\Psi_{c_R}u=A_0\Psi_{c_R}u$, so that our equation becomes:
\begin{equation}
\left\{
\begin{aligned}
\partial_tv &= i\Delta v + \tilde{A}_1v+A_0v +F\\
v\big|_{t=0} &= v_0
\end{aligned}
\right.
\end{equation}
We will now prove the desired estimates for $v$, which in light of the invertibility of $\Psi_{c_R}$ and Theorem \ref{T:7} gives the estimate.  We consider the equation obtained by conjugating, multiplying the first one by $\bar{v}$, the second one by $v$, integrating both equations in $x$, and adding.  Note that $R$ is now fixed forever.  We then get, as in the first lecture:
\begin{align*}
\partial_t \int |v|^2 &= i\int [ \Delta v \bar{v} - \Delta \bar{v} v ] + 2 \text{Re}\int \tilde{A}_1v \cdot \bar{v} + 2 \text{Re}\int A_0v \cdot \bar{v} + 2\text{Re }\int F\cdot \bar{v} \\
&= 2\text{Re}\int \tilde{A}_1 v \cdot \bar{v} + 2 \text{Re }\int A_0 v \cdot \bar{v} + 2\text{Re}\int F\cdot \bar{v}
\end{align*}
Recall that $A_0$ is of order $0$, and hence $\left| 2\text{Re}\int A_0v \cdot \bar{v} \right| \leq C\|v\|_{L^2}^2$.  Moreover, recall that the symbol of $\tilde{A}_1$ is $a_1(x,\xi)\exp(+\theta_Rp)$, and that 
\begin{align*}
a_1(x,\xi) &\leq -c_0\theta_R(\xi)\lambda_m(x)|\xi|c_R(x,\xi)+ c_0^{-1}\theta_R(\xi)c_R(x,\xi) \\
&= -c_0 \theta_R(\xi)\lambda_m(x)|\xi|\exp (-\theta_Rp) + c_0^{-1}\theta_R(\xi)\exp(-\theta_Rp)
\end{align*}
so that
$$\tilde{a}_1(x,\xi)\leq -c_0\theta_R(\xi)\lambda_m(x)|\xi|+c_0^{-1}\theta_R(\xi)$$
Consider now the symbol $b(x,\xi)=\frac{c_0}{2}\lambda_m(x)\left< \xi \right>-c_0^{-1}$ which belongs to $S^1$.  Notice that $-b(x,\xi)-\tilde{a}_1(x,\xi) \geq 0$ for $|\xi|\geq R$.  We can therefore apply the sharp G\r{a}rding inequality (Theorem \ref{T:5}) to conclude that $\text{Re}\left< \Psi_{-b}v,v\right> \geq \text{Re}\langle \tilde{A}_1 v, v\rangle - C\|v\|_{L^2}^2$, and hence, using the form of $b$: 
$$\text{Re}\langle \tilde{A}_1v,v\rangle \leq C \|v\|_{L^2}^2-\text{Re}\left< \Psi_{\frac{c_0}{2}\lambda_m(x)\left<\xi\right>}v,v\right>$$
 and so 
$$\partial_t\int|v|^2 \leq 2 \left| \int F\cdot \bar{v} \right| + C \|v\|_{L^2}^2 - \text{Re} \left< \Psi_{\frac{c_0}{2}\lambda_m(x)\left<\xi \right>}v,v\right>$$
  Next we note that the calculus shows that, since 
$$c_0\lambda_m(x)\left<\xi\right>=c_0^{1/2}\lambda_m(x)^{1/2}\left<\xi\right>^{1/2}c_0^{1/2}\lambda_m(x)^{1/2}\left<\xi\right>^{1/2}$$ 
and $c_0^{1/2}\lambda_m(x)^{1/2}\left<\xi\right>^{1/2}$ is  a real valued symbol in $S^{1/2}$, we have 
$$\text{Re}\left< \Psi_{\frac{c_0}{2}\lambda_m(x)\left<\xi\right>}v,v\right> = \frac{1}{2}\|\psi_{c_0^{1/2}\lambda_m^{1/2}(x)\left< \xi\right>^{1/2}}v \|_{L^2}^2+O(\|v\|_{L^2}^2)$$
  Moreover, the first term equals $\frac{c_0}{2}\|J^{1/2}v\|_{L^2(\lambda_m(x)dx)}^2$, and so 
$$\partial_t \int |v|^2 + \frac{c_0}{2} \|J^{1/2}v\|_{L^2(\lambda_m(x)dx)}^2 \leq C\|v\|_{L^2}^2 + 2 \left| \int F\cdot \bar{v} \right|$$
The first estimate now follows immediately by integration in $t$, for $T$ small enough.  For the second one, just write 
$$\int F\cdot \bar{v} = \int J^{-1/2}F \overline{J^{1/2}v} = \int \lambda_m^{-1/2}(x)J^{-1/2}F \lambda_m^{1/2}(x)J^{1/2}\bar{v}$$
 and use Cauchy-Schwarz, and integrate in $t$.  The theorem follows by iterating in $T$.
\end{proof}

It remains to prove the lemma.
\begin{proof}[Proof of Lemma \ref{L:1}]
Define $f(t) =\int_0^t \lambda(|r|)\, dr$, and let $$\Phi(x)=(f(x_1), \ldots, f(x_n))$$ so that $f$ is smooth and bounded.  Let 
$$\Phi'_{\text{sym}}= \tfrac{1}{2}\left[ \partial_{x_j}\Phi_i+\partial_{x_i}\Phi_j\right]= 
\begin{pmatrix} 
\lambda(|x_1|) & 0              & \dots  & 0 \\ 
0              & \lambda(|x_2|) & \dots  & 0 \\
\vdots         & \vdots         & \ddots & \vdots \\
0              & 0              & \dots  & \lambda(|x_n|) 
\end{pmatrix}  
\geq \lambda(|x|)I$$
since $\lambda$ is radially decreasing.  Let $p(x,\xi) = \Phi(x) \cdot \frac{\xi}{\left<\xi \right>} \in S^0$.  Then 
$$2\xi \cdot \nabla_xp(x,\xi)= \frac{\xi}{\left< \xi \right>} \cdot 2\nabla_x \left[ \Phi(x) \cdot \xi \right] = \frac{2}{\left< \xi \right>} \Phi'_{\text{sym}}(x)\xi\cdot \xi \geq 2 \lambda(|x|) \frac{|\xi|^2}{\left< \xi \right>}$$
\end{proof}
Finally, we remark that similar arguments, using Theorem \ref{T:8}, its proof and Corollary \ref{C:1}, and choosing $\lambda \equiv 1$ on $Q_0$, $\text{supp }\lambda \subset 8 Q_0$, give:
\begin{theorem} \label{T:10}
Let $u$ solve 
\begin{equation*}
\left\{ 
\begin{aligned}
\partial_t u &= i\Delta u +f\\
u\big|_{t=0} &=u_0
\end{aligned}
\right. \quad \text{in }\mathbb{R}^n\times [0,T]
\end{equation*}
Then: 
$$\sup_{0<t<T} \| u(t) \|_{L^2} + \tv J^{1/2}u \tv_T \leq C
\begin{cases}
\|u_0\|_{L^2} + \|f\|_{L_t^1L_x^2} \\
\|u_0\|_{L^2} + \tv J^{-1/2}f \tv_T'
\end{cases}
$$
\end{theorem}

\section*{Problems for Lecture \arabic{chapter}}
\begin{problems}
\item \label{P:1} A function $c(x,y,\xi)$ is called a ``compound symbol'' if it satisfies
$$|\partial_y^\gamma \partial_x^\beta \partial_\xi^\alpha c(x,y,\xi)| \leq C_{\alpha, \beta, \gamma}(1+|\xi|)^{m-|\alpha|}$$
To each such $c$, we associate the operator $\Psi_{[c]}$ given by
\begin{align*}
\Psi_{[c]} f(x) &= \int c(x,y,\xi) e^{i\xi\cdot (x-y)} f(y) \, dy d\xi \\
&= \lim_{\epsilon \to 0^+} \int c(x,y,\xi) \gamma(\epsilon y, \epsilon \xi) e^{i\xi \cdot(x-y)} f(y) dy \, d\xi
\end{align*}
$\gamma\in C_0^\infty(\mathbb{R}^n\times \mathbb{R}^n)$, $\gamma(0,0)=1$.  
\begin{enumerate}
\item Show $\Psi_{[c]}$ is well defined, $\Psi_{[c]}: \mathcal{S} \to \mathcal{S}$.
\item Show that $\exists \; a\in S^m$ such that $\Psi_a=\Psi_{[c]}$, moreover 
$$a(x,\xi) - \sum_{|\alpha|<N} \frac{i^{-|\alpha|}}
{\alpha!} \partial_\xi^\alpha \partial_y^\alpha c(x,y,\xi) \big|_{y=x} \in S^{m-N}$$
for all $N\geq 0$.
\end{enumerate}
\item Use Problem \chapterref{P:1}.\ref{P:1} to study $\Psi_a^\ast$, $a\in S^m$.
\item \underline{The G\r{a}rding inequality}:  Suppose $a\in S^m$, $\text{Re }a(x,\xi) \geq C|\xi|^m$ for $|\xi|$ large.  Then $\text{Re }\langle \Psi_a f , f \rangle \geq C_0 \|f\|_{H^{m/2}}^2 - C_1 \|f\|_{H^s}^2$ for any $s\in \mathbb{R}$.
\item Let $u$ solve
$$\left\{
\begin{aligned}
& \partial_t u = i\Delta u +f \\
& u\big|_{t=0} = u_0
\end{aligned}
\right.
$$
\begin{enumerate}
\item Assume $f=0$, $n=1$.  Show that
$$\|D_x^{1/2} u \|_{L_x^\infty L_t^2} \leq C\|u_0\|_{L^2}$$
(Use the Fourier transform).
\item Use the Fourier transform to show, for $n>1$, that
$$\sup_\alpha \|D_x^{1/2}u \|_{L^2(Q_\alpha\times (-\infty, +\infty))} \leq C \|u_0\|_{L^2}$$
\item Again use the Fourier transform to show, when $u_0=0$, and $n=1$, that
$$\|D_x u \|_{L_x^\infty L_t^2} \leq C\|f\|_{L_x^1L_t^2}$$
and when $n>1$
$$\sup_\alpha \|D_x u \|_{L^2(Q_\alpha\times (-\infty, +\infty))} \leq C\sum_\alpha \|f\|_{L^2(Q_\alpha\times (-\infty, +\infty))}$$
\end{enumerate}
\item Give the proof of Theorem \ref{T:10}, using $\Psi$DO.
\end{problems}

\lecture{The semilinear Schr\"{o}dinger equation}

We are now going to sketch the proof of the fact that those estimates already give non-trivial results for the semi-linear, constant coefficient Cauchy problem.  We need one more estimate:
\begin{equation*}
\left\{ 
\begin{aligned}
\partial_t u &= i \Delta u \\
u\big|_{t=0} &= u_0
\end{aligned}
\right.
\end{equation*}
$u(t)=S(t)u_0$.  Then
\begin{lemma} \label{L:MX}
For any $s>\frac{n}{2}+2$, we have 
$$\|S(t)u_0\|_{l_\mu^2(L^\infty(Q_\mu\times [0,T]))} \leq C_T\|u_0\|_{H^s}$$
For any integer $s> 2n+2$, we have,
$$\| S(t)u_0 \|_{l_\mu^1(L^\infty(Q_\mu\times [0,T]))} \leq C_T \sum_{|\alpha|\leq s} \| (1+|x|^2)^{n+1}\partial_x^\alpha u_0\|_{L^2}$$
\end{lemma}
\begin{proof}
For the first inequality, we just need the estimate $\sup_{0<t<T} \|u(-,t)\|_{H^s} \leq C \|u_0\|_{H^s}$ and the Sobolev embedding theorem:  For $s_0>n/2$, we have (assuming $s_0$ is an integer):
\begin{align*}
\|u\|_{L^\infty(Q_\mu\times [0,T])}^2 & \leq C\sup_{[0,T]} \|u(-,t)\|_{H^{s_0}(Q_\mu)}^2 \\
& \leq C_T \int_0^T \|u(-,t)\|^2_{H^{s_0}(Q_\mu)} + \partial_t \|u(-,t)\|_{H^{s_0}(Q_\mu)}^2 \, dt \\
& \leq C_T \int_0^T \biggl\{ 
\begin{aligned}[t]
\|u(-,t)\|_{H^{s_0}(Q_\mu)}^2 + \partial_t \int_{Q_\mu} |u(x,t)|^2 \, dx \\
+ \partial_t \int_{Q_\mu} |\partial_x^{s_0}u(x,t)|^2 \, dx \biggl\} \, dt 
\end{aligned}\\
& \leq C_T \int_0^T \mspace{-6mu} \int_{Q_\mu}|u|^2+|\partial_tu|^2+|\partial_t\partial_x^{s_0}u|^2 + |\partial_x^{s_0}u|^2 \, dx\,dt
\end{align*}
But since $\partial_t u=i\Delta u$, the result follows.  For the second part, we use the inequality $\|w\|_{L^1} \leq C \|(1+|x|^2)^{n+1/4}w\|_{L^2}$, and the identity $x_jS(t)u_0=S(t)(x_ju_0)-2itS(t)(\partial_{x_j}u_0)$ together with the above argument.  To check the identity, apply $(\partial_t-i\Delta)$ to both sides, and use the fact that they are equal at $t=0$.  
\end{proof}

We recall the estimates in Theorem \ref{T:10} (at the $H^s$ level)
\begin{equation*}
\left\{
\begin{aligned}
\partial_t u &= i\Delta u + f \\
u\big|_{t=0} &= u_0
\end{aligned}
\right.
\end{equation*}
Then:
\begin{equation}\label{E:100}
\sup_{0<t<T} \| u(t) \|_{H^s} + \tv J^{s+1/2}u \tv_T \leq C
\begin{cases}
\|u_0\|_{H^s} + \|f\|_{L_t^1H^s_x} \\
\|u_0\|_{H^s} + \tv J^{s-1/2}f \tv_T'
\end{cases}
\end{equation}
Let us consider the IVP
$$ \left\{
\begin{aligned}
\partial_t u &= i \Delta u + P(u, \bar{u}, \nabla_x u, \nabla_x \bar{u} ) \\
u\big|_{t=0} &= u_0
\end{aligned}
\right.$$
Let us assume that $P(z_1,z_2, \ldots, z_{2n+2})=\sum_{d\leq |\alpha|\leq \rho} a_\alpha z^\alpha$ with $\alpha \in \mathbb{N}^{2n+2}$.  We assume $a_{\alpha_0}\neq 0$, $|\alpha_0|=d$, and consider $d\geq 2$.
\begin{theorem} \label{T:S1}
Let $d\geq 3$.  Then $\exists \; \delta=\delta(P)>0$ such that $\forall \, u_0 \in H^s(\mathbb{R}^n)$, $s\geq s_0=n+2+\frac{1}{2}$, $\|u_0\|_{H^{s_0}}\leq \delta$, we have local well-posedness in $H^s$. 
\end{theorem}
Let $$\|f\|_{H^{s,j}(\mathbb{R}^n)}=\sum_{|\gamma|\leq s} \left( \int |\partial_x^\gamma f|^2(1+|x|^2)^{j/2} \, dx \right)^{1/2}$$
\begin{theorem}\label{T:S2}
Let $d=2$.  Then $\exists \; \delta=\delta(P)$ such that $\forall \; u_0 \in H^s(\mathbb{R}^n)\cap H^{2n+3,2n+2}(\mathbb{R}^n)=G_s$, $s\geq s_0=3n+4+\frac{1}{2}$, $\|u_0\|_{H^{s_0}}+\|u_0\|_{H^{2n+3,2n+2}} \leq \delta$, we have local well-posedness in $G_s$.
\end{theorem}
\begin{proof}[Sketch of proof of Theorem \ref{T:S1}] 
For simplicity, let $$P(u,\bar{u},\nabla_x u, \nabla_x \bar{u})= \frac{\partial u}{\partial x_l} \frac{\partial u}{\partial x_j} \frac{\partial u}{\partial x_k}$$  Let us take $s_0=n+4+\frac{1}{2}$.  For fixed $u_0\in H^{s_0}(\mathbb{R}^n)$, $\|u_0\|_{H^{s_0}}\leq \delta$, $\delta$ to be determined, consider, for $v$ fixed, a solution to the inhomogeneous linear problem
\begin{equation*}
\left\{
\begin{aligned}
\partial_t u &= i \Delta u + \frac{\partial v}{\partial x_l}\cdot \frac{\partial v}{\partial x_j}\cdot \frac{\partial v}{\partial x_k} \\
u\big|_{t=0} &= u_0
\end{aligned}
\right.
\end{equation*}
for $v\in Z^a_T= \{ v\colon \mathbb{R}^n\times [0,T]\to \mathbb{C} \; \mid \; \lambda_j^T(v) \leq a$, $j=1,2,3$ \} with $T\leq 1$, where 
\begin{align*}
\lambda_1^T(v) &= \sup_{0<t<T} \|v(t) \|_{H^{s_0}} \\
\lambda_2^T(v) &= \sum_{|\beta|=s_0+\frac{1}{2}} \tv \partial_x^\beta v \tv_T \\
\lambda_3^T(v) &= \| \nabla v \|_{l_\mu^2(L^\infty(Q_\mu\times [0,T]))}
\end{align*}
We show that, for $\delta$ small enough, $T=T(\delta)$, $a=a(\delta)$, if $v\in Z_T^a$, so does $u$, and $v\mapsto \Phi(v)=u$ is a contraction in $Z_T^a$, so we find $u$ such that $u=\Phi(u)$ and our non-linear problem is solved.  By Duhamel's formula (variation of the constants),
$$u=S(t)u_0 + \int_0^t S(t-t')[ \partial_{x_l}v \cdot \partial_{x_j}v \cdot \partial_{x_k}v ] \, dt'$$
We first estimate $\lambda_2^T(u)$.  Note that if $|\beta|=s_0-\frac{1}{2}$, then 
\begin{align*}
\partial_x^\beta(\partial_{x_l}v\cdot \partial_{x_j}v \cdot \partial_{x_k}v) &= \partial_x^\beta \partial_{x_l}v \cdot \partial_{x_j}v \cdot \partial_{x_k}v +  \partial_{x_l}v \cdot \partial_x^\beta \partial_{x_j}v \cdot \partial_{x_k}v \\
& \qquad  + \partial_{x_l}v \cdot \partial_{x_j}v \cdot \partial_x^\beta \partial_{x_k}v + R((\partial_x^\gamma v)_{1\leq|\gamma|\leq s_0-1/2})
\end{align*}
We plug this into the Duhamel formula, and we use the first estimate in \eqref{E:100} for $R$, the second one for the main terms, to get:
\begin{align*}
\lambda_2^T(v) &\leq C\|u_0\|_{H^{s_0}}+C\sum_{|\beta|=s_0-1/2} \tv \partial_x^\beta \partial_{x_l} v \cdot \partial_{x_j}v\cdot \partial_{x_k}v \tv_T'+ \cdots  + C \| J^{1/2}R \|_{L_T^1L_x^2}  \\
& \leq 
\begin{aligned}[t]
&C\|u_0\|_{H^{s_0}}+ C \sum_{|\beta|=s_0+1/2} \sup_\mu \|\partial_x^\beta v \|_{L^2( Q_\mu\times [0,T])}\sum_\mu \| |\nabla v |^2 \|_{L^\infty(Q_\mu\times [0,T])}\\
&+CT\sup_{[0,T]}\|v\|_{H^{s_0}}^3 
\end{aligned}\\
& \leq C \|u_0\|_{H^s_0}+CT\sup_{[0,T]} \|v(t)\|_{H^{s_0}}^3 + C \lambda_2^T(v)(\lambda_3^T)^2(v)
\end{align*}
\textit{Note:} To handle $\|J^{1/2}R\|_{L^1_TL_x^2}$, we note that $s_0=n+4+\frac{1}{2}$, $|\beta|=s_0-\frac{1}{2}=n+4$, and 
$$R=\sum_{\substack{ \alpha \leq \gamma\leq \beta\\ |\alpha|\geq 1 }} C_{\alpha\gamma\beta} \partial_x^{\beta-\gamma} (\partial_xv)\partial_x^{\gamma-\alpha}(\partial_xv)\partial_x^\alpha(\partial_x v)$$
where $\partial_x$ refers to $\partial_{x_l}$, $\partial_{x_j}$, or $\partial_{x_k}$.  Note that, if $|\beta-\gamma|\geq \frac{|\beta|}{2}$, since $|\beta-\gamma|+|\gamma-\alpha|+|\alpha|=|\beta|$, we have $|\gamma|\leq \frac{|\beta|}{2}$, $|\alpha| \leq \frac{|\beta|}{2}$, and hence $J^{1/2}\partial_x^{\gamma-\alpha}(\partial_x v)  \in L^\infty$, since $|\beta|-\frac{|\beta|}{2}-\frac{3}{2}=\frac{n}{2}+\frac{1}{2}>\frac{n}{2}$.  A similar remark applies to $\partial_x^\alpha(\partial_x v)$.  Thus, for these terms, we use 
\begin{align*}
\hspace{0.5in}&\hspace{-0.5in}\|J^{1/2}(\partial_x^{\beta-\gamma}(\partial_xv))\cdot \partial_x^{\gamma-\alpha}(\partial_x v) \cdot \partial_x^\alpha(\partial_x v) \|_{L^2} \\
& \leq 
\begin{aligned}[t] 
&\|J^{1/2}(\partial_x^{\beta-\gamma}(\partial_x v)) \|_{L^2} \| \partial_x^{\gamma-\alpha}(\partial_x v) \partial_x^\alpha(\partial_x v) \|_{L^\infty} \\
&+ \|\partial_x^{\beta-\gamma}(\partial_x v)\|_{L^2}\|J^{1/2}(\partial_x^{\gamma-\alpha}(\partial_xv)\cdot \partial_x^\alpha(\partial_x v)\|_{L^\infty}
\end{aligned}
\end{align*}
If on the other hand, $|\beta-\gamma|\leq \frac{|\beta|}{2}$, since we must have either $|\gamma-\alpha|\leq \frac{|\beta|}{2}$ or $|\alpha|\leq \frac{|\beta|}{2}$, we proceed in the same way.

$\lambda_1^T$ is estimated similarly, while for $\lambda_3^T$, we use the maximal function estimate, and Minkowskii's integral inequality, to obtain
$$\lambda_3^T(u) \leq C \|u_0\|_{H^{\frac{n}{2}+3+\frac{1}{2}}} + CT^{1/2}\lambda_1^T(v)^3$$
(We have also used that $H^s$, $s>n/2$, is an algebra, and that $\|J^{1/2}(f\cdot g)\|_{L^2} \leq \|f\|_{L^\infty}\|J^{1/2}g\|_{L^2}+\|g\|_{L^\infty}\|J^{1/2}f\|_{L^2}$).  Gathering terms, we get, with $\lambda^T(u)=\max_{j=1,2,3}\lambda_j^T(u)$ that 
$$\lambda^T(u) \leq C\delta + C\lambda^T(v)^3$$ and $$\lambda^T(u_1-u_2)\leq \lambda^T(v_1-v_2)[\lambda^T(v_1)^2+\lambda^T(v_2)^2]$$
  From this the desired result follows.  To see why we need the small norm, notice that in the estimate for $\lambda_2^T(u_1-u_2)$, we will get a term of the form $(\lambda_3^T(v_1)^2+\lambda_3^T(v_2)^2)\lambda_2^T(v_1-v_2)$, with no factor of $T$ in front.  To have a contraction, we need $\lambda_3^T(v_i)=\|\nabla v_i \|_{l_\mu^2(L^\infty(Q_\mu\times [0,T]))}$ small, which on $v=S(t)u_0$ forces small data.
\end{proof}

To understand the result for $d=2$ (Theorem \ref{T:S2}), let us take $$P(u,\bar{u},\nabla_xu,\nabla_x \bar{u})=|\nabla u|^2$$ for instance, and let's go to the estimate corresponding to $\lambda_2^T$ in the previous proof.  We get now ($|\beta|=s_0-\frac{1}{2}$) $\tv \partial_x^\beta \partial_xv\cdot \partial_x v \tv_T'$,
which is controlled by $\tv \partial_x^\beta \partial_x v \tv_T \cdot l_\mu^1( \|\nabla v \|_{L^\infty(Q_\mu \times [0,T])} )$ so that we need to control the $L^1$ norm of the maximal function.  This leads us to the weights.   This might seem at first an artifice of the proof.  Here it is useful to recall that when one obtains the solution by the contraction principle, the mapping $u_0\mapsto u$ is not only continuous, but in fact real analytic.  We now have the following result: consider when $n=1$,
\begin{equation} \label{E:9}
\left\{ 
\begin{aligned}
\partial_t u &= i\partial_x^2u+u\partial_x u  \\
u\big|_{t=0} &= u_0
\end{aligned}
\right.
\end{equation}
\begin{theorem}[Molinet-Saut-Tzvetkov \cite{MoSaTz}] \label{T:MST}
Fix $s\in \mathbb{R}$.  Then there does not exist a $T>0$ such that  \eqref{E:9} admits a unique local solution in $[0,T]$ such that the mapping $u\big|_{t=0} \mapsto u(t)$ is $C^2$ differentiable at $0$ from $H^s(\mathbb{R})$ to $H^s(\mathbb{R})$.
\end{theorem}
\begin{proof}
Consider 
\begin{equation}
\left\{
\begin{aligned}
\partial_t u &= i \partial_x^2 u + u \partial_x u \\
u\big|_{t=0} &= \gamma u_0, \quad \gamma <<1, \quad u_0 \in H^s(\mathbb{R})
\end{aligned}
\right.
\end{equation}
If $u(\gamma,x,t)$ is a local solution and the flow map is $C^2$ at the origin, 
$$\frac{\partial^2 u}{\partial \gamma^2}(0,t,x) = -2 \int_0^t S(t-t')[S(t')u_0S(t')\partial_x u_0 ] , dt'$$
The assumption of $C^2$ regularity gives
$$\left\| \int_0^t S(t-t')[S(t')u_0S(t')\partial_x u_0] \, dt' \right\|_{H^s} \leq C \|u_0\|_{H^s}^2$$
One then shows directly that this fails for any $s$.  (One chooses $\hat{u}_0(\xi)=\alpha^{-1/2}\chi_{I_1}(\xi)+\alpha^{-1/2}N^{-s}\chi_{I_2}(\xi)$ where $0<\alpha<<1$, $N>>1$, $I_1=[\frac{\alpha}{2},\alpha]$, $I_2=[N,N+\alpha]$, and uses 
\begin{equation*}
\begin{split} \hspace*{0.3in} & \hspace*{-0.3in} \int_0^t S(t-t')[S(t')u_0S(t')\partial_xu_0] \,dt' \\
&= c \int_{\mathbb{R}^2} e^{i(x\xi+tp(\xi))} \xi \hat{u}_0(\xi_1)\hat{u}_0(\xi-\xi_1) \frac{ e^{it[p(\xi_1)+p(\xi-\xi_1)-p(\xi)]}-1}{p(\xi_1)+p(\xi-\xi_1)-p(\xi)} \, d\xi \, d\xi_1
\end{split}
\end{equation*}
where $p(\xi)=\xi^2$). 
\end{proof}

How does one remove the smallness?  Let us go back to the equation that we studied, and rewrite it in the following way:
\begin{equation*}
\left\{
\begin{aligned}
\frac{\partial u}{\partial t} &= i\Delta u + \partial_{x_l}u_0 \partial_{x_k}u_0 \partial_{x_k}u - \underbrace{[\partial_{x_l}u_0 \partial_{x_j}u_0 - \partial_{x_l} u \partial_{x_j} u]}_{\text{ small near $t=0$}}\partial_{x_k} u\\
u\big|_{t=0} &= u_0
\end{aligned}
\right.
\end{equation*}
Consider the new linear equation
$$
\left\{ 
\begin{aligned}
\partial_t w &= i\Delta w + [ \partial_{x_l}u_0 \partial_{x_j} u_0 ] \partial_{x_k} w + f \\
w\big|_{t=0} &= w_0
\end{aligned}
\right.
$$
Suppose that we could prove, for suitably good $u_0$, the same estimates as before.  Then we would be done.  This actually works (Kenig, Ponce, Vega \cite{KPV98}).  We are thus lead to studying linear problems of the form
\begin{equation} \label{E:10}
\left\{
\begin{aligned}
\partial_t u &= i \Delta u + \vec{b}_1(x)\cdot \nabla u + \vec{b}_2(x)\cdot \nabla \bar{u} + c_1(x)u + c_2(x)\bar{u} + f \\
u\big|_{t=0} &= u_0
\end{aligned}
\right.
\end{equation}
and to prove under suitable assumptions on $\vec{b}_1$, $\vec{b}_2$, $c_1$, $c_2$, the estimate 
$$\sup_{|t|\leq T} \|u(t)\|_{H^s} + \tv J^{s+1/2}u\tv_T \leq C_T 
\begin{cases}
\|u_0\|_{H^s} + \tv J^{s-1/2}f \tv_T' \\
\|u_0\|_{H^s} + \| J^s f \|_{L^1_TL_x^2}
\end{cases}$$
I will next make some remarks about \eqref{E:10} when $\vec{b}_2\equiv 0$, $c_1\equiv 0$, $c_2 \equiv 0$.  The first point is that there exists a very revealing necessary condition for the estimate
\begin{equation}\label{E:11}
\sup_{|t|\leq T} \|u(t) \|_{L^2} \leq C_T \{ \|u_0\|_{L^2} + \|f\|_{L^1_TL_x^2} \}
\end{equation}
for solutions to 
$$
\left\{
\begin{aligned}
\partial_t u &= i \Delta u + \vec{b}_1(x)\cdot \nabla u + f \\
u\big|_{t=0} &= u_0 
\end{aligned}
\right.$$
This was discovered by Takeuchi \cite{Ta}.  It is the following: We must have, for all $\omega \in S^{n-1}$, $(x,t)\in \mathbb{R}^n\times \mathbb{R}$
\begin{equation}\label{E:12}
 \left| \text{Im} \int_0^t \vec{b}_1(x+s\omega) \cdot \omega \, ds \right| \leq C
\end{equation}
I will now present a proof of this, due to Mizohata \cite{Mi}.  Let us first explain the condition when $n=1$.  In this case, the equation becomes
$$\left\{
\begin{aligned}
\partial_t u &= i \partial_x^2 u + b_1(x)\partial_x u + f\\
u\big|_{t=0} &= u_0
\end{aligned}
\right.$$
Let now $v(x,t) = p(x)u$, where $p$ will be chosen.  Since $p(x) \partial_x^2 u = \partial_x ^2 (p(x) u) - 2p'(x)\partial_x u - p''(x) u$, if we choose $p(x)$ in such a way that $p(x)b_1(x)-2ip'(x)=0$, (namely, $p(x)=\exp \left( -\frac{i}{2}\int_0^x b_1(y) \, dy \right)$), then our equation for $v$ becomes
$$\partial_tv = i \partial_x^2 v + c(x) v + p(x) f$$
where $c(x)=-ip''(x)/p(x)$.  Thus, if $c_1\leq |p(x)|\leq C_1$, this equation is $L^2$ wellposed, and the condition is equivalent with $\left| \text{Im}\int_0^t b_1(y) \, dy\right| \leq M$.  Note for example, that if $b_1$ is constant, $\text{Im }b_1 = 0$ is necessary (and sufficient) for $L^2$ well-posedness.  

To now show the necessity of \eqref{E:12}, we write $u(x,t) = e^{i\Psi(x,t;\xi)}v(x,t;\xi)$, and apply $L=\partial_t-i\Delta - \vec{b}_1\cdot \nabla$ to obtain:
\begin{align*}
 e^{-i\Psi}L(e^{i\Psi}v) &= i \{ \partial_t \Psi + |\nabla \Psi|^2 \} v +  \\
& \quad + \{ \partial_t v + \Delta \Psi v + 2 \nabla \Psi \cdot \nabla v - i \vec{b}_1\nabla \Psi v \} - i \{ \Delta v + i \vec{b}_1\cdot \nabla v \}
\end{align*}
We first choose $\Psi$ as a solution of $\partial_t\Psi + |\nabla \Psi|^2=0$, so we take $\Psi(x,t;\xi)=x\cdot \xi - t|\xi|^2$.   Observe that $\Delta \Psi=0$, so the equation becomes
$$e^{-i\Psi} L(e^{i\Psi} v) = \{ \partial_t v + 2 \nabla \Psi \cdot \nabla v - i \vec{b}_1 \cdot \nabla \Psi v \} - i \{ \Delta v + i \vec{b}_1 \cdot \nabla v \}$$
We now choose $v$ as a solution of (transport equation)
$$
\left\{
\begin{aligned}
&\partial_t v + 2 \nabla \Psi \cdot \nabla v - i \vec{b}_1\cdot \nabla \Psi v = 0 \\
&v \big|_{t=0} = v_0(x) 
\end{aligned}
\right.
$$
or 
$$\partial_t v + 2 \sum \xi_j \frac{\partial v}{\partial x_j} - i \vec{b}_1(x) \cdot \xi v = 0$$
so that
$$v(x,t; \xi) =  \exp \left[ i \int_0^t \vec{b}_1(x-2\xi s) \cdot \xi \, ds \right] v_0(x-2\xi t )$$
We will show that, if \eqref{E:12} is violated, then \eqref{E:11} fails to hold.  First, note that
$$\int_0^{-t} \vec{b}_1(x+\omega s) \cdot \omega \, ds = \int_0^t \vec{b}_1(x-\omega s)\cdot (-\omega) \, ds$$
$$\int_0^t \vec{b}_1(x+\omega s) \cdot \omega \, ds = -\int_0^t \vec{b}_1(x'-\omega s)\cdot (-\omega) \, ds$$
where $x'=x+t\omega$, and for $\rho>0$
$$ \int_0^t \vec{b}_1(x-\omega s) \cdot \omega \, ds = \int_0^{t/\rho} \vec{b}_1(x-\omega \rho s) \rho \omega \, ds$$
Because of these properties, if \eqref{E:12} fails, we can find $t_0>0$, $\omega^0\in S^{n-1}$, and $x_1\in \mathbb{R}^n$ such that 
$$\text{Re }i\int_0^{t_0} \vec{b}_1(x_1-2\omega^0 s) \cdot \omega^0 \, ds \geq \log 2C_T$$
where $C_T$ is the constant in \eqref{E:11}.  Put now $x_0=x_1-2\omega^0t_0$, $t_0/\rho = t_\rho$.  We can use the third equation to get
$$v(x_1, t_\rho; \rho\omega^0) = \exp \left[ i \int_0^{t_0} \vec{b}_1(x_1-2\omega^0 s) \cdot \omega^0 \, ds \right] v_0(x_0)$$
$$v(x, t_\rho; \rho\omega^0) = \exp \left[ i \int_0^{t_0} \vec{b}_1(x-2\omega^0 s) \cdot \omega^0 \, ds \right] v_0(x-2\omega^0 t_0)$$
We will now choose $v_0$ with small support, near $x=x_0$, satisfying $\int |v_0|^2 \, dx = 1$.  Hence $\|v(-,t_\rho; \rho \omega^0) \|_{L^2} \geq \frac{3}{2}C_T$, since $v_0(x-2\omega^0 t_0)$ has small support around $x=x_1$.  Moreover, $[\Delta v + i \vec{b}_1\cdot \nabla v ]$, $v=v(x,t;\rho \omega^0)$ is uniformly bounded for $\rho\to \infty$, $t\leq t_\rho$ which implies that
$$\int_0^{t_\rho} \left\| i [ \Delta v + i \vec{b}_1\cdot \nabla v] (-,s,\rho\omega^0) \right\|_{L^2} \, ds \to 0 \quad \text{as} \quad \rho \to \infty$$
Moreover, $e^{-i\Psi}L(e^{i\Psi}v)=-i\{ \Delta v + i\vec{b}_1\cdot \nabla v \}$, and $e^{i\Psi} v\big|_{t=0}=e^{i\Psi(x,0,\xi)}v_0$, so that \eqref{E:11} would give $\frac{3}{2}C_T\leq C_T\{ 1+o(1) \}$ as $\rho\to \infty$, a contradiction.  

We now turn to the positive results on \eqref{E:10}.  Here, we follow Kenig, Ponce, Vega \cite{KPV98} \cite{KPV98b} 
\begin{theorem} \label{T:S3}
Consider 
\begin{equation*} 
\left\{
\begin{aligned}
\partial_t u &= i \Delta u + \vec{b}_1(x)\cdot \nabla u + \vec{b}_2(x)\cdot \nabla \bar{u} + c_1(x)u + c_2(x)\bar{u} + f \\
u\big|_{t=0} &= u_0
\end{aligned}
\right.
\end{equation*}
Assume that $\vec{b}_i$, $c_i\in C^N(\mathbb{R}^n)$, $N=N(n)$.
\begin{enumerate}
\item \label{I:1} Assume that $|\text{Im }\vec{b}_1(x) |\leq C\lambda_m(x)= \frac{C}{(1+|x|^2)^{m/2}}$, for some $m>1$.  Then $\exists \, ! \; u$ such that $\forall \; T>0$ we have
\begin{align*}
\hspace{0.5in}&\hspace{-0.5in} \sup_{0<t<T} \|u(t)\|_{L^2} + \|J^{1/2}u\|_{L^2(\mathbb{R}^n\times [0,T]; \lambda_m dx dt)} \\
& \leq C_T
\begin{cases}
\|u_0\|_{L^2} + \|f\|_{L_T^1L_x^2} \\
\|u_0\|_{L^2} + \|J^{-1/2}f \|_{L^2(\mathbb{R}^n\times [0,T], \lambda_m^{-1}(x) dxdt)}
\end{cases}
\end{align*}
\item If $\text{Im }\vec{b}_1(x)= \sum_{\mu \in \mathbb{Z}^n} \alpha_\mu \varphi_\mu(x)$, with $\text{supp }\varphi_\mu \subset Q_\mu^\ast$, $\|\varphi_\mu\|_{C^N} \leq 1$, $\sum_\mu |\alpha_\mu| \leq A$, then we have:
$$\sup_{0<t<T} \|u(t)\|_{L^2} + \tv J^{1/2}u\tv_T \leq C_T
\begin{cases}
\|u_0\|_{L^2} + \|f\|_{L_T^1L_x^2} \\
\|u_0\|_{L^2} + \tv J^{-1/2}f \tv_T'
\end{cases}$$
\end{enumerate}
\end{theorem}
This theorem allows us to extend Theorem \ref{T:S1}, Theorem \ref{T:S2} to data of arbitrary size.
\begin{remark}
The condition \eqref{E:12} is an $L^1$ condition and another way to view the need for the weights is to mediate between \eqref{E:12} and $L^2$ conditions on $u$.  This is only needed in ``bilinear settings''.  
\end{remark}
In the next lecture, we will discuss generalizations of Theorem \ref{T:S3} (\ref{I:1}).

\section*{Problems for Lecture \arabic{chapter}}
\begin{problems}
\item Use oscillatory integrals and $T^\ast T$ arguments to show:  Let $u$ solve, for $n=1$
$$\left\{
\begin{aligned}
& \partial_t u = i\Delta u \\
& u \big|_{t=0} = u_0
\end{aligned}
\right.
$$
Then
\begin{enumerate}
\item $\| \sup_t | u(x,t) | \|_{L^4(\mathbb{R}_x)} \leq C \|D_x^{1/4} u_0 \|_{L^2(\mathbb{R})}$
\item $\| \sup_{|t|\leq 1} |u(x,t)| \|_{L^2(\mathbb{R}_x)} \leq C_s \|u_0\|_{H^s(\mathbb{R})}$, $s>\frac{1}{2}$.
\item $\| u \|_{L_t^4L_x^\infty} \leq C \|u_0\|_{L^2(\mathbb{R})}$
\end{enumerate}
\item Prove Theorem \ref{T:S1} for non-polynomial non-linearities.  Do the same for Theorem \ref{T:S2}.
\item Carry out the proof of the Molinet-Saut-Tzvetkov example, outlined in the lecture.
\item \label{P:4} Prove Theorem \ref{T:S3} using the sharp G\r{a}rding inequality, when $\vec{b}_2 \equiv 0$.  Use the pseudo-differential proofs of the local smoothing effect as a model.
\end{problems}

\lecture[Variable coefficient linear Schr\"{o}dinger, I]{The linear Schr\"{o}dinger equation with variable coefficients, Part I}

In order to study the Cauchy problem in the quasi-linear setting, we need to understand linear problems where the top order coefficients are also variable.  In order to begin our study of such problems, we need to introduce some new terminology.  Let $a(x,\xi)$ be a second order, real valued symbol in $S^2$.  The Hamiltonian vector field, associated to $a$, on $\mathbb{R}^n\times \mathbb{R}^n$, $H_a$ is given by
$$H_a(\phi)= \sum_{j=1}^n \left[ \partial_{\xi_j}a(x,\xi)\partial_{x_j}\phi - \partial_{x_j}a(x,\xi)\partial_{\xi_j}\phi \right]$$
and the bi-characteristic flow, denoted by $(X(s;x_0,\xi_0), \Xi(s;x_0,\xi_0))$ is its flow, i.e.\ the solution of the Hamilton-Jacobi equations
\begin{equation*}
\left\{
\begin{aligned}
\frac{d}{ds}X_j(s;x_0,\xi_0) &= \partial_{\xi_j}a(X,\Xi) \\
\frac{d}{ds}\Xi_j(s;x_0,\xi_0) &= -\partial_{x_j}a(X,\Xi)
\end{aligned}
\right.
\end{equation*}
for $j=1,\ldots, n$, and with data $(X(0;x_0,\xi_0), \Xi(0;x_0,\xi_0))=(x_0,\xi_0)$.  ODE theory implies that the bi-characteristic flow always exists and is unique, in a maximal interval, $s\in (-\delta,\delta)$, with $\delta=\delta(x_0,\xi_0)$, depending continuously on $(x_0,\xi_0)$.  Note that when $a(\xi)=-|\xi|^2$ (the case of the Laplacian), we have that 
$$(X(s;x_0,\xi_0),\Xi(s;x_0,\xi_0))=(x_0+2s\xi_0,\xi_0)$$
and it exists for all $s$.

Suppose now that $a(x,\xi)=-\sum a_{jk}(x)\xi_j\xi_k$, where the matrix $A(x)=(a_{jk}(x))$ is elliptic and real valued, i.e.\
$$\lambda^{-1}|\xi|^2 \leq a_{jk}(x)\xi_j\xi_k \leq \lambda |\xi|^2 \quad \forall \; \xi\in \mathbb{R}^n$$
Then, since 
$$H_a(\phi)(x_0,\xi_0) = \frac{d}{ds} \phi(X(s; x_0, \xi_0), \Xi(s; x_0,\xi_0))\Big|_{s=0}$$
we have, with $a(x,\xi)=-\sum a_{jk}(x) \xi_j \xi_k$, that $H_a(a)=0$, and hence we see that 
$$\lambda^{-2}|\xi_0|^2 \leq |\Xi(s; x_0, \xi_0)|^2 \leq \lambda^2|\xi_0|^2$$
This immediately gives that $\delta=+\infty$, i.e.\ the bi-characteristic flow exists for all times.  Since $a(x,\xi)$ is homogeneous of degree $2$ in $\xi$, uniqueness for ODE gives:
\begin{equation*}
\left\{
\begin{aligned}
X(s; x_0, r\xi_0) &= X(sr; x_0, \xi_0) \\
\Xi(s; x_0, r\xi_0) &= r \Xi(sr; x_0, \xi_0)
\end{aligned}
\right.
\end{equation*}
The role of the Hamiltonian vector field in our context can be understood from the following consequence of the calculus of $\Psi$DO:  The symbol of $i[\Psi_a\Psi_\phi - \Psi_\phi\Psi_a]=H_a(\phi)$ modulo lower order symbols.

In order to explain the relevance of this to our context, we first recall a result of Ichinose \cite{Ic}, which generalized the Takeuchi-Mizohata condition:
Consider the Cauchy problem
\begin{equation*}
\left\{
\begin{aligned}
\partial_t u &= i \partial_{x_k} a_{k,j}(x)\partial_{x_j}u + \vec{b}_1(x)\cdot \nabla u + f \\
u\big|_{t=0} &= u_0
\end{aligned}
\right.
\end{equation*}
where $(a_{k,j}(x))$ is elliptic, and ``asymptotically flat'' (i.e.\ $|a_{k,j}-\delta_{kj}| \leq \frac{C}{\langle x \rangle^{1+\eta}}$, $\eta>0$, as $x\to \infty$).  Then, a necessary condition for the estimate
$$\sup_{0<t<T} \|u(t)\|_{L^2} \leq C_T \left\{ \|u_0\|_{L^2} + \|f\|_{L_T^1L_x^2} \right\}$$
to hold  is:
$$\sup_{\substack{ x_0\in \mathbb{R}^n \\ \xi_0\in S^{n-1} \\ t_0\in \mathbb{R} }} \left| \int_0^{t_0} \text{Im }\vec{b}_1(X(s; x_0,\xi_0)) \cdot \Xi(s; x_0, \xi_0) \, ds \right| < +\infty$$
This is a direct generalization of the Takeuchi-Mizohata condition.  Notice that, by ellipticity, if $\xi_0\in S^{n-1}$, $|\Xi(s; x_0, \xi_0)| \simeq 1$.  Thus, a very bad situation would arise if for a fixed $(x_0,\xi_0)$, $X(s; x_0, \xi_0)$ was contained, $\forall \; s$, in a fixed compact $K$.  One would then choose $\vec{b}_1\in C_0^\infty$, but identically constant in $K$, and trouble would arise as $t_0\to \infty$ for the boundedness of this integral.  This leads us to a non-trapping condition:  For each $(x_0,\xi_0)\in \mathbb{R}^n\times \mathbb{R}^n\backslash \{ 0 \}$, the set $\{ X(s; x_0, \xi_0) \}$ is unbounded in $\mathbb{R}^n$.  The non-trapping condition, in fact, is also closely connected to the ``local smoothing'' estimate.  In fact, the role of this property in the study of local smoothing effects first appeared in the works of Kapitanski-Safarov\cite{KS}, Craig-Kappeler-Strauss \cite{CKS}, and Doi \cite{Do1}\cite{Do2}, from the early to mid 90's.  They showed that, under appropriate smoothness assumptions, ellipticity and asymptotic flatness assumptions, if the non-trapping condition holds, solutions to 
\begin{equation}
\left\{
\begin{aligned}
\partial_t u &= i \partial_{x_k}a_{kj}(x)\partial_{x_j}u \\
u\big|_{t=0} &= u_0
\end{aligned}
\right.
\end{equation}
verify the estimate $\|J^{1/2}u\|_{L^2(\mathbb{R}^n\times [0,T], \lambda_m(x)dxdt)}\leq C_T\|u_0\|_{L^2}$, $m>1$.  Moreover, Doi \cite{Do3} showed that, under the same conditions, if the above estimate holds, the non-trapping assumption must hold.

In our work on quasi-linear equations, we need to study equations whose coefficients depend also on $t$.  Moreover, we need to introduce an ``artificial viscosity'' term $\epsilon \Delta^2$, and establish a family of uniform estimates.  We thus need to study the following family of linear Cauchy problems:
\begin{equation*}
 \left\{ 
\begin{aligned}
\partial_t u &= -\epsilon \Delta^2 u + ia_{lk}(x,t) \partial_{x_lx_k}^2u + ib_{lk}(x,t) \partial_{x_lx_k}^2\bar{u} \\
& \quad + \vec{b}_1(x,t)\cdot \nabla_xu + \vec{b}_2(x,t)\cdot \nabla_x\bar{u} \\
& \quad + c_1(x,t)u+c_2(x,t)\bar{u} + f(x,t) \\
u\big|_{t=0} &= u_0 
\end{aligned}
\right. 
\end{equation*}
where $a_{lk}(x,t)$ are real valued, and the other coefficients may be complex.  In order to clarify the presentation, we will first deal with $b_{lk}(x,t)\equiv 0$, and then treat the general case.  We thus study, for $0\leq \epsilon \leq 1$:
\begin{equation} \label{E:13}
 \left\{ 
\begin{aligned}
\partial_t u &= -\epsilon \Delta^2 u + ia_{lk}(x,t) \partial_{x_lx_k}^2u \\
& \quad + \vec{b}_1(x,t)\cdot \nabla_xu + \vec{b}_2(x,t)\cdot \nabla_x\bar{u} \\
& \quad + c_1(x,t)u+c_2(x,t)\bar{u} + f(x,t) \\
u\big|_{t=0} &= u_0 
\end{aligned}
\right. 
\end{equation}
Our hypotheses on the coefficients are:\\

\noindent ($\text{H}_{1,l}$) \underline{Ellipticity}. $(a_{lk}(x,t))$ is real valued, and with $h(x,\xi)=a_{lk}(x,0)\xi_l\xi_k$, we have $h(x,\xi) \geq \gamma |\xi|^2$, $\gamma >0$.

\noindent ($\text{H}_{2,l}$) \underline{Regularity}. There exists $N=N(n)$, and $C>0$ so that $a_{lk}$, $\vec{b}_1$, $\vec{b}_2$, $c_1$, $c_2 \in C_b^N(\mathbb{R}^n\times \mathbb{R})$ with norm bounded by $C$.  Moreover, there exists $C_1>0$, so that the corresponding norms, at $t=0$, as functions of $x$ in $C_b^N(\mathbb{R}^n)$, are bounded by $C_1$. 

\noindent $(\text{H}_{3,l})$ \underline{Asymptotic flatness}.  There exists $C>0$, $C_1>0$ such that, for any $(x,t)\in \mathbb{R}^n\times \mathbb{R}$, we have 
$$|\nabla_x a_{lk}(x,t)| + |\partial_ta_{lk}(x,t)|+|\partial_{x_j} \partial_{x_r} a_{lk}(x,t)| + |\partial_{x_j} \partial_t a_{lk}(x,t)| \leq \frac{C}{\langle x \rangle^2}$$
$$|\nabla_x a_{lk}(x,0)| \leq \frac{C_1}{\langle x \rangle^2}$$

\noindent $(\text{H}_{4,l})$ \underline{Growth of 1st order coefficients}.  There exists $C, C_1>0$ such that, for $(x,t)\in \mathbb{R}^n\times \mathbb{R}$,
$$|\partial_t \text{Im }\vec{b}_1(x,t)| \leq \frac{C}{\langle x \rangle^2}$$
$$|\text{Im }\vec{b}_1(x,0)| \leq \frac{C_1}{\langle x \rangle^2}$$

\noindent $(\text{H}_{5,l})$ \underline{Non-trapping}.  The symbol $h(x,\xi)=a_{lk}(x,0)\xi_l\xi_k$ gives rise to a ``non-trapping'' bi-characteristic flow, with non-trapping character controlled by $C_1$. (We will explain this quantitative dependence very soon).

\begin{theorem} \label{T:1B}
There exist $N=N(n)$, $T_0>0$, $T_0$ depending only on $C$, $C_1$, $(\text{H}_{5,l})$, and $A>0$, depending only on $\gamma$, $C_1$ and $(\text{H}_{5,l})$, so that, for any $T\in [0,T_0]$, $\epsilon\in [0,1]$, we have, for any $u$ a solution of \eqref{E:13}:
\begin{align*}
 \hspace{0.5in}&\hspace{-0.5in} \sup_{0<t<T} \|u(t)\|_{L^2} + \left( \int_0^T \int \lambda_2 |J^{1/2}u|^2 \, dx\, dt \right)^{1/2} \\
&\leq A 
\begin{cases}
\|u_0\|_{L^2} + \|f\|_{L_T^1L_x^2} \\
\|u_0\|_{L^2} + \|J^{-1/2}f\|_{L^2(\mathbb{R}^n\times [0,T], \lambda_2^{-1}(x)dxdt)}
\end{cases}
\end{align*}
\end{theorem}
We will see that $\lambda_2$ here can be replaced by $\lambda_m$, $m>1$.  We will sketch the proof of the first inequality, the second one being similar.  One of our main tools will be the following lemma due to S. Doi \cite{Do2}.

\begin{lemma} Let $h$ be as above.  Then, there exists a real valued symbol $p\in S^0$, with seminorms for $|\alpha|+|\beta|\leq M(N)$ bounded in terms of ``the non-trapping character of $h$'', $C_1$, and $\gamma$, by a constant $C_1^\ast$, and a constant $B$, $1>B>0$, with the same dependence, such that
$$H_hp \geq \frac{B|\xi|}{\langle x \rangle^2}-\frac{1}{B}, \quad \forall \; (x,\xi)\in \mathbb{R}^n\times \mathbb{R}$$
\end{lemma}
We will quantify the ``non-trapping'' character of $h$ in terms of the constants $C_1^\ast$ and $B$ in the above lemma.
\begin{remark} 
The fact that the constant $A$ in Theorem \ref{T:1B} depends only on the coefficients at $t=0$ will be a crucial point in the application of this result to the non-linear problem. 
\end{remark}
 We will take Doi's lemma for granted, and use it to prove Theorem \ref{T:1B}.  At the end, we will prove the lemma.  We proceed in several steps.

\underline{Step 1. Reduction to a system}. We look at \eqref{E:13} and its complex conjugate, to obtain, with $\vec{w}= \left( \begin{matrix} u \\ \bar{u} \end{matrix} \right)$, $\vec{f} = \left( \begin{matrix} f \\ \bar{f} \end{matrix} \right)$, $\vec{w}_0 = \begin{pmatrix} u_0 \\ \bar{u}_0 \end{pmatrix}$, the system
\begin{equation*}
\left\{
\begin{aligned}
\partial_t \vec{w} &= -\epsilon \Delta^2I\vec{w} + [iH+B+C]\vec{w}+\vec{f} \\
\vec{w}\big|_{t=0} &= \vec{w}_0
\end{aligned}
\right.
\end{equation*}
where \hfill $H=\begin{pmatrix} \mathcal{L} & 0 \\ 0 & -\mathcal{L} \end{pmatrix}$, \hfill $\mathcal{L}=a_{lk}(x,t)\partial_{x_lx_k}^2$, \hfill $B=\begin{pmatrix} \vec{b}_1\cdot \nabla & \vec{b}_2\cdot \nabla \\ \overline{\vec{b}}_2\cdot\nabla & \overline{\vec{b}}_1\cdot \nabla \end{pmatrix} = \begin{pmatrix} B_{11} & B_{12} \\ B_{21} & B_{22} \end{pmatrix}$, \\ $C= \begin{pmatrix} c_1 & c_2 \\ \overline{c}_2 & \overline{c}_1 \end{pmatrix}$.

\underline{Step 2. Diagonalization of the 1st order terms}.  We first note that, given $(\text{H}_{1,l})$, $(\text{H}_{2,l})$, $\mathcal{L}$ is elliptic, with ellipticity constant $\gamma/2$, for $0<t<T$, $T$ sufficiently small, depending on $C$:
$$a_{lk}(x,t)\xi_l\xi_k = a_{lk}(x,0)\xi_l\xi_k + [a_{lk}(x,t)-a_{lk}(x,0)]\xi_l\xi_k\geq \gamma|\xi|^2-CT|\xi|^2$$ because of our bounds on $\partial_t a_{lk}$.  This type of argument is used frequently.  We write 
$$B=B_{\text{diag}}+B_{\text{anti}}= \begin{pmatrix} B_{11} & 0 \\ 0 & B_{22} \end{pmatrix}+ \begin{pmatrix} 0 & B_{12} \\ B_{21} & 0 \end{pmatrix}$$
 and our goal in this step is to eliminate $B_{\text{anti}}$.  To do this, we will set $\Lambda=I-S$, where $S= \begin{pmatrix} 0 & s_{12} \\ s_{21} & 0 \end{pmatrix}$, and $S$ is of order $-1$, and write our system in the new variable $\vec{z}=\Lambda \vec{w}$, for appropriately chosen $S$, so that $B_{\text{anti}}$ is eliminated, and $\Lambda$ is invertible.  Estimates on $\vec{z}$ are then equivalent to estimates on $\vec{w}$.  Let $h(x,t,\xi)=a_{lk}(x,t)\xi_l\xi_k$, so that, for each $t$, $\mathcal{L}=\Psi_{-h}$.  Choose $\Phi\in C_0^\infty(\mathbb{R}^n)$, $\Phi(y)\equiv 1$ for $|y|\leq 1$, $\Phi(y)\equiv 0$ for $|y|\geq 2$, $\theta_R(\xi)= [1-\Phi(\xi/R)]$, and let $\tilde{h}(x,t,\xi)=-h^{-1}(x,t,\xi)\cdot\theta_R(\xi)$.  Let $\tilde{\mathcal{L}}=\Psi_{\tilde{h}}$, so that $\tilde{h}\in S^{-2}$, uniformly in $t$, and $\tilde{\mathcal{L}}\mathcal{L}=I+\Psi_{r_1}$, $r_1\in S^{-1}$, uniformly in $t$.  We now define $s_{12}=-\frac{1}{2}iB_{12}\tilde{\mathcal{L}}$, $s_{21}=+\frac{1}{2}iB_{21}\tilde{\mathcal{L}}$, $S=\begin{pmatrix} 0 & s_{12} \\ s_{21} & 0 \end{pmatrix}$, $\Lambda=I-S$.  Notice that the entries of $S$ are of order $-1$, and we can choose $R$ so large that $\Lambda^{-1}$ is a 0-th order $\Psi$DO, for each $t$, and so that $\Lambda$ has operator norm in $H^{1/2}(\langle x \rangle^2dx)$, $L^2(\mathbb{R}^n)$ between $(\frac{1}{2},2)$, and so does $\Lambda^{-1}$.  Let us first consider 
\begin{align*}
i\begin{pmatrix} \mathcal{L} & 0 \\ 0 & -\mathcal{L} \end{pmatrix}\Lambda-\Lambda i\begin{pmatrix} \mathcal{L} & 0 \\ 0 & -\mathcal{L} \end{pmatrix} & = i \begin{pmatrix} \mathcal{L} & 0 \\ 0 & -\mathcal{L} \end{pmatrix} \begin{pmatrix} 0 & -s_{12} \\ -s_{21} & 0 \end{pmatrix} + i\begin{pmatrix} 0 & s_{12} \\ s_{21} & 0 \end{pmatrix} \begin{pmatrix} \mathcal{L} & 0 \\ 0 & -\mathcal{L} \end{pmatrix}\\
& = \begin{pmatrix} 0 & -i\mathcal{L}s_{12}-is_{12}\mathcal{L} \\ i\mathcal{L}s_{21}+is_{21}\mathcal{L} & 0 \end{pmatrix}
\end{align*}
Now, 
$$-i\mathcal{L}s_{12}-is_{12}\mathcal{L}= -\frac{1}{2}\mathcal{L}B_{12}\tilde{\mathcal{L}}-\frac{1}{2}B_{12}\tilde{\mathcal{L}}\mathcal{L}=-B_{12}+\text{order }0$$
and similarly 
$$i(\mathcal{L}s_{21}+s_{21}\mathcal{L})= -B_{21}+\text{order }0$$
Thus, $i(\Lambda H - H \Lambda)=-B_{\text{anti}}+C$, $C$ of order $0$.  Observe also that 
\begin{align*}
\Lambda B_{\text{diag}} & =(I-S)B_{\text{diag}}=B_{\text{diag}}-SB_{\text{diag}}=B_{\text{diag}}\Lambda +B_{\text{diag}}S-SB_{\text{diag}} \\
& =B_{\text{diag}}\Lambda+ \{ \underbrace{(B_{\text{diag}}S-SB_{\text{diag}})\Lambda^{-1}}_{\text{order }0} \} \Lambda = B_{\text{diag}}\Lambda + C\Lambda
\end{align*}
where $C$ is of order $0$, uniformly in $t$.  Also, $\Lambda B_{\text{anti}} = B_{\text{anti}} - S B_{\text{anti}} = B_{\text{anti}}+C\Lambda$, $C$ of order $0$, and $\Lambda C = (\Lambda C \Lambda^{-1})\Lambda = \tilde{C}\Lambda$, where $\tilde{C}$ is again of order $0$.  It remains to study $\Lambda \Delta^2I$, and $\Lambda \partial_t \vec{w}$.  The latter one equals $\partial_t \Lambda \vec{w}-(\partial_t \Lambda)\vec{w}=\partial_t\Lambda\vec{w}-(\partial_t\Lambda)\Lambda^{-1}\Lambda \vec{w}$, and $(\partial_t \Lambda)\Lambda^{-1}$ is of order $0$, uniformly in $t$.  For,
$$\Lambda \Delta^2 I = \Delta^2I- \begin{pmatrix} 0 & s_{12} \\ s_{21} & 0 \end{pmatrix} \Delta^2I=\Delta^2I- \begin{pmatrix} 0 & s_{12}\Delta^2 \\ s_{21}\Delta^2 & 0 \end{pmatrix}$$
 and $s_{12}\Delta^2 = \Delta^2s_{12}+R_{12}$, where $R_{12}$ is of order $2$.  Thus, $\epsilon\Lambda\Delta^2I=\epsilon \Delta^2 \Lambda I+\epsilon \tilde{R}$, where the entries of $\tilde{R}$ are of order $2$.  Writing $\tilde{R}=(\tilde{R}\Lambda^{-1})\Lambda = R\Lambda$, where the entries of $R$ are of order $2$, we obtain
$$
\left\{
\begin{aligned}
\partial_t \vec{z} &= -\epsilon \Delta^2 I \vec{z} + \epsilon R \vec{z} + i H \vec{z} + B_{\text{diag}}\vec{z}+C\vec{z}+\vec{F} \\
z\big|_{t=0} &= z_0
\end{aligned}
\right.
$$
where $\Lambda \vec{f}=\vec{F}$, $\Lambda \vec{w}_0 = z_0$, $R$ is of order 2, $C$ is of order $0$, uniformly in $t$.

\underline{Step 3. Construction of a ``gauged'' system}.  Recall that our ``non-trapping'' assumption is on $h(x,\xi)=a_{lk}(x,0)\xi_l\xi_k$, and that the symbol of $\mathcal{L}$ is $-h(x,t,\xi)=-a_{lk}(x,t)\xi_l\xi_k$.  Let $p \in S^0$ be the symbol associated to $h$, through Doi's Lemma, so that, 
$$H_hp \geq \frac{B|\xi|}{\langle x \rangle^2} - \frac{1}{B}$$
 and recall that, the symbol of $i[\Psi_c \Psi_a - \Psi_a \Psi_c]=-H_a(c)$, modulo lower order terms.  Note that, if $h_t(x,\xi)=h(x,t,\xi)$, 
\begin{align*}
H_{h_t}p &= \sum \frac{\partial h_t}{\partial \xi_j} \frac{\partial p}{\partial x_j} - \frac{\partial h_t}{\partial x_j} \frac{\partial p}{\partial \xi_j} \\
&= \sum \frac{\partial h}{\partial \xi_j}\frac{\partial p}{\partial x_j} - \frac{\partial h}{\partial x_j}  \frac{\partial p}{\partial \xi_j}+ \\
& \qquad + \sum \{ a_{lk}(x,t)-a_{lk}(x,0) \} \frac{\partial}{\partial \xi_j} (\xi_l \xi_k) \frac{\partial p }{\partial x_j} \\
& \qquad - \sum \left\{ \frac{\partial}{\partial x_j} a_{lk}(x,t)-\frac{\partial}{\partial x_j} a_{lk}(x,0) \right\} \xi_l \xi_k \frac{\partial p}{\partial \xi_j}
\end{align*}
so that, by our asymptotic flatness assumption $(\text{H}_{3,l})$, for small $T_0$, (depending on $C, C_1$), we have, $t<T_0$, 
$$H_{h_t}p \geq \frac{B}{2}\cdot \frac{|\xi|}{\langle x \rangle^2} - \frac{2}{B}$$
 We now define, for $M$ large, $R$ large to be chosen, $r_1(x,\xi)=\exp( ( -Mp(x,\xi)\theta_R(\xi) )$, $r_2(x,\xi)=\exp(+Mp(x,\xi)\theta_R(\xi))$, so that $\Psi_{r_1}$, $\Psi_{r_2}$ depend only on $M$, $h(x,\xi)$, $R$.  Note that $\Psi_{r_1}\Psi_{r_2}=I+T_{R_{-1}^1}$; $\Psi_{r_2}\Psi_{r_1} = I + T_{R_{-1}^2}$, where $T_{R_{-1}^i}$ are of order $-1$, with $S^0$ seminorms small in $R$.  Thus, for $R$ large, $\Psi_{r_1}$, $\Psi_{r_2}$ are invertible and their inverses are operators of order $0$.  This fixes $R$, depending only on $h(x,\xi)$, $M$, and controls the norm of $\Psi_{r_i}$, $\Psi_{r_i}^{-1}$ in various spaces, only in terms of those quantities.  Note that, modulo 0th order operators, the symbol of 
$$ i[\Psi_{r_1}\mathcal{L}-\mathcal{L}\Psi_{r_1}]  =i[\Psi_{r_1}\Psi_{-h_t}-\Psi_{-h_t}\Psi_{r_1} ] = -i[ \Psi_{r_1}\Psi_{h_t}-\Psi_{h_t}\Psi_{r_1}] $$ 
is  $H_{h_t}r_1 = -MH_{h_t}(p\theta_R)r_1$.
A similar computation gives that the symbol of $i[-\Psi_{r_2}\mathcal{L}+\mathcal{L}\Psi_{r_2}]=-MH_{h_t}(p\theta_R)r_2$, modulo 0th order operators.  Let now $\Psi = \begin{pmatrix} \Psi_{r_1} & 0 \\ 0 & \Psi_{r_2} \end{pmatrix}$, and define $\vec{\alpha}=\Psi\vec{z}$.  We will write the system for $\vec{\alpha}$.  The constant $M$, and hence $R$, will be eventually chosen depending only on $C_1$.
\begin{align*}
i[\Psi H - H \Psi] &= i \left[ \begin{pmatrix} \Psi_{r_1} & 0 \\ 0 & \Psi_{r_2} \end{pmatrix} \begin{pmatrix} \mathcal{L} & 0 \\ 0 & -\mathcal{L} \end{pmatrix} - \begin{pmatrix} \mathcal{L} & 0 \\ 0 & -\mathcal{L} \end{pmatrix}\begin{pmatrix} \Psi_{r_1} & 0 \\ 0 & \Psi_{r_2} \end{pmatrix} \right] \\
& = \begin{pmatrix} -M \Psi_{H_{h_t}(\theta_R p)} & 0 \\ 0 & -M\Psi_{H_{h_t}(\theta_R p)} \end{pmatrix} \Psi + \tilde{C}, \quad \tilde{C} \text{ order }0 \\
&= \begin{pmatrix} -M\Psi_{H_{h_t}(\theta_R p)} & 0 \\ 0 & -M\Psi_{H_{h_t}(\theta_R p)} \end{pmatrix} \Psi + C\Psi, \quad C \text{ order }0 
\end{align*}
Next, 
\begin{align*}
\Psi B_{\text{diag}} &= \begin{pmatrix} \Psi_{r_1}B_{11} & 0 \\ 0 & \Psi_{r_2}B_{22} \end{pmatrix} = \begin{pmatrix} B_{11}\Psi_{r_1} & 0 \\ 0 & B_{22}\Psi_{r_2} \end{pmatrix} + \text{$0$th order} \\
&= B_{\text{diag}}\Psi+C\Psi, \quad C \text{ order }0 \\
\Psi C &= (\Psi C\Psi^{-1})\Psi
\end{align*}
Since $\Psi$ is $t$ independent, $\Psi\partial_t = \partial_t\Psi$, so it only remains to study $\Psi R$, $R$ of order $2$, and $\Psi\Delta^2I$.  We just write $\Psi R= (\Psi R\Psi^{-1})\Psi$, and note that 
$$\Psi\Delta^2 I  = \begin{pmatrix} \Psi_{r_1}\Delta^2 & 0 \\ 0 & \Psi_{r_2}\Delta^2 \end{pmatrix} = \begin{pmatrix}\Delta^2\Psi_{r_1} & 0 \\ 0 & \Delta^2\Psi_{r_2} \end{pmatrix} +E$$
$E$ of order $3$.  All in all, we obtain, (grouping the order $2$ and order $3$ terms together)
$$\left\{
\begin{aligned}
\partial_t \vec{\alpha} &= 
\begin{aligned}[t]
&-\epsilon \Delta^2I\vec{\alpha} + \epsilon E\vec{\alpha} + iH\vec{\alpha} + B_{\text{diag}}\vec{\alpha} \\
&-M \begin{pmatrix} \Psi_{H_{h_t}(\theta_R p)} & 0 \\ 0 & \Psi_{H_{h_t}(\theta_R p)} \end{pmatrix} \vec{\alpha} + C\vec{\alpha}+\vec{G}
\end{aligned} \\
\vec{\alpha}\big|_{t=0} & = \vec{\alpha}_0
\end{aligned}
\right.$$
where $E$ is of order $3$, uniformly in $t$, $C$ is of order 0, $B_{\text{diag}}= \begin{pmatrix} B_{11} & 0 \\ 0 & B_{22} \end{pmatrix}$, $B_{11}=\vec{b}_1\cdot \nabla$, $B_{22}=\overline{\vec{b}}_1 \cdot \nabla$, $\vec{G}=\Psi \vec{F}$, $\vec{\alpha}_0 = \Psi \vec{w}_0$, and where $M$ and hence $R$ are still to be chosen, depending only on $C_1$. (Proof to be continued in Lecture 5).\\

\section*{Problems for Lecture \arabic{chapter}}
\begin{problems}
\item Prove that, under ellipticity, the Hamiltonian flow exists for all $s$.
\item Prove that, under ellipticity and asymptotic flatness, the non-trapping condition implies the following ``uniform non-trapping'':  Let  $K\subset \mathbb{R}^n\times \mathbb{R}^n\backslash \{ 0\}$ be a compact set.  Then, given $\mu>0$, $\exists \; s_0 = s_0(K, \mu)$ such that $|X(s; x_0, \xi_0)| \geq \mu$, $\forall \; |s|\geq |s_0|$, $\forall \; (x_0, \xi_0)\in K$.
\item Combine Step 2 in the proof of Theorem \ref{T:1B} with the technique used for the proof of Problem \chapterref{P:4}.\ref{P:4}, to give a proof of Theorem \ref{T:S3}.
\item  Use the problem before, and the techniques in Lecture 3, to remove the smallness assumption in Theorems \ref{T:S1} and \ref{T:S2}.
\end{problems}

\lecture[Variable coefficient linear Schr\"{o}dinger, II]{The linear Schr\"{o}dinger equation with variable coefficients, Part II}
We continue with our proof of Theorem \ref{T:1B}.  

\underline{Step 4.  Energy estimates for the ``gauged'' system}.  We will now choose $M$ large, depending on $C_1$, so that $\vec{\alpha}$ verifies the estimates in Theorem \ref{T:1B}.  This will, in turn, imply Theorem \ref{T:1B}.  We define $\langle \vec{\alpha}, \vec{\beta} \rangle = \int \alpha_1\bar{\beta}_1 + \alpha_2\bar{\beta}_2$, and set out to compute 
\begin{align*}
\partial_t \langle \vec{\alpha}, \vec{\alpha} \rangle 
&= -\epsilon \langle \Delta^2 I \vec{\alpha}, \vec{\alpha} \rangle - \epsilon \langle \vec{\alpha}, \Delta^2 I \vec{\alpha} \rangle + \epsilon \langle E \vec{\alpha}, \vec{\alpha} \rangle + \epsilon \langle \vec{\alpha}, E\vec{\alpha} \rangle \\
& \qquad + \langle iH\vec{\alpha}, \vec{\alpha} \rangle + \langle \vec{\alpha}, iH\vec{\alpha} \rangle + \langle B_{\text{diag}}\vec{\alpha}, \vec{\alpha} \rangle + \langle \vec{\alpha}, B_{\text{diag}}\vec{\alpha} \rangle \\
& \qquad -M  \left< \begin{pmatrix} \Psi_{H_{h_t}(\theta_R p)} & 0 \\ 0 & \Psi_{H_{h_t}(\theta_R p)} \end{pmatrix} \vec{\alpha}, \vec{\alpha} \right> \\
& \qquad -M \left< \vec{\alpha}, \begin{pmatrix} \Psi_{H_{h_t}(\theta_R p)} & 0 \\ 0 & \Psi_{H_{h_t}(\theta_R p)} \end{pmatrix} \vec{\alpha} \right>  \\
& \qquad + \langle C\vec{\alpha}, \vec{\alpha} \rangle + \langle \vec{\alpha}, C\vec{\alpha} \rangle + \langle \vec{G}, \vec{\alpha} \rangle + \langle \vec{\alpha}, \vec{G} \rangle \\
&= -2\epsilon \text{Re }\langle \Delta^2I\vec{\alpha},  \vec{\alpha} \rangle + 2\epsilon \text{Re  }\langle E\vec{\alpha}, \vec{\alpha} \rangle + i[ \langle H\vec{\alpha}, \vec{\alpha}\rangle - \langle \vec{\alpha}, H\vec{\alpha} \rangle ] \\
& \qquad + 2 \text{Re} \langle B_{\text{diag}}\vec{\alpha}, \vec{\alpha} \rangle -2M \text{Re } \left< \begin{pmatrix} \Psi_{H_{h_t}(\theta_Rp)} & 0 \\ 0 & \Psi_{H_{h_t}(\theta_R p)}  \end{pmatrix} \vec{\alpha}, \vec{\alpha} \right> \\
& \qquad + 2\text{Re }\langle C\vec{\alpha}, \vec{\alpha} \rangle + 2\text{Re }\langle \vec{G}, \vec{\alpha} \rangle
\end{align*}
We start analyzing terms:\\
\noindent \textbf{Term I}. $ -2\epsilon \text{Re } \langle \Delta^2 I \vec{\alpha}, \vec{\alpha} \rangle = -2\epsilon \|\Delta \vec{\alpha} \|_{L^2}^2$.

\noindent \textbf{Term II}.  
$$ \epsilon |\text{Re}\langle E\vec{\alpha} , \vec{\alpha} \rangle | = \epsilon |\text{Re}\langle J^{3/2}J^{-3/2}E\vec{\alpha}, \vec{\alpha} \rangle| = \epsilon |\text{Re}\langle J^{-3/2}E\vec{\alpha}, J^{3/2}\vec{\alpha} \rangle| \leq C\epsilon \| \vec{\alpha}\|_{H^{3/2}}^2$$
since $J^{-3/2}E$ is of order $3/2$.  We now invoke the interpolation inequality$$\|\vec{\alpha}\|_{H^{3/2}}^2 \leq \eta_0 \|\vec{\alpha} \|_{H^2}^2 + \frac{1}{\eta_0} \|\vec{\alpha} \|_{L^2} \leq \eta_0\|\Delta \vec{\alpha} \|_{L^2}^2 + \frac{C_n}{\eta_0} \|\vec{\alpha} \|_{L^2}$$
where $\eta_0>0$ is arbitrary.  If we now choose $\eta_0 = \eta_0(C, n)$ so small that $C\eta_0 \leq 1$, we obtain that the sum of Terms I and II is smaller than $-\epsilon \|\Delta \vec{\alpha} \|_{L^2}^2 + C\epsilon \|\vec{\alpha} \|_{L^2}^2$.

\noindent \textbf{Term III}.  Next, 
$$ i[ \langle H \vec{\alpha}, \vec{\alpha} \rangle - \langle \vec{\alpha}, H \vec{\alpha} \rangle] = i [ \langle (H-H^\ast)\vec{\alpha}, \vec{\alpha} \rangle ]$$
Since $H= \begin{pmatrix} \mathcal{L} & 0 \\ 0 & -\mathcal{L} \end{pmatrix}$, $H-H^\ast = \begin{pmatrix} \mathcal{L} - \mathcal{L}^\ast & 0 \\ 0 & \mathcal{L}^\ast - \mathcal{L} \end{pmatrix}$.  Recall that $\mathcal{L}=\Psi_{-h_t}$, $h_t(x,\xi)= h(x,t,\xi)=a_{kl}(x,t)\xi_k\xi_l$.  Thus, by the calculus, since $h_t$ is real valued, the symbol of $\mathcal{L}^\ast$ is 
$$-a_{kl}(x,t)\xi_k\xi_l-i\sum_{j=1}^n \partial_{\xi_j} \partial_{x_j} [ a_{kl}(x,t)\xi_k\xi_l ] + \text{order }0$$
and hence, if $\tilde{b}(x,t,\xi) = \sum_{j=1}^n \partial_{\xi_j} \partial_{x_j} [ a_{kl}(x,t)\xi_k\xi_l ]$, $|\partial_t \tilde{b}(x,t,\xi)| \leq C \frac{|\xi|}{\langle x \rangle^2}$, and $|\tilde{b}(x,0,\xi)|\leq C_1 \frac{|\xi|}{\langle x \rangle^2}$, and $C^N$ norms of $\tilde{b}(x,t,\xi)$, $\tilde{b}(x,0,\xi)$ have similar bounds, because of $(\text{H}_{3,l})$, $(\text{H}_{2,l})$.  We have then 
$$i[ \langle (H-H^\ast)\vec{\alpha}, \vec{\alpha} \rangle ] =  \left< \begin{pmatrix} \Psi_{-\tilde{b}} & 0 \\ 0 & \Psi_{\tilde{b}} \end{pmatrix} \vec{\alpha}, \vec{\alpha} \right>  + \langle C \vec{\alpha}, \vec{\alpha} \rangle$$
where $C$ is of order $0$, and, since the symbol $\tilde{b}$ is real valued,
$$\left< \begin{pmatrix} \Psi_{-\tilde{b}} & 0 \\ 0 & \Psi_{\tilde{b}} \end{pmatrix} \vec{\alpha}, \vec{\alpha} \right> = \text{Re } \left< \begin{pmatrix} \Psi_{-\tilde{b}} & 0 \\ 0 & \Psi_{\tilde{b}} \end{pmatrix}\vec{\alpha}, \vec{\alpha} \right> + \langle C \vec{\alpha}, \vec{\alpha} \rangle, \quad C \text{ order }0$$
\noindent \textbf{Term IV}.  We next consider the terms
\begin{align*}
\text{Re} \left< \begin{pmatrix} \Psi_{-\tilde{b}} & 0 \\ 0 & \Psi_{\tilde{b}} \end{pmatrix} \vec{\alpha}, \vec{\alpha} \right> + 2\text{Re} \langle B_{\text{diag}} \vec{\alpha}, \vec{\alpha} \rangle \hspace{-1in}&\\
-2M \text{Re} \left< \begin{pmatrix} \Psi_{H_{h_t}(\theta_R p)} & 0 \\ 0 & \Psi_{H_{h_t}(\theta_R p)} \end{pmatrix} \vec{\alpha}, \vec{\alpha} \right> \hspace{-1in}& \\
&= \begin{aligned}[t]
&\text{Re} \left< \begin{pmatrix} \Psi_{-\tilde{b}}+2B_{11} & 0 \\ 0 & \Psi_{\tilde{b}} + 2B_{22} \end{pmatrix} \vec{\alpha}, \vec{\alpha} \right> \\
&+ \text{Re} \left< \begin{pmatrix} -2M\Psi_{H_{h_t}(\theta_R p)} & 0 \\ 0 & -2M\Psi_{H_{h_t}(\theta_R p)} \end{pmatrix} \vec{\alpha}, \vec{\alpha} \right>
\end{aligned}\\
&=\begin{aligned}[t] 
&\text{Re} \int [ \Psi_{-\tilde{b}} + 2B_{11}-2M\Psi_{H_{h_t}(\theta_R p)}] (\alpha_1) \, \overline{\alpha}_1  \\
&+ \text{Re} \int [ \Psi_{\tilde{b}} + 2B_{22}-2M\Psi_{H_{h_t}(\theta_R p)} ] (\alpha_2) \, \overline{\alpha}_2
\end{aligned}
\end{align*}
Now, 
$$H_{h_t}(\theta_R p)= \theta_R(\xi) H_{h_t}(p) + p H_{h_t}(\theta_R)$$
and $H_{h_t}(\theta_R) = -\partial_{x_j}h(x,t,\xi) \partial_{\xi_j}\theta_R$, so that $|H_{h_t}(\theta_R)p|\leq C$, for $R\geq 1$.  Also, $B_{11} = \vec{b}_1 \cdot \nabla$, $B_{22}= \overline{\vec{b}}_1 \cdot \nabla$, so that $B_{11}=\Psi_{i\vec{b}_1(x,t)\cdot \xi}$, $B_{22}=\Psi_{i\overline{\vec{b}}_1(x,t)\cdot \xi}$.  Consider $\text{Re }(i\vec{b}_1(x,t)\cdot \xi)=-\text{Im }(\vec{b}_1(x,t)\cdot \xi)$.  Now $|\text{Im }\vec{b}_1(x,0,\xi)| \leq \frac{C_1}{\langle x \rangle^2}$, and  $|\partial_t \text{Im } \vec{b}_1(x,t,\xi) | \leq \frac{C}{\langle x \rangle^2}$ implies that, for $0<t<T_0$, $T_0=T_0(C)$, we have 
$$|\text{Im }\vec{b}_1(x,t,\xi)| \leq \frac{2C_1}{\langle x \rangle^2}$$
  A similar estimate holds for $\tilde{b}$.  Recall also from Doi's lemma and our previous remark that, for $0<t<T_0$, 
$$H_{h_t}p \geq \frac{B}{2}\frac{|\xi|}{\langle x \rangle^2} - \frac{2}{B}$$
  Then, we have
\begin{align*}
\hspace{.5in}&\hspace{-.5in}\text{Re }[ -\tilde{b}(x,t,\xi) + 2i\vec{b}_1(x,t)\cdot \xi -2M H_{h_t}(\theta_R p) ] \\
& \leq  \frac{6C_1 |\xi|}{\langle x \rangle^2}+\frac{4M}{B} \theta_R(\xi) + 2MC- MB\frac{|\xi|}{\langle x \rangle^2}\theta_R(\xi)
\end{align*}
Now, choose $M$ so that $MB>6C_1+1$.  Now, choose $R$, depending on $M$, so that $\Psi_{r_i}$ in the construction of the gauged system are invertible.  We then have, for $|\xi|\geq R$, that 
$$\text{Re} [ -\tilde{b}(x,t,\xi) + 2i\vec{b}_1(x,t)\cdot \xi - 2M H_{h_t}(\theta_R p) ] \leq 2MC - \frac{|\xi|}{\langle x \rangle^2}$$
and so, by the sharp G\r{a}rding inequality, we have
$$\text{Re} \int [ \Psi_{-\tilde{b}}+2B_{11}-2M\Psi_{H_{h_t}(\theta_Rp)}]\alpha_1 \cdot \bar{\alpha}_1  \leq C \|\alpha_1\|_{L^2}^2 - \text{Re}\left< \Psi_{\langle \xi \rangle/\langle x \rangle^2} \alpha_1, \alpha_1 \right>$$
Thus, our sum of terms is bounded by 
$$C\|\vec{\alpha}\|_{L^2}^2 - \text{Re} \left< \begin{pmatrix} \Psi_{\langle \xi \rangle/\langle x \rangle^2} & 0 \\ 0 & \Psi_{\langle \xi \rangle/\langle x \rangle^2} \end{pmatrix} \vec{\alpha}, \vec{\alpha} \right>$$
But $\Psi_{\langle \xi \rangle/\langle x \rangle^2} = \Psi^\ast_{\langle \xi \rangle^{1/2}/\langle x \rangle} \circ \Psi_{\langle \xi \rangle^{1/2}/\langle x \rangle} + C$, $C$ of order $0$, and $$\left< \Psi^\ast_{\langle \xi \rangle^{1/2}/ \langle x \rangle} \circ \Psi_{\langle \xi \rangle^{1/2}/\langle x \rangle} f, f \right>= \|J^{1/2}f\|_{L^2(dx/\langle x \rangle^2)}^2$$  

Gathering all the terms, we obtain
$$\frac{d}{dt}\langle \vec{\alpha}, \vec{\alpha} \rangle + \epsilon \|\Delta \vec{\alpha} \|_{L^2}^2 + \|J^{1/2}\vec{\alpha} \|_{L^2(dx/\langle x \rangle^2)}^2 \leq C \|\vec{\alpha}\|_{L^2}^2 + 2 |\langle \vec{G}, \vec{\alpha} \rangle|$$
To obtain the first bound in Theorem \ref{T:1B}, we use $|\langle \vec{G}, \vec{\alpha} \rangle | \leq \|\vec{G}\|_{L^2}\|\vec{\alpha}\|_{L^2}$ while for the second one, we use 
$$|\langle \vec{G}, \vec{\alpha} \rangle| \leq \|J^{-1/2}\vec{G}\|_{L^2(\langle x \rangle^2dx)} \cdot \|J^{1/2} \vec{\alpha} \|_{L^2(dx/\langle x \rangle^2)}$$
Let us complete the proof of the first estimate: Fix $T<T_0$, and, for $t<T$, consider
\begin{align*}
&\|\vec{\alpha}(t) \|_{L^2}^2 + \epsilon \int_0^t \|\Delta \vec{\alpha} \|_{L^2}^2 + \int_0^t \|J^{1/2}\vec{\alpha} \|_{L^2(dx/\langle x \rangle^2)}^2\\
&= \|\vec{\alpha}(0)\|_{L^2}^2 + \int_0^t \left[ \frac{\partial}{\partial t} \|\vec{\alpha} \|_{L^2}^2 +\epsilon \|\Delta \vec{\alpha} \|_{L^2}^2 + \|J^{1/2}\vec{\alpha} \|_{L^2(dx / \langle x \rangle^2)}^2 \right] \, dt\\
&\leq \|\vec{\alpha}(0)\|_{L^2}^2 + \int_0^t \left[ C \|\vec{\alpha} \|_{L^2}^2 + 2\|\vec{G}\|_{L_x^2}\|\vec{\alpha}\|_{L_x^2}\right] \, dt \\
&\leq \|\vec{\alpha}(0)\|_{L^2}^2 + CT \sup_{0<t<T} \|\vec{\alpha}\|_{L^2}^2 + 2\|\vec{G}\|_{L_T^1L_x^2} \sup_{0<t<T} \|\vec{\alpha}\|_{L^2}
\end{align*}
If $CT<\frac{1}{2}$, we are done.
\begin{remark} \label{R:51}
Note that the above proof carries over, almost verbatim, if $(\text{H}_{5,l})$ is replaced by $(\text{H}_{5,l})'$: Let $A(x)=a_{kl}(x,0)$.  Then, $A(x)=A_0(x)+\eta A_1(x)$, where $A_0$ verifies $(\text{H}_{1,l})-(\text{H}_{5,l})$, and $|A_1(x)|\leq \frac{B_1}{\langle x \rangle^2}$, $|\nabla A_1(x)|\leq \frac{B_1}{\langle x \rangle^2}$, where $0\leq \eta\leq \eta_0$, for $\eta_0$ small enough depending on the constants in Doi's Lemma for $A_0$, on $C_1$, and on $B_1$.
\end{remark}
\begin{remark} \label{R:52}
The first order terms $\vec{b}_1(x,t) \cdot \nabla u$, $\vec{b}_2(x,t) \cdot \nabla \bar{u}$ can be replaced by $\Psi$DO $B_i$ of order $1$, depending in a $C^1$ fashion on $t$, and whose symbols $b_1(x,t,\xi)$, $b_2(x,t,\xi)$ verify estimates like those in $(\text{H}_{2,l})$, and $\text{Re }b_1(x,t,\xi)$ verifies estimates like those in $(\text{H}_{4,l})$.  
\end{remark}
Remark \ref{R:51} follows because the $p$ in Doi's Lemma that works for $A_0$, for small $\eta_0$, will work for $A$, and the proof is then identical.  Remark \ref{R:52} follows by using the same proof.

To complete our proof, we sketch the proof of Doi's Lemma.  Thus, $h(x,\xi)=a_{kl}(x,0)\xi_k\xi_l$ is assumed to verify $(\text{H}_{1,l})$ (ellipticity), $(\text{H}_{2,l})$ (regularity), $(\text{H}_{3,l})$ (asymptotic flatness), and $(\text{H}_{5,l})$, the non-trapping condition for the associated Hamiltonian vector field, i.e.\ the solutions of
$$
\left\{
\begin{aligned}
\frac{d}{ds} X_j(s; x_0, \xi_0) &= \partial_{\xi_j}h(X,\Xi) \\
\frac{d}{ds} \Xi_j(s;x_0,\xi_0) &= -\partial_{x_j}h(X,\Xi) \\
X(0; x_0, \xi_0) &= x_0 \\
\Xi(0;x_0,\xi_0) &=\xi_0
\end{aligned}
\right.
$$
have the property that $|X(s;x_0,\xi_0)| \to \infty$ as $s\to \infty$, $\forall \; x_0, \; \xi_0\neq 0$.  Recall also the homogeneity properties
\begin{align*}
X(s; x_0, r\xi_0) &= X(rs_0; x_0, \xi_0) \\
\Xi(s; x_0, r\xi_0) &= r\Xi(rs; x_0, \xi_0)
\end{align*}
We have to find $p(x,\xi)\in S^0$, real valued, such that 
$$H_hp\geq \frac{B|\xi|}{\langle x \rangle^2} - \frac{1}{B},  \quad \forall \; (x,\xi)\in \mathbb{R}^n\times \mathbb{R}^n$$
The first reduction is the following: Assume $\exists \; q$ such that $$|\partial_x^\alpha \partial_\xi^\beta q(x,\xi)| \leq C_{\alpha \beta} \langle x \rangle \langle \xi \rangle^{-|\beta|}$$ and $$H_hq \geq B_1|\xi|-B_2, \quad \forall \; x,\xi$$  Then we can construct a $p$ as we desire: let $K\geq 1$ be such that $|q(x,\xi)| \leq K\langle x \rangle$.  Define now $f(t)=2K^2 \int_0^t \frac{ds}{\langle s \rangle^2}$, so that $f'(|q|) \geq \frac{1}{\langle x \rangle^2}$, for all $(x,\xi)\in \mathbb{R}^n\times \mathbb{R}^n$.  Let $\phi(t)\in C^\infty(\mathbb{R})$ be such that $\phi(t)=0$ if $t\leq 1$ and $\phi(t)=1$ if $t\geq 2$, $\phi'(t)\geq 0$.  Let $\phi_+(t)=\phi(\frac{t}{\epsilon})$, $\phi_-(t)=\phi_+(-t)$, $\phi_0=1-\phi_+-\phi_-$. Define $\Psi_0$, $\Psi_\pm\in S^0$ by $\Psi_0=\phi_0(\frac{q}{\langle x \rangle})$, $\Psi_\pm = \phi_\pm(\frac{q}{\langle x \rangle})$.  By our construction of $f$, $|\partial_x^\beta \partial_\xi^\alpha f(|q(x,\xi)|)|\leq C_{\alpha \beta} \langle \xi \rangle^{-|\alpha|}$ on $\text{supp }\Psi_+\cup\text{ supp }\Psi_-$, given the estimates on $q$.  We now put $$p=\frac{q}{\langle x \rangle}\Psi_0 + [ f(|q|)+2\epsilon][\Psi_+-\Psi_-]\in S^0$$ and check that, for $\epsilon$ small, it has the desired estimates: In fact, on support of $\Psi_0$ (i.e.\ $|q|\leq \epsilon \langle x \rangle$), for $\epsilon$ small enough,
$$H_h\Bigl(\frac{q}{\langle x \rangle}\Bigr) = \frac{H_h q}{\langle x \rangle} - q \frac{x}{\langle x \rangle} \cdot \frac{\nabla_\xi h(x,\xi)}{\langle x \rangle^2} \geq \tilde{B}_1 \frac{|\xi|}{\langle x \rangle} - \tilde{B}_2$$
Fix such $\epsilon$, then
\begin{equation*}
H_h(p)= \begin{aligned}[t]
&H_h\Bigl(\frac{q}{\langle x \rangle}\Bigr) \Psi_0 + \frac{q}{\langle x \rangle} \phi_0'\Bigl(\frac{q}{\langle x \rangle} \Bigr) H_h\Bigl(\frac{q}{\langle x \rangle}\Bigr) + f'(|q|) H_h(|q|)(\Psi_+-\Psi_-) \\
&+ [ f(|q|)+2\epsilon]\left\{ \phi_+'\Bigl( \frac{q}{\langle x \rangle} \Bigr)-\phi_-'\Bigl( \frac{q}{\langle x \rangle} \Bigr) \right\} H_h\Bigl(\frac{q}{\langle x \rangle }\Bigr)
\end{aligned}
\end{equation*}
Note that $\phi_0'=-\phi_+'-\phi_-'$ and
$$H_h(|q|)\{ \Psi_+-\Psi_- \}= (\text{sign }q)H_h(q)\{ \Psi_+-\Psi_- \} = H_h(q)\{ \Psi_+ + \Psi_- \}$$
 so that 
\begin{equation*}
H_h(p)= 
\begin{aligned}[t]
&H_h\Bigl(\frac{q}{\langle x \rangle}\Bigr) \Psi_0 + f'(|q|)H_h(q)\{ \Psi_+ + \Psi_- \} \\
&+ \left(f(|q|) + 2\epsilon - \frac{|q|}{\langle x \rangle}\right) \left\{ \phi_+'\Bigl(\frac{q}{\langle x \rangle}\Bigr) - \phi_-'\Bigl(\frac{q}{\langle x \rangle}\Bigr) \right\}H_h\Bigl(\frac{q}{\langle x \rangle}\Bigr)
\end{aligned}
\end{equation*}
Note that $\phi_+'-\phi_-'\geq 0$, and on $\text{supp }\phi_+'( \frac{q}{\langle x \rangle} ) \cup \text{ supp } \phi_-'( \frac{q}{\langle x \rangle} ) $, we have 
$$f(|q|)+2\epsilon-\frac{|q|}{\langle x \rangle} \geq 0$$ and 
$$H_h\Bigl(\frac{q}{\langle x \rangle}\Bigr)\geq \tilde{B}_1 \frac{|\xi|}{|x|} - \tilde{B}_2$$
Thus, we get a lower bound of 
$$\tilde{B}_1 \frac{|\xi|}{\langle x \rangle }\Psi_0 + \frac{B_1|\xi|}{\langle x \rangle^2}\{ \Psi_+ + \Psi_- \} - \tilde{B}_2$$
 which gives the desired bound.  

We now turn to the construction of $q$:  Recall $h(x,\xi)=a_{kl}(x)\xi_k\xi_l = \left< A(x)\xi, \xi \right>$, where $A(x)=(a_{kl}(x))$ and that 
$$H_h\phi = \sum_{j=1}^n \partial_{\xi_j}h(x,\xi)\partial_{x_j}\phi - \partial_{x_j}h(x,\xi) \partial_{\xi_j}\phi$$
For $M$ large to be chosen, let $\psi\in C^\infty(\mathbb{R})$, $\psi\equiv 0$ for $t\leq M^2$, $\psi(t)=1$ for $t\geq (M+1)^2$, $\psi'\geq 0$.  Let 
$$q_1(x,\xi)=\langle \xi \rangle^{-1} \psi(|x|^2) H_h(|x|^2) = -4 \langle \xi \rangle^{-1} \psi(|x|^2) \langle A(x)\xi, x \rangle$$
By calculation, we have 
\begin{align*}
H_hq_1 = &\langle \xi \rangle^{-1}\psi'(|x|^2) H_h(|x|^2)^2 \\
&+ \langle \xi \rangle^{-1}\psi(|x|^2)
\begin{aligned}[t]
\Big[ & 8|A(x)\xi|^2+8\sum_{j,k,l,m}x_l\partial_{x_j}a_{lm}(x)a_{jk}(x)\xi_k\xi_m \\
& -4\sum_{j,k,l,m,p} x_ja_{jk}(x)\partial_{x_l}a_{mp}(x)\frac{\xi_k\xi_l\xi_m\xi_p}{\langle \xi \rangle^2} \Big]
\end{aligned}
\end{align*}
Because of our assumptions on $A(x)$, one can fix $M$ sufficiently large so that
$$H_hq_1 \geq c \psi(|x|^2)\frac{|\xi|^2}{\langle \xi \rangle} \quad \; x,\xi \in \mathbb{R}^n$$
Next we will use the following fact about the non-trapping condition, in the presence of asymptotic flatness:

\begin{lemma} Let $K\subset \mathbb{R}^n\times \mathbb{R}^n\backslash \{ 0 \}$ be a compact set.  Then, given $\mu >0$, $\exists \; s_0=s_0(K,\mu)$ such that $|X(s;x_0,\xi_0)|\geq \mu$, $\forall \; |s|\geq |s_0|$, $\forall \; (x_0,\xi_0) \in K$.  
\end{lemma}
\noindent Now choose $\phi_1\in C_0^\infty(\mathbb{R}^n)$, $\phi_1\equiv 1$ on $|x|<M+1$.  For $\xi\neq 0$, let 
$$q_2(x,\xi) = - \int_0^\infty \phi_1(X(s;x,\xi)) \langle \Xi(s; x, \xi) \rangle \, ds$$
Note that by the lemma, for each $(x_0,\xi_0)\in \mathbb{R}^n\times \mathbb{R}^n\backslash \{ 0 \}$, there is a neighborhood $U$ of $(x_0,\xi_0)\in \mathbb{R}^n\times \mathbb{R}^n\backslash \{ 0 \}$, such that $\forall (x,\xi) \in U$, the integral defining $q_2$ is taken over a fixed compact interval of $s$, and hence $q_2$ is smooth.  Furthermore, by homogeneity of the flow, and a change of variables, 
$$q_2(x,\xi)=-|\xi|^{-1}\int_0^\infty \phi_1(X(s;x,\tfrac{\xi}{|\xi|})) \langle |\xi| \Xi(s;x,\tfrac{\xi}{|\xi|}) \rangle \, ds$$
Choose now $\phi_2\in C^\infty(\mathbb{R}^n)$, $\phi_2\equiv 0$ for $|\xi|\leq 1$, $\phi_2\equiv 1$ for $|\xi|\geq 2$.  Let $q_3(x,\xi)=\phi_1(x)\phi_2(\xi)q_2(x,\xi)$, for $(x,\xi)\in \mathbb{R}^n$.  Then $q_3\in S^0$, and 
\begin{align*}
H_hq_3(x,\xi) = & \left[ 2 \sum a_{jk}(x)\xi_k \partial_{x_j}\phi_1(x) \right] \phi_2(\xi)q_2(x,\xi) \\
&+\phi_1(x)H_h\phi_2(\xi)q_2(x,\xi)+\phi_1(x)^2 \phi_2(\xi)\langle \xi \rangle
\end{align*}
We now let $q(x,\xi)=Nq_1(x,\xi)+q_3(x,\xi)$, for $N$ large.  Then $|\partial_x^\alpha\partial_\xi^\beta q(x,\xi)|\leq C_{\alpha,\beta}\langle x \rangle \langle \xi \rangle^{-|\beta|}$, and $H_hq(x,\xi) \geq c|\xi|-d$, $\forall \; x,\xi\in\mathbb{R}^n$.  This finishes the proof of Doi's Lemma.  

Finally, we turn to the general case, i.e.
$$\left\{
\begin{aligned}
\partial_t u &= -\epsilon \Delta^2 u + i a_{lk}(x,t) \partial_{x_lx_k}^2 u + ib_{lk}(x,t)\partial_{x_l x_k}^2 \bar{u} + \vec{b}_1(x,t) \cdot \nabla u + \vec{b}_2(x,t) \cdot \nabla \bar{u} \\
& \quad + c_1(x,t)u + c_2(x,t) \bar{u} + f(x,t) \\
u\big|_{t=0} &= u_0
\end{aligned}
\right.
$$
where the $a_{kl}$ are real valued, and all the other coefficients may be complex.  Our assumptions are:

\noindent $(\text{H}_{1,\text{gl}})$ \underline{Ellipticity}. There exists $\gamma >0$, so that $a_{lk}(x,0)\xi_l\xi_k - |b_{lk}(x,0)\xi_l\xi_k|\geq \gamma |\xi|^2$.  (Thus, $a_{lk}(x,0)\xi_l\xi_k \geq \gamma |\xi|^2$, and
$$h(x,\xi) = \sqrt{ [a_{lk}(x,0)\xi_l\xi_k ]^2 - |b_{lk}(x,0) \xi_l\xi_k|^2}$$
verifies $h(x,\xi)\geq \gamma |\xi|^2$).

\noindent $(\text{H}_{2,\text{gl}})$ \underline{Regularity}.  $a_{lk}$, $b_{lk}$, $\vec{b}_1$, $\vec{b}_2$, $c_1$, $c_2$, verify $(\text{H}_{2,\text{l}})$.

\noindent $(\text{H}_{3,\text{gl}})$ \underline{Asymptotic flatness}. Both $a_{lk}$, $b_{lk}$ verify the ``asymptotic flatness'' assumption $(\text{H}_{3,\text{l}})$. 

\noindent $(\text{H}_{4,\text{gl}})$ \underline{Growth of the 1st order coefficients}.
\begin{align*}
|\partial_t\vec{b}_1(x,t)| \leq \frac{C}{\langle x \rangle^2} && |\partial_t \vec{b}_2(x,t)| \leq \frac{C}{\langle x \rangle^2} \\
|\vec{b}_1(x,0)| \leq \frac{C_1}{\langle x \rangle^2} && |\vec{b}_2(x,0)| \leq \frac{C_1}{\langle x \rangle^2}
\end{align*}

\noindent $(\text{H}_{5,\text{gl}})$ \underline{Approximate non-trapping}. The function $h(x,\xi)$, defined in $(\text{H}_{1,\text{gl}})$, which is real valued, homogeneous of degree $2$, and elliptic, is ``approximately non-trapping'', i.e.\ we can write $h(x,\xi) = a(x,\xi) + \eta a_1(x,\xi)$, $0\leq \eta \leq \eta_0$, where $a(x,\xi)$ is real, homogeneous of degree $2$, with $\partial_x^\beta a(x,\xi)\in C^{1,1}(\mathbb{R}^n\times \mathbb{R}^n)$, $|\beta| \leq N(n)$, with norm bounded by $C_1$, and $a(x,\xi)\in C^{N(n)}(\mathbb{R}^n\times \mathbb{R}^n\backslash \{ |\xi| < 1 \} )$, with norm bounded by $C_1$, and with $a_1$ verifying similar properties and estimates, and in addition
$$|a_1(x,\xi)| + |\nabla_xa_1(x,\xi)| \leq C_1 \frac{|\xi|^2}{\langle x \rangle^2}$$
and the Hamiltonian flow $H_a$, is non-trapping for each $(x_0,\xi_0)$, $\xi_0\neq 0$, where $\eta_0$ is small enough, depending only on $\gamma$, $C_1$, and the non-trapping character of $a$ (in terms of the bounds in the analog of Doi's lemma for $H_a$)

We then have:
\begin{theorem} \label{T:2B}
There exist $N=N(n)$, $T_0>0$, with $T_0$ depending on $C$, $C_1$, $(\text{H}_{5, \text{gl}})$ and $A>0$, depending only on $\gamma, C_1$ and $(\text{H}_{5,\text{gl}})$, so that, for any $T\in [0,T_0]$, $\epsilon \in [0,1]$, we have, for any solution $u$
\begin{align*}
\hspace{.5in}&\hspace{-.5in} \sup_{0<t<T} \|u(t)\|_{L^2} + \left( \int_0^T \int \langle x \rangle^{-2} |J^{1/2} u |^2 \, dxdt \right)^{1/2} \\
&\leq A 
\begin{cases}
\|u_0\|_{L^2} + \|f\|_{L_T^1L_x^2} \\
\|u_0\|_{L^2} + \|J^{-1/2}f \|_{L^2(\langle x \rangle^2 dxdt)}
\end{cases}
\end{align*}
\end{theorem}
The explicit dependence on $(\text{H}_{5,\text{gl}})$ is through the constants in Doi's Lemma for $H_a$.  Next time we will sketch the proof of Theorem \ref{T:2B}.

\section*{Problems for Lecture \arabic{chapter}}
\begin{problems}
\item Prove the commutator estimate
$$\| J^s(fg)-fJ^sg \|_{L^2} \leq C \|g\|_{L^\infty} \|J^s f \|_{L^2}$$
for $0<s<1$.
\item Formulate the analog of Theorem \ref{T:1B}, with data in $H^s$, $s\geq 0$.  Carry out the proof, also showing that the interval of existence can be taken to be independent of $s$, for $s>0$.
\item Verify Remark \ref{R:51}, \ref{R:52}.
\item \label{P:5.4}Check that Doi's Lemma still works for $h(x,\xi)$ real valued, elliptic, homogeneous of degree $2$ in $\xi$, with $\partial_x^\alpha h(x,\xi)\in C^{1,1}(\mathbb{R}^n\times \mathbb{R}^n)$, $|\alpha|\leq N(n)$, and $\partial_x^\alpha \partial_\xi^\beta h \in C^{N(n)}(\mathbb{R}^n\times \mathbb{R}^n \backslash |\xi|<1)$.  In fact, show then that if $H_h$ is non-trapping, and $\theta\in C^\infty$, $\theta\equiv 0$ for $|\xi|\leq 1$, $\theta \equiv 1$ for $|\xi|\geq 2$, we can construct $p\in S^0$ so that 
$$H_{\theta h} p \geq \frac{B|\xi|}{\langle x \rangle^2} - \frac{1}{B}$$
(We also need to assume that $h$ is ``asymptotically flat''.)
\item Prove that if $n=1$, $(\text{H}_{1,\text{l}})-(\text{H}_{4,\text{l}})$ imply $(\text{H}_{5,\text{l}})$.
\end{problems}

\lecture{The quasilinear Schr\"{o}dinger equation}
We first start out by sketching the new ideas that are needed for the proof of Theorem \ref{T:2B}, and at the same time explain the ellipticity condition $(\text{H}_{1, \text{gl}})$.  The key extra ingredient is a ``symmetrization argument'' analogous to the one introduced by A.P. Calder\'{o}n in his work on symmetrizable hyperbolic systems \cite{Ca}.  We first write our equation as a system in $\begin{pmatrix} u \\ \bar{u} \end{pmatrix}=\vec{w}$.  This now becomes
$$\left\{ 
\begin{aligned}
\partial_t \vec{w} &= -\epsilon \Delta^2I\vec{w} + (iH_0+B+C)\vec{w}+\vec{f}\\
\vec{w}\big|_{t=0} &= \vec{w}_0
\end{aligned}
\right.
$$
where now $H_0 = \begin{pmatrix} \mathcal{L} & \mathcal{L}_B \\ -\mathcal{L}_{\bar{B}} & -\mathcal{L} \end{pmatrix}$, where $\mathcal{L}_B=b_{kl}(x,t)\partial_{x_kx_l}^2$, $\mathcal{L}_{\bar{B}}= \bar{b}_{kl}(x,t)\partial_{x_kx_l}^2$.  The first step is to diagonalize $H_0$.  It is in this step that the ellipticity hypothesis $(\text{H}_{1,\text{gl}})$ appears.  The symbol of $H_0$ is the matrix $$M= \begin{pmatrix} -a_{kl}(x,t)\xi_k\xi_l & -b_{kl}(x,t)\xi_k\xi_l \\ \bar{b}_{kl}(x,t)\xi_k\xi_l & a_{kl}(x,\xi) \xi_k \xi_l \end{pmatrix}$$ whose eigenvalues are the roots of $(\lambda + a_{lk}(x,t)\xi_k\xi_l)(\lambda-a_{lk}(x,t)\xi_k\xi_l)+|b_{kl}(x,t)\xi_k\xi_l|^2$ or $\lambda^2-[a_{lk}(x,t)\xi_k\xi_l]^2+|b_{kl}(x,t)\xi_k\xi_l|^2=0$, i.e.\ the eigenvalues are $\lambda_{\pm}(x,t,\xi)=\pm h(x,t,\xi)$, which explains our ellipticity hypothesis, since $\lambda_+(x,t,\xi)\geq \gamma|\xi|^2$, and $\lambda_+$ is real valued, homogeneous of degree 2.  A computation shows that, if we define
$$S=\frac{1}{\gamma |\xi|^2} \begin{pmatrix} \overline{b_{kl}(x,t)}\xi_k\xi_l & \lambda_++a_{kl}(x,t)\xi_k\xi_l \\ -\lambda_+-a_{kl}(x,t)\xi_k\xi_l & -b_{kl}(x,t)\xi_k\xi_l \end{pmatrix}$$
then we have
$$SM= \begin{pmatrix} \lambda_+ & 0 \\ 0 & -\lambda_+ \end{pmatrix}S$$
Note also that $\det S = \frac{2\lambda_+^2+2\lambda_+a}{\gamma^2|\xi|^4}$, and since $\lambda_+\geq \gamma|\xi|^2$, $a_{kl}(x,t)\xi_k\xi_l \geq \gamma|\xi|^2$, one has $\det S \geq 4$.  Moreover, $S$ is homogeneous of degree $0$, and thus defines a 0th order invertible $\Psi$DO.  (We must multiply $S$ by $\theta_R(\xi)$, $R$ large, but this is a technical detail).  We then rewrite our system in $\vec{z}=S\vec{w}$, which now gives
$$
\left\{
\begin{aligned}
\partial_t \vec{z} &= -\epsilon \Delta^2 I \vec{z} + \epsilon E_3 \vec{z} + i \begin{pmatrix} \Psi_{-\lambda_+} & 0 \\ 0 & \Psi_{\lambda_+} \end{pmatrix} \vec{z} + \tilde{B}\vec{z} + C\vec{z} + \vec{F} \\
z\big|_{t=0} &= \vec{z}_0
\end{aligned}
\right.
$$
where $E_3$ is of order $3$.  The strengthened decay assumptions $(\text{H}_{3,\text{gl}})$ guarantee that the entries of $\tilde{B}$ (the order $1$ part) still have decay.  Moreover, $\lambda_+$ is real valued, elliptic, homogeneous of degree $2$, but now pseudo-differential, but the proof proceeds exactly as before, using the analog of Doi's lemma for pseudo-differential operators $h$.  (This was actually carried out by Chihara \cite{Chi}.)  The proof then proceeds as before.  

We now finish the course by turning to the application to  quasi-linear problems.  We are thus going to study equations of the form
\begin{equation}\label{E:QLCP2}
\text{(QLCP)} \left\{ 
\begin{aligned}
\partial_t u &= ia_{lk}(x,t;u,\bar{u},\nabla_xu,\nabla_x\bar{u}) \partial_{x_lx_k}^2u \\
& \quad + ib_{lk}(x,t;u,\bar{u},\nabla_xu,\nabla_x\bar{u}) \partial_{x_lx_k}^2\bar{u} \\
& \quad + \vec{b}_1(x,t;u,\bar{u},\nabla_xu,\nabla_x\bar{u})\cdot \nabla_xu \\
&\quad + \vec{b}_2(x,t;u,\bar{u},\nabla_xu,\nabla_x\bar{u})\cdot \nabla_x\bar{u} \\
& \quad + c_1(x,t;u,\bar{u})u+c_2(x,t;u,\bar{u})\bar{u} \\
& \quad + f(x,t) \\
u\big|_{t=0} &= u_0 
\end{aligned}
\right. \qquad  x\in \mathbb{R}^n, \; t\in [0,T]
\end{equation}
We saw already that when $a_{lk}=\Delta$, $b_{lk}\equiv 0$, $\vec{b}_1\equiv 0$, $\vec{b}_2\equiv 0$, $f\equiv 0$, this is locally well-posed in $H^s(\mathbb{R}^n)$, $s>n/2$, and that when $a_{lk}=\Delta$, $b_{lk}\equiv 0$, $\vec{b}_i$ are independent of $(x,t)$, and $\vec{b}_i=O(|u|^2)$, we have local well-posedness in $H^s(\mathbb{R}^n)$, $s$ large, and when $\vec{b}_i=O(|u|)$, we have local well-posedness in $H^s(\mathbb{R}^n)\cap L^2(|x|^{2N}dx)$, $s$ large, $N$ large, by Picard iteration, and thus, the flow map is real analytic.  We also saw that (Molinet-Saut-Tzvetkov \cite{MoSaTz}) for $n=1$, $\partial_t u =i\partial_x^2 u + u\partial_x u$, the flow map is not $C^2$ for any $H^s$, and hence we cannot have solvability by Picard iteration.  We also have seen that for $b_{lk}\equiv 0$, $\vec{b}_2\equiv 0$, $c_1\equiv c_2 \equiv 0$, and $a_{lk}$ elliptic, independent of $t$, $u$, $\bar{u}$, $\nabla_x u$, $\nabla_x \bar{u}$, and $\vec{b}_1=\vec{b}_1(x)\in C_0^\infty$, the ``non-trapping'' condition is ``necessary'' (Ichinose \cite{Ic}).  Moreover, Doi proved its necessity for the ``local smoothing'' effect \cite{Do3}.  In the context of non-linear problems, its relevance can be seen, because, one of the worst forms of its failure, i.e.\ periodic orbits, yields ill-posedness.  In fact, Chihara \cite{Chi02} has shown that for semi-linear problems
$$\left\{
\begin{aligned}
\partial_t u &= i\Delta u + \text{div } \vec{G}(u) \\
u\big|_{t=0} &= u_0
\end{aligned}
\right. \qquad x\in \mathbb{T}^n, \; t\in [0,T]
$$
where $\vec{G}\nequiv 0$, and $\vec{G}=(G_1, \ldots, G_n)$, $G_i$ holomorphic, we have ill-posedness in any Sobolev space $H^s(\mathbb{T}^n)$.  

The first general results on quasilinear problems were obtained in $n=1$.  In this case (QLCP) takes the form
$$
\left\{ 
\begin{aligned}
\partial_t u &= ia(x,t;u,\bar{u},\nabla_xu,\nabla_x\bar{u}) \partial_x^2u \\
& \quad + ib(x,t;u,\bar{u},\nabla_xu,\nabla_x\bar{u}) \partial_x^2\bar{u} \\
& \quad + b_1(x,t;u,\bar{u},\nabla_xu,\nabla_x\bar{u}) \partial_x u \\
&\quad + b_2(x,t;u,\bar{u},\nabla_xu,\nabla_x\bar{u}) \partial_x\bar{u} \\
& \quad + c_1(x,t;u,\bar{u})u+c_2(x,t;u,\bar{u})\bar{u} \\
& \quad + f(x,t) \\
u\big|_{t=0} &= u_0 
\end{aligned}
\right. \qquad  x\in \mathbb{R}, \; t\in [0,T]
$$
For coefficients independent of $(x,t)$, such problems were studied by Poppenberg \cite{Po}, who showed that, under ellipticity:
$(\text{H}1)$.  $a$ is real valued, and for $|(z_1, z_2, z_3, z_4)|\leq R$, there exists $\lambda(R)>0$ such that $$a(z_1, z_2, z_3, z_4)-|b(z_1,z_2,z_3,z_4)|\geq \lambda(R)$$ and if $\frac{\partial a}{\partial z}(0,0,0,0)=\frac{\partial b}{\partial z}(0,0,0,0) = 0$, and $b_1$, $b_2$ vanish quadratically at $(0,0,0,0)$, then the above problem is locally well-posed in $H^\infty(\mathbb{R})= \cap_{s\geq 0} H^s(\mathbb{R})$, using the Nash-Moser iteration scheme.  In \cite{LiPo}, Lim and Ponce showed, in the $(x,t)$ dependent setting, that, under Poppenberg's hypothesis, one has local well-posedness in $H^{s_0}(\mathbb{R})$, $s_0$ large, and if $b_1$, $b_2$ vanish linearly or $\frac{\partial a}{\partial z}\neq 0$, or $\frac{\partial b}{\partial z} \neq 0$, this holds in $H^{s_0}(\mathbb{R})\cap L^2(|x|^{m_0}dx)$.  To clarify the ellipticity condition $(\text{H}1)$, note that when $b\equiv 0$, this is the usual condition, and in general it says that $\partial_x^2 u$ ``dominates'' $\partial_x^2 \bar{u}$.  This is certainly needed.  For example, the problem
$$\left\{
\begin{aligned}
\partial_t u &= i\Delta \bar{u} \\
u\big|_{t=0}
\end{aligned}
\right.
$$
 is the backward heat equation in disguise, and hence, it is ill-posed on any Sobolev space.  

We now turn to our positive results on (QLCP), for $n>1$.  The coefficients $a_{lk}$ are real valued, all others may be complex valued.  We let $\vec{z}=(z_1,z_2,\vec{z}_1, \vec{z}_2)=(u,\bar{u},\nabla_xu, \nabla_x\bar{u})$.  Our assumptions are

\noindent $(\text{H}1)$ \underline{Ellipticity}.  Given $R>0$, there exist $\gamma_R>0$ such that $\left< a_{lk}(x,0;\vec{z})\xi, \xi \right> - | \left< b_{lk}(x,0;\vec{z})\xi,\xi \right> | \geq \gamma_R|\xi|^2$, for all $\xi\in \mathbb{R}^n$, $|\vec{z}|\leq R$.

\noindent $(\text{H}2)$ \underline{Regularity}. For any $N\in \mathbb{N}$, $R>0$, the coefficients $a_{lk}$, $b_{lk}$, $\vec{b}_1$, $\vec{b}_2$, $c_1$, $c_2\in C_b^N(\mathbb{R}^n\times \mathbb{R}^n\times |\vec{z}|\leq R)$.

\noindent $(\text{H}3)$ \underline{Asymptotic flatness}.  There exists $C>0$ such that $\forall \; (x,t) \in \mathbb{R}^n\times \mathbb{R}$, $|\partial_{x_j}a_{lk}(x,t;\vec{0})| + |\partial_t a_{lk}(x,t; \vec{0})| + |\partial_{x_j}\partial_{x_r}a_{lk}(x,t; \vec{0})| + |\partial_t \partial_{x_j} a_{lk}(x,t; \vec{0})| \leq \frac{C}{\langle x \rangle^2}$, and similarly for $b_{lk}$.

\noindent $(\text{H}4)$ \underline{Growth of the first order coefficients}.  There exists $C$, $C_1>0$ such that, for $(x,t)\in \mathbb{R}^n\times \mathbb{R}$,
$$|\vec{b}_i(x,0;\vec{0})| \leq \frac{C_1}{\langle x \rangle^2}$$
$$|\partial_t \vec{b}_i(x,t;\vec{0})| \leq \frac{C}{\langle x \rangle^2}$$

\noindent $(\text{H}5)$ \underline{Approximate non-trapping}.  Fix an initial data $u_0\in H^r(\mathbb{R}^n)$, $r>\frac{n}{2}+2$, $r$ large.  Define 
\begin{gather*}
h_1(x,\xi)=a_{lk}(x,0;u_0,\bar{u}_0,\nabla_x u_0, \nabla_x \bar{u}_0)\xi_l\xi_k \\
h_2(x,\xi)=b_{lk}(x,0; u_0, \bar{u}_0, \nabla_x u_0, \nabla_x \bar{u}_0)\xi_l\xi_k \\
h(x,\xi)= \sqrt{ h_1^2(x,\xi)-|h_2|^2(x,\xi)}
\end{gather*}
  Note that by $(\text{H}1)$, $h(x,\xi) \geq \gamma |\xi|^2$, $\gamma=\gamma(\|u_0\|_{H^r})$ and that $h$ is positive, homogeneous of degree $2$.   Suppose that there exists $0\leq \eta <1$ such that $h(x,\xi)=a(x,\xi)+\eta a_1(x,\xi)$, where $a(x,\xi)$ is real, homogeneous of degree $2$, with $\partial_x^\beta a(x,\xi)\in C^{1,1}(\mathbb{R}^n\times \mathbb{R}^n)$, $|\beta|\leq N(n)$, $a(x,\xi)\in C^{N(n)}(\mathbb{R}^n\times \mathbb{R}^n \backslash \{ |\xi| < 1 \} )$, where $N(n)$ is as in Theorem \ref{T:2B}, with $a_1$ verifying similar estimates, and $|a_1(x,\xi)| + |\nabla_xa_1(x,\xi)| \leq \frac{C|\xi|^2}{\langle x \rangle}$, and the Hamiltonian $H_a$ is non-trapping, and $\eta\leq \eta_0$, where $\eta_0$ is as in Theorem \ref{T:2B}.

Then we have
\begin{theorem}
Under $(\text{H}1)-(\text{H}5)$, given $u_0\in H^s(\mathbb{R}^n)$, $\langle x \rangle^2 \partial_x^\alpha u_0\in L^2(\mathbb{R}^n)$, $|\alpha|\leq s_1$, and $f\in L^\infty(\mathbb{R}; H^s(\mathbb{R}^n))$, $\langle x \rangle^2\partial_x^\alpha f\in L^\infty(\mathbb{R}; L^2(\mathbb{R}^n))$, $|\alpha|\leq s_1$, where $s_1\geq \frac{n}{2}+7$, $s\geq \max \{ s_1+4, N(n)+n+3 \}$, where $N=N(n)$ is as in Theorem \ref{T:2B}, then there exists $T_0>0$, depending only on $(\text{H}1)-(\text{H}5)$, and on 
$$\lambda=\|u_0\|_{H^s}+\sum_{|\alpha|\leq s_1} \|\langle x \rangle^2 \partial_x^\alpha u_0 \|_{L^2} + \|f(t)\|_{L_t^\infty H^s} + \sum_{|\alpha|\leq s_1} \|\langle x \rangle^2 \partial_x^\alpha f \|_{L_t^\infty L^2_x}$$
 so that (QLCP) is locally well-posed in $(0,T_0)$, in the obvious space.
\end{theorem}
  Before sketching the proof of the theorem, we will make some remarks.
\begin{remark}
When $n=1$, it is not difficult to show that ellipticity implies non-trapping, and hence $(\text{H}5)$ is not needed.
\end{remark}
\begin{remark}
For $n>1$, if $h_{\vec{0}}(x,\xi)=\sqrt{ (a_{lk}(x,0;\vec{0})\xi_l\xi_k)^2-|b_{lk}(x,0;\vec{0})\xi_l\xi_k|^2}$  is non-trapping, then we obtain local well-posedness for small data, since $(\text{H}5)$ is automatic.  This holds for instance, if $h_{\vec{0}}(x,\xi)=h_{\vec{0}}(\xi)$ is independent of $x$.
\end{remark}
\begin{remark} \label{R:c}
The proof actually gives the ``local-smoothing'' estimate for the solution, namely $J^{s+\frac{1}{2}}u\in L^2(\mathbb{R}^n\times [0,T_0], \langle x \rangle^{-2} \, dx dt)$
\end{remark}
\begin{remark} $\langle x \rangle^2$ can be replaced by $\langle x \rangle^{1+\epsilon}$, $\epsilon>0$.
\end{remark}
\begin{remark}[Koch-Tataru \cite{KoTa}] The solution map is not $C^2$, and hence the result cannot be proved by Picard iteration.  
\end{remark}
The key step in the proof is the \textit{a priori} linear estimate, provided by Theorem \ref{T:2B}.  Let us now sketch the proof of the theorem, when $s$ and $s_1$ are assumed to be even integers, $f\equiv 0$, $s_1\geq \frac{n}{2}+7$, $s\geq \max\{ s_1+4, N(n)+n+3 \}$.  We first consider the non-linear parabolic IVP
\begin{equation*}
(\text{IVP})_\epsilon \left\{
\begin{aligned}
\partial_t u &= -\epsilon \Delta^2 u + \mathcal{L}(u)u \\
u(x,0) &= u_0(x)
\end{aligned}
\right. \quad \epsilon\in (0,1)
\end{equation*}
where 
\begin{equation*}
\begin{split}
\mathcal{L}(u)(v) &= ia_{lk}(x,t; u,\bar{u},\nabla_x u, \nabla_x \bar{u})\partial_{x_kx_l}^2 v \\
& \quad + ib_{lk}(x,t; u, \bar{u}, \nabla_x u, \nabla_x \bar{u}) \partial_{x_lx_k}^2 \bar{v} \\
& \quad +\vec{b}_1(x,t; u, \bar{u}, \nabla_x u, \nabla_x \bar{u})\nabla v \\
& \quad +\vec{b}_2(x,t; u, \bar{u}, \nabla_x u, \nabla_x \bar{u})\nabla \bar{v} \\
& \quad +c_1(x,t; u,\bar{u})v + c_2(x,t;u,\bar{u})\bar{v} 
\end{split}
\end{equation*}
Let $\lambda = \tv u_0 \tv = \| u_0 \|_{H^s} + \sum_{|\alpha|\leq s_1} \|\langle x \rangle^2 \partial_x^\alpha u_0 \|_{L^2}$, where $s>s_1$.  We also consider, for $T>0$, $M_0>0$, 
\begin{equation*}
\begin{split}
X_{T,M_0} = \{ \; & v: \mathbb{R}^n\times [0,T]\to \mathbb{C}, \; v\in C([0,T]; H^s) , \;  \langle x \rangle^2 \partial_x^\alpha v \in C([0,T],L^2), \\
& |\alpha|\leq s_1, \; v(0)=v_0, \; \tv v \tv_T \leq M_0 \; \}
\end{split}
\end{equation*}
where $\tv v \tv_T = \sup_{0<t<T} \tv v(-,t) \tv$.  One then uses standard parabolic theory, using Picard iteration, to show: If $\tv v_0 \tv \leq \frac{M_0}{2}$, $\exists \; T_\epsilon = O(\epsilon)$, and a unique solution $v^\epsilon$ to $(\text{IVP})_\epsilon$ in $X_{T_\epsilon, M_0}$.  We in fact consider the integral equation version of $(\text{IVP})_\epsilon$, namely
$$v(t) = e^{-\epsilon t \Delta^2}v_0 + \int_0^t e^{-\epsilon(t-t')\Delta^2} [ \mathcal{L}(v) v ] \,dt'$$
and define 
$$\Gamma(w)(t) = e^{-\epsilon t \Delta^2}v_0 + \int_0^t e^{-\epsilon (t-t') \Delta^2}[\mathcal{L}(w)w] \, dt'$$
We then show that for appropriate $T_\epsilon=O(\epsilon)$, $\Gamma$ is a contraction on $X_{T_\epsilon, M_0}$.  For this one only needs to use the estimate
$\| \Delta e^{-\epsilon t \Delta^2} g \|_{L^2} \leq \frac{1}{\epsilon^{1/2}t^{1/2}} \|g\|_{L^2}$, and to deal with the weights, the identities 
$$x_r \Delta^2\Gamma(w)=\Delta^2(x_r\Gamma w) - 2 \Delta \partial_{x_r}\Gamma w$$
$$ x_ra_{lk}\partial_{x_lx_k}^2 w = a_{lk} \partial_{x_lx_k}^2(x_r w) - 2 a_{lr}\partial_{x_l}w$$
$$x_r \vec{b}_1 \cdot \nabla w = \vec{b}_1 \cdot \nabla (x_r w) - b_{1,r} w$$
$$x_r \vec{b}_2 \cdot \nabla \bar{w} = \vec{b}_2 \nabla (\overline{x_r w}) - b_{2,r}\bar{w}$$
$T_\epsilon=O(\epsilon)$ depends on $M_0$, $(\text{H}1)-(\text{H}4)$.  We next fix $M_0=20A\lambda$, where $A$ is the constant in Theorem \ref{T:2B}, for linear operators satisfying $(\text{H}_{1,\text{gl}})-(\text{H}_{5,\text{gl}})$, with a fixed $C_1$ and a fixed ``non-trapping'' character, which will depend only on $u_0$, $s$, $s_1$, $(\text{H}1)-(\text{H}5)$.  With such choices made, we will next show that there exist $T>0$, independent of $\epsilon$, and a solution to $(\text{IVP})_\epsilon$ in $(0,T)$, such that $\tv u^\epsilon \tv_T$ is uniformly bounded for $\epsilon\in (0,1)$.  The key claim is the following:

\begin{claim} If $u_\epsilon$ is a solution to $(\text{IVP})_\epsilon$, then in an interval $(0,T)$, and for which $\tv u_\epsilon \tv_T\leq M_0$, we have that there exists $\tilde{T}\in (0,T)$, independent of $\epsilon$, so that, for a fixed increasing function $R$, which depends only on $(\text{H}1)-(\text{H}4)$, $s$, $s_1$, $n$, we have
\begin{equation} \label{E:20}
\tv u^\epsilon \tv_{\tilde{T}} \leq A[ \lambda+\tilde{T}R(M_0)]
\end{equation}
\end{claim}
Assume \eqref{E:20}, and choose $\tilde{T}$ in addition, so small that $A [ \lambda + \tilde{T}R(M_0)] \leq \frac{M_0}{4} = 5A\lambda$.  Then, we can apply the existence theorem for $(\text{IVP})_\epsilon$ in the interval $(0,T_\epsilon)$, and reapply it in $[T_\epsilon, 2T_\epsilon]$, since $\tv u^\epsilon(T_\epsilon) \tv \leq \frac{M_0}{4}\leq \frac{M_0}{2}$.  We then get $\tv u^\epsilon(2T_\epsilon) \tv \leq \frac{M_0}{4}$, and we can continue $k$ times as long as $kT_\epsilon \leq \tilde{T}$.  We then obtain a solution $u^\epsilon$, with $\tv u^\epsilon \tv_{\tilde{T}}\leq \tv u^\epsilon \tv_{kT_\epsilon +T_\epsilon}\leq M_0$, where $k$ is such that $kT_\epsilon \leq \tilde{T}< (k+1)T_\epsilon$.

In order to establish \eqref{E:20}, we first show that if $\tv u^\epsilon \tv_T \leq M_0 = 20A\lambda$, the coefficients of the linear equations verified by $J^{2m}u = (I-\Delta)^mu$, $2m=s$, and $|x|^2J^{2m_1}u$, $2m_1=s_1$, can be written so that the constants $C_1, \gamma$ appearing in $(\text{H}_{1, \text{gl}})-(\text{H}_{5,\text{gl}})$ depend only on $u_0$, $s$, $s_1$, $n$, $(\text{H}1)-(\text{H}5)$ (and hence determine $A$), and the constants $C$ in $(\text{H}_{1, \text{gl}})-(\text{H}_{5,\text{gl}})$, and the ``right-hand sides'' $f$, depend only on $M_0$, $(\text{H}1)-(\text{H}5)$, $s$, $s_1$, $n$.  In order to show this, we first make a 
\begin{remark}
There exists an increasing function $Q$ depending only on the coefficients, such that, if $w\in X_{T,M_0}$, $T>0$, is a solution to $(\text{IVP})_\epsilon$, 
$$\sup_{[0,T]} \sum_{|\alpha|\leq s_1-4} \|\langle x \rangle^2 \partial_x^\alpha \partial_t w \|_{L^2} \leq Q(M_0)$$
\end{remark}
This is clear because $\partial_t w = -\epsilon \Delta^2 w + \mathcal{L}(w)w$.
We now write the equation obtained after applying $J^{2m}$ to the equation in $(\text{IVP})_\epsilon$, after some calculations: $(2m=s)$
$$\partial_t J^{2m}u = -\epsilon \Delta^2J^{2m}u + i \mathcal{L}_{2m}(u)J^{2m}u + f_{2m}(x,t; (\partial_x^\beta u)_{|\beta|\leq 2m-1}; (\partial_x^\beta\bar{u})_{|\beta|\leq 2m-1} )$$
where 
\begin{equation*}
\begin{split}
\mathcal{L}_{2m}(u)v = &ia_{lk}(x,t;u,\bar{u}, \nabla_x u, \nabla_x \bar{u} )\partial_{x_lx_k}^2v \\
&+ ib_{lk}(x,t; u,\bar{u},\nabla_xu,\nabla_x \bar{u})\partial_{x_lx_k}^2\bar{v} \\
& + b_{2m,1,j}(x,t;(\partial_x^\alpha u)_{|\alpha|\leq 2}, (\partial_x^\alpha \bar{u})_{|\alpha|\leq 2}) R_j\partial_{x_j} v \\
& + b_{2m,2,j}(x,t;(\partial_x^\alpha u)_{|\alpha|\leq 2}, (\partial_x^\alpha \bar{u})_{|\alpha|\leq 2}) \tilde{R}_j\partial_{x_j} \bar{v} \\
& + c_{1,2m}(x,t;(\partial_x^\beta u)_{|\beta|\leq 3}, (\partial_x^\beta \bar{u})_{|\beta|\leq 3}) R_{2m,1} v \\
& + c_{2,2m}(x,t;(\partial_x^\beta u)_{|\beta|\leq 3}, (\partial_x^\beta \bar{u})_{|\beta|\leq 3}) R_{2m,2} \bar{v} \\
\end{split}
\end{equation*}
where $R_j$, $\tilde{R}_j$, $R_{2m,1}$, $R_{2m,2}$ are fixed 0th order $\Psi$DO.  The principal part of $\mathcal{L}_{2m}(u)$ is independent of $m$.  The coefficients $b_{2m,1,j}$, $b_{2m,2,j}$ depend on $m$ as a multiplicative constant, and on the original coefficients $a_{lk}$, $b_{lk}$, $\vec{b}_1$, $\vec{b}_2$, and their first derivatives.  They verify the asymptotic flatness, and the required decay property in $(\text{H}_{3,\text{gl}})$, $(\text{H}_{4,\text{gl}})$ by inspection, using the remark, and $(\text{H}1)-(\text{H}4)$.

It is then clear that $\mathcal{L}_{2m}(u)$ verifies the desired property.  Moreover, since $\tv u^\epsilon \tv_T\leq M_0$, it is easy to show that $$\sup_{0<t<T} \|f_{2m}(x,t;(\partial_x^\beta u)_{|\beta|\leq 2m-1}, (\partial_x^\beta \bar{u})_{|\beta|\leq 2m-1}) \|_{L^2} \leq P(M_0)$$ for some fixed increasing function $P$ of $M_0$.  Moreover, $x_lJ^{2m_1}u$ verifies similar equations.  One can then apply the first estimate in Theorem \ref{T:2B}, to obtain \eqref{E:20}.  Once a solution to $(\text{IVP})_\epsilon$ is constructed in $(0,\tilde{T})$ with $\tv u^\epsilon \tv_{\tilde{T}} \leq M_0$, $M_0$, $\tilde{T}$ independent of $\epsilon$, by considering the equation verified by $u^\epsilon-u^{\epsilon'}$, and using similar arguments, we obtain the existence of the limit as $\epsilon \to 0$, and its uniqueness in $C([0,\tilde{T}]; H^{s-1})\cap \{ u: \langle x \rangle^2 \partial_x^\alpha u\in C([0,\tilde{T}], L^2) \}$.  The solution $u$, in addition, belongs to $L^\infty([0,\tilde{T}], H^s)$.  To show that $u\in C([0,\tilde{T}]; H^s)$ and the continuous dependence, one uses the ``Bona-Smith regularization'' argument.

\section*{Problems for Lecture 6}
\begin{problems}
\item Carry out in detail the proof of Theorem \ref{T:2B}.  The key idea is to introduce the symmetrizer $S$ as in Lecture 6, and then to look at the system verified by $\vec{\beta}_1=\Psi_S \Psi_{\theta_{2R}}\vec{w}$, $\vec{\beta}_2=\Psi_{\phi_{2R}}\vec{w}$, where $\theta_{2R}=1-\phi_{2R}$, $\phi_{2R}\in C_0^\infty$, $\text{supp }\phi_{2R} \subset \{ \; |\xi| < 4R \; \}$, $\phi_{2R}\equiv 1$ on $\{ \; |\xi|< 2R \; \}.$  One chooses $R$ large, and uses the fact that if $T = \theta_R(\xi) S^{-1}$, $\Psi_T\Psi_S \Psi_{\theta_{2R}}\vec{w}=\Psi_{\theta_{2R}}\vec{w} + L_{-1}\Psi_{\theta_{2R}}\vec{w}$, where $L_{-1}$ is of order $-1$, with $S^0$ seminorms for its symbol small as $R\to +\infty$.  Thus, $I+L_{-1}$ is invertible, with inverse $N$ a 0th order $\Psi$DO.  We have also $N\Psi_+\Psi_S\Psi_{\theta_{2R}}\vec{w}=\Psi_{\theta_{2R}}\vec{w}$.  We also use Problem \chapterref{P:5.4}.\ref{P:5.4}.
\item Show that $(\text{H}_{4,\text{gl}})$ in Theorem \ref{T:2B} can be replaced by\\
$(\text{H}_{4,\text{gl}}')$:  $|\partial_t \text{Im }\vec{b_1}(x,t)| \leq \frac{C}{\langle x \rangle^2}$, $|\text{Im }\vec{b}_1(x,0)| \leq \frac{C_1}{\langle x \rangle^2}$, and
$$|\partial_t \text{Im}(b_{jk}(x,t)\xi_j\xi_k i \overline{\vec{b}_{2,l}}(x,t)\xi_l)| \leq \frac{C|\xi|^3}{\langle x \rangle^2}$$
and
$$| \text{Im}(b_{jk}(x,0)\xi_j \xi_k i \overline{ \vec{b}_{2,l}}(x,0)\xi_l)| \leq C_1 \frac{|\xi|^3}{\langle x \rangle^2}$$
\item \label{P:6.3}Formulate and prove variants of Theorems \ref{T:1B}, \ref{T:2B}, where the asymptotic flatness and decay of first order term conditions are of the type used in Theorem \ref{T:S3} (2).
\item Verify the details of the proof of the Theorem in Lecture 6.  Verify that Remark \ref{R:c} holds.
\item Use Problem \chapterref{P:6.3}.\ref{P:6.3} above to prove that, when $\frac{\partial}{\partial z} a_{kl}(x,0; 0, \ldots, 0)=0$, $\frac{\partial}{\partial z}b_{kl}(x,0;\ldots,0) =0$ and $\vec{b}_1$, $\vec{b}_2$ vanish quadratically at $\vec{z}=0$, the analog of the Theorem in Lecture 6 holds, without the use of weighted Sobolev spaces.
\end{problems}

\lecture[Solutions to problems (by Justin Holmer)]{Solutions to selected problems (An appendix by Justin Holmer)}

\section{Solutions to Lecture 1 problems}

\noindent \textbf{Problem 1a}.  
$$\|u\|_{L^\infty} \leq c \int |\hat{u}(\xi)| \, d\xi \leq c\left( \int |\hat{u}(\xi)|^2 (1+|\xi|)^{2s} \, d\xi \right)^{1/2} \left( \int (1+|\xi|)^{-2s} \, d\xi \right)^{1/2}$$

\noindent \textbf{Problem 1b}.  
We shall show that (where $\widehat{D^sf}(\xi)=|\xi|^s\hat{f}(\xi)$)
$$\|D^s(fg)\|_{L^2} \leq C\|f\|_{H^s}\|g\|_{H^s}$$
Let $F$ and $G$ be defined by $\hat{F}(\xi) = |\hat{f}(\xi)|$, $\hat{G}(\xi)=|\hat{g}(\xi)|$.  Note that 
$$\|F\|_{L^\infty} \leq \left( \int (1+|\xi|)^{2s} |\hat{F}(\xi)|^2 d\xi \right)^{1/2} = \left( \int (1+|\xi|)^{2s} |\hat{f}(\xi)|^2 d\xi \right)^{1/2}$$
by (1a), and similarly for $G$.  By Plancherel,
\begin{align*}
\|D^s(fg)\|_{L^2} &= \left( \int_\xi |\xi|^{2s} \left| \int_{\xi=\xi_1+\xi_2} \hat{f}(\xi_1)\hat{g}(\xi_2) \, d\xi_1 \, d\xi_2 \right|^2 \, d\xi\right)^{1/2}\\
& \leq \left( \int_\xi  |\xi|^{2s} \left| \int_{\xi=\xi_1+\xi_2} \hat{F}(\xi_1)\hat{G}(\xi_2) \, d\xi_1 \, d\xi_2 \right|^2 \, d\xi \right)^{1/2}\\ 
&\leq  
\begin{aligned}[t]
& \left( \int_\xi  \left| \int_{\substack{\xi=\xi_1+\xi_2\\ |\xi_1|\geq |\xi_2|}}  |\xi_1|^{2s} \hat{F}(\xi_1)\hat{G}(\xi_2) \, d\xi_1 \, d\xi_2 \right|^2 \, d\xi\right)^{1/2} \\
&+ \left( \int_\xi \left| \int_{\substack{\xi=\xi_1+\xi_2\\ |\xi_1|\leq |\xi_2|}} \hat{F}(\xi_1) |\xi_2|^{2s}\hat{G}(\xi_2) \, d\xi_1 \, d\xi_2 \right|^2 \, d\xi\right)^{1/2}
\end{aligned}\\
& \leq \| D^sF \cdot G \|_{L^2} + \|F \cdot D^s G \|_{L^2} \\
& \leq \| D^sF\|_{L^2}\|G\|_{L^\infty} + \|F\|_{L^\infty} \|D^sG\|_{L^2}
\end{align*}

\noindent \textbf{Problem 1c}.  First, consider the case $s=k$ an integer.  The chain rule is (for $\alpha$ a nonzero multiindex, $|\alpha|\leq k$, and $\alpha_1, \ldots \alpha_j$ multi-indices)
$$\partial^\alpha f(u) = \sum_{\substack{ \alpha_1+\cdots +\alpha_j=\alpha \\ 1\leq j \leq k}} C_{\alpha_i,j} f^{(j)}(u)\partial^{\alpha_1}u \cdots \partial^{\alpha_j}u$$
Then
$$\| \partial^\alpha f(u)\|_{L^2} \leq  \sum_{\substack{ \alpha_1+\cdots +\alpha_j=\alpha \\ 1\leq j \leq k}} C_{\alpha_i,j} \| f^{(j)}(u)\|_{L^\infty} \|\partial^{\alpha_1}u \cdots \partial^{\alpha_j}u\|_{L^2}$$
We then use that
$$\|f^{(j)}(u)\|_{L^\infty} \leq R_j(\|u\|_{L^\infty})$$
where 
$$R_j(r)=\sup_{|y|\leq r} |f^{(j)}(y)|$$
and also derive a mulitilinear generalization of the estimate in (1b):
$$\|\partial^{\alpha_1}u_1 \cdots \partial^{\alpha_j}u_j \|_{L^2} \leq C \|u_1\|_{H^k} \cdots \|u_j\|_{H^k}$$
This gives that
$$\| \partial^\alpha f(u)\|_{L^2} \leq  \sum_{j=1}^k C_jR_j(\|u\|_{H^k}) \|u\|_{H^k}^j =: R(\|u\|_{H^k})$$
where $R(r)$ is an increasing function with $R(0)=0$.
Also,
$$\|f(u)\|_{L^2} \leq C R_0(\|u\|_{H^k})$$
Note that $R_0(0) = 0$ since $f(0)=0$.
To handle the case of fractional $s$, use the Leibniz rule and chain rule for fractional derivatives.  The following 1-D statements appear in the appendix of \cite{KPV93}.
\begin{theorem}[Leibniz rule 1]  Let $\sigma \in (0,1)$, $\sigma_1, \sigma_2 \in [0,\sigma]$ with $\sigma=\sigma_1+\sigma_2$.  Let $p,p_1,p_2 \in (1,\infty)$ be such that 
$$\frac{1}{p}=\frac{1}{p_1}+\frac{1}{p_2}$$
Then
$$\| D^\sigma(fg) - fD^\sigma g - g D^\sigma f \|_{L^p} \leq c \| D^{\sigma_1} f \|_{L^{p_1}} \| D^{\sigma_2} g \|_{L^{p_2}}$$
\end{theorem}

\begin{theorem}[Leibniz rule 2]
Let $\sigma \in (0,1)$ and $p\in (1,\infty)$.  Then
$$\| D^\sigma (fg) - fD^\sigma g - g D^\sigma f \|_p \leq c \|g\|_\infty \|D^\sigma f \|_p$$
\end{theorem}

\begin{theorem}[Chain rule 1]
Let $\sigma\in (0,1) $ and $p,p_1,p_2 \in (1,\infty)$,  such that
$$ \frac{1}{p}=\frac{1}{p_1}+\frac{1}{p_2} \quad $$
Then
$$\| D^\sigma F(f)\|_{L^p} \leq c\|F'(f)\|_{L^{p_1}}\|D^\sigma f\|_{L^{p_2}}$$
\end{theorem}

\begin{theorem}[Chain rule 2]
If $p\in (1,\infty)$, $r>1$, and $h\in L_{\text{loc}}^{rp}(\mathbb{R})$.  Then
$$\| D^\sigma F(f) h \|_p \leq c \|F'(f)\|_\infty \|D^\sigma(f) M(h^{rp})^{1/rp}\|_p$$
where $M$ is the Hardy-Littlewood maximal operator.
\end{theorem}

\noindent \textbf{Problem 2}.  I will give an elaboration of the Bona-Smith method, which appears as Step 4 in the notes, however I will instead work with the equation
$$\partial_t u =i\partial_x^2 u + iu\partial_x\bar{u}$$
in order to underscore the wide range of applicability of the method.  First, we apply $\partial_x^k$ for $k\geq 3$:
\begin{equation}
\label{E:aux1}
\partial_t (\partial_x^k u) = i\partial_x^2 (\partial_x^ku) + u \partial_x^{k+1}\bar{u} + \text{lower order terms}
\end{equation}
(we shall drop the lower order terms in the remainder of the exposition). Pairing with $\partial_x^k \bar{u}$, integrating in $x$, and taking the real part gives
\begin{equation}
\label{E:aux2}
\partial_t \int_x |\partial_x^k u|^2 = 2\text{Re }i\int_x u\, \partial_x^{k+1}\bar{u}\, \partial_x^k\bar{u} =  -\text{Re }i\int_x \partial_x u \,  (\partial_x^k\bar{u})^2
\end{equation}
If $\|u_0\|_{H^k} \leq R$, we can integrate in time to obtain $T=T(R)>0$ for which $\sup_{[0,T]}\|u(t)\|_{H^k}\leq 2\|u_0\|_{H^k}$ is \textit{a priori} bounded.  Thus existence and uniqueness of a solution on $[0,T]$ for this equation follows by the techniques of Step 1-3 in Lecture 1.  Now we use the Bona-Smith method to show that the ``data to solution'' map is continuous as a map from $H^k$ to $C([0,T];\, H^k)$.  Set $u_0^\delta = \varphi_\delta\ast u_0$, where $\varphi\in \mathcal{S}(\mathbb{R}^n)$, $\int \varphi =1$, $\int x^\alpha \varphi(x)dx =0$ for $|\alpha|\neq 0$.  (Get $\varphi$ by taking $\hat{\varphi}(\xi)=1$ on $|\xi|\leq 1$, $\hat{\varphi}(\xi)=0$ on $|\xi|\geq 2$.)  Then let $u^\delta$ be the solution corresponding to $u_0^\delta$. 

\noindent \underline{Step A}.  For $l\geq 0$, $\sup_{[0,T]}\|u^\delta(t)\|_{H^{k+l}} \leq 2R\delta^{-l}$.  This is obtained from \eqref{E:aux1}, \eqref{E:aux2} with $k$ replaced by $k+l$ and also noting that $\|u_0^\delta\|_{H^{k+l}} \leq \delta^{-l}\|u_0\|_{H^k}$.

\noindent \underline{Step B}.  $\sup_{[0,T]} \|(u^\delta-u)(t)\|_{L^2} \leq 2\|u_0^\delta - u_0\|_{L^2} \leq \delta^kh(\delta)$, where $h(\delta) \to 0$ and $|h(\delta)|\leq R$.
From the equation,
\begin{align*}
\partial_t(u^\delta-u) &= i\partial_x^2(u^\delta - u) + iu^\delta \partial_x \bar u^\delta - iu\partial_x \bar{u} \\
&= i\partial_x^2(u^\delta - u) + i(u^\delta-u)\partial_x\bar{u}^\delta +iu\partial_x(\overline{u^\delta - u})
\end{align*}
Pair with $\overline{u^\delta-u}$, integrate in $x$, take the real part, integrate in time to obtain:
$$\| (u^\delta - u)(t)\|_{L_x^2}^2 \leq \|u_0^\delta - u_0\|_{L^2}^2 + T(\|\partial_x u^\delta \|_{L_T^\infty L_x^\infty}+ \|\partial_x u\|_{L_T^\infty L_x^\infty})\|u^\delta - u\|_{L_T^\infty L_x^2}^2  $$
where, to estimate the last nonlinear term, we used that $\partial_x (\overline{u^\delta-u}) \, (\overline{u^\delta-u}) = \frac{1}{2} \partial_x (\overline{u^\delta-u})^2$ and integration by parts.  Thus, by suitable choice of $T=T(R)>0$,
$$\|(u^\delta-u)(t)\|_{L_x^2} \leq 2\|u_0^\delta-u_0\|_{L^2}$$
Now observe
$$|\hat{\varphi}(\delta\xi) - 1| \leq \delta |\xi|  \sup_{[0,\delta\xi]}|(\partial_\xi \hat{\varphi})(\eta)|$$
However, because $\partial_\xi\hat{\varphi}(0) =0$, we also have
$$|\partial_{\xi}\hat{\varphi}(\eta) | \leq \delta|\xi| \sup_{[0,\delta\xi]}|\partial_\xi^2\hat{\varphi}(\eta) |$$
Continuing, we have $\forall$ integer $k$,
$$|\hat{\varphi}(\delta\xi)-1| \leq \delta^k |\xi|^k \sup_{[0,\delta\xi]}|(\partial_\xi^k\hat{\varphi})(\eta)|$$
and thus
$$\left( \int \left| \hat{\varphi}(\delta\xi)-1\right|^2 |\hat{u}_0(\xi)|^2 \, d\xi \right)^{1/2} \leq \delta^k \underbrace{\left( \int_\xi \sup_{[0,\delta\xi]}|\partial_\xi^k\hat{\varphi}(\eta)|^2 |\xi|^{2k}|\hat{u}_0(\xi)|^2 \, d\xi \right)^{1/2}}_{h(\delta)}$$
with $\lim_{\delta\to 0}h(\delta)=0$ by dominated convergence.

\noindent \underline{Step C}. For $r\leq k$,  $\sup_{[0,T]} \|(u^\delta-u)(t)\|_{H^r} \leq R^\frac{r}{k}\delta^{k-r}h(\delta)^\frac{k-r}{k}$.  This follows from Step B by interpolation:
$$\|u^\delta - u\|_{H^r} \leq \|u^\delta - u\|_{L^2}^\frac{k-r}{k} \|u^\delta-u\|_{H^k}^\frac{r}{k}$$

\noindent \underline{Step D}.   $\sup_{[0,T]} \|(u^\delta-u)(t)\|_{H^k} \leq 2\|u_0^\delta - u_0\|_{H^k}$.  By \eqref{E:aux1} for $u$ and $u^\delta$,
\begin{align*}
\partial_t\partial_x^k(u^\delta-u) &= i\partial_x^2\partial_x^k(u^\delta-u) + iu^\delta \partial_x^{k+1}\bar{u}^\delta-iu\partial_x^{k+1}\bar{u}\\
&= i\partial_x^2\partial_x^k(u^\delta-u) +i(u^\delta-u)\partial_x^{k+1}\bar{u}^\delta +iu\partial_x^{k+1}(\overline{u^\delta-u})
\end{align*}
and thus
$$\| \partial_x^k (u^\delta-u)(t)\|_{L_x^2}^2 \leq 
\begin{aligned}[t]
&\|\partial_x^k(u_0^\delta-u_0)\|_{L^2}^2 \\
&+ T\|u^\delta - u\|_{L_T^\infty L_x^2}\|\partial_x^{k+1}u^\delta\|_{L_T^\infty L_x^\infty}\|\partial_x^k(u^\delta-u)\|_{L_T^\infty L_x^2} \\
&+ T\|\partial_x u\|_{L_x^\infty}\|\partial_x^k(u^\delta-u)\|_{L_x^2}^2
\end{aligned}$$
We further estimate the first nonlinear piece using Step A and C (with $r=0$ in this case) to obtain
$$\| \partial_x^k(u^\delta - u)(t)\|_{L_x^2} \leq 2\|\partial_x^k(u_0^\delta - u_0)\|_{L^2} + TR\delta^{k-2}h(\delta)$$

\noindent \underline{Step E}.  If both $\|u_{10}\|_{H^k}\leq R$ and $\|u_{20}\|_{H^k}\leq R$, then $\sup_{[0,T]} \|(u_1^\delta - u_2^\delta)(t) \|_{H^k} \leq 2\|u_{10}-u_{20}\|_{H^k}$, where $T=T(R)$. This follows by the above techniques.

We can now complete the argument.  Let $\epsilon>0$, and suppose $u_{10}$ and $u_{20}$ are such that $\|u_{10}\|_{H^k}\leq R$, $\|u_{20}\|_{H^k}\leq R$ and $\|u_{10}-u_{20}\|_{H^k} \leq \frac{\epsilon}{10}$.  Then obtain $\delta=\delta(u_{10},u_{20})$ such that $\|u_{10}^\delta-u_{10}\|_{H^k}\leq \frac{\epsilon}{10}$ and $\|u_{20}^\delta-u_{20}\|_{H^k} \leq \frac{\epsilon}{10}$.  Let $T=T(R)$ (independent of $\delta$) be such that the claims in Steps A-E hold; then the results of Steps A-E give that
$$\sup_{[0,T]}\|(u^1-u^2)(t)\|_{H^k} \leq \epsilon$$

\noindent \textbf{Problem 3}.  The following proof seems only to apply to $k> \frac{n}{2}+2$.  In the presentation, I shall restrict to the case $s=k$ integer and to $n=1$ (1-D), and to monomial nonlinearity, i.e.\
$$ F(u,\bar{u}, \partial_x \bar{u}) = u^\alpha\bar{u}^\beta (\partial_x\bar{u})^\gamma$$
Then the equation takes the form
$$\partial_t u = i\partial_x^2 u + u^\alpha \bar{u}^\beta (\partial_x \bar{u})^\gamma$$
Apply $\partial_x^k$, $k\geq 3$, and separate terms in the Leibniz expansion of $F$
\begin{align}
\partial_t (\partial_x^ku) &= i\partial_x^2 (\partial_x^ku) + \sum_{j=0}^{k-1} C_j \partial_x^{k-j}(u^\alpha\bar{u}^\beta) \partial_x^j(\partial_x \bar{u})^\gamma + u^\alpha \bar{u}^\beta \partial_x^k(\partial_x \bar{u})^\gamma \notag\\
&= 
\begin{aligned}[t]
&i\partial_x^2 (\partial_x^ku) + \sum_{j=0}^{k-1} C_j \partial_x^{k-1-j}[(\alpha-1)u^{\alpha-1}(\partial_xu)\bar{u}^\beta+(\beta-1)u^\alpha \bar{u}^{\beta-1} (\partial_x \bar{u})] \partial_x^j(\partial_x \bar{u})^\gamma \\
&+ u^\alpha \bar{u}^\beta \partial_x^k(\partial_x \bar{u})^\gamma 
\end{aligned}\notag\\
&= i\partial_x^2(\partial_x^k u) + \text{I} + \text{II} \label{PE:50}
\end{align}
We further separate term II as:
\begin{align*}
\text{II} &= \gamma u^\alpha \bar{u}^\beta (\partial_x \bar{u})^{\gamma-1} \partial_x^{k+1}\bar{u}+ u^\alpha\bar{u}^\beta \sum_{\substack{\nu\geq 2\\ j_1\geq 1, \ldots, j_\nu\geq 1 \\ j_1+\cdots+j_\nu=k}} C_{j,\nu}(\partial_x \bar{u})^{\gamma-\nu} \partial_x^{j_1+1} \bar{u} \cdots \partial_x^{j_\nu+1} \bar{u}\\
&= \text{II}_1 + \text{II}_2
\end{align*}
Pair \eqref{PE:50} with $\partial_x^k \bar{u}$, integrate, and take the real part.  For term I
\begin{align*}
&\sum_{j=0}^{k-1} C_j \text{Re}\int \partial_x^{k-1-j}[(\alpha-1)u^{\alpha-1}(\partial_xu)\bar{u}^\beta+(\beta-1)u^\alpha \bar{u}^{\beta-1} \partial_x \bar{u}] \partial_x^j(\partial_x \bar{u})^\gamma \partial_x^k \bar{u}\\
& \leq C \| \partial_x^{k-1-j}[(\alpha-1)u^{\alpha-1}(\partial_xu)\bar{u}^\beta+(\beta-1)u^\alpha \bar{u}^{\beta-1} \partial_x \bar{u}]\partial_x^j(\partial_x \bar{u})^\gamma \|_{L^2} \| \partial_x^k \bar{u}\|_{L^2}\\
&\leq C(\|u^{\alpha-1}(\partial_xu)\bar{u}^\beta \|_{H^{k-1}} + \|u^\alpha\bar{u}^{\beta-1}(\partial_x \bar{u})\|_{H^{k-1}})\|(\partial_x\bar{u})^\gamma\|_{H^{k-1}}
\end{align*}
and use that $H^{k-1}$ is an algebra. For term $\text{II}_1$, 
$$\text{Re}\int \gamma u^\alpha \bar{u}^\beta (\partial_x \bar{u})^{\gamma-1} \partial_x^{k+1}\bar{u}\partial_x^k\bar{u}$$
Use that $\partial_x^{k+1}\bar{u} \partial_x^k \bar{u} = \frac{1}{2} \partial_x (\partial_x^k \bar{u})^2$, and integrate by parts.  For term $\text{II}_2$, 
\begin{align*}
&\text{Re}\int \sum_{\substack{\nu\geq 2\\ j_1\geq 1, \ldots, j_\nu\geq 1 \\ j_1+\cdots +j_\nu=k}} u^\alpha \bar{u}^\beta (\partial_x \bar{u})^{\gamma-\nu} \partial_x^{j_1+1} \bar{u} \cdots \partial_x^{j_\nu+1} \bar{u}\partial_x^k\bar{u}\\
& \leq \sum \| u^\alpha \bar{u}^\beta (\partial_x \bar{u})^{\gamma-\nu}\|_{L^\infty} \| \partial_x^{j_1+1} \bar{u} \cdots \partial_x^{j_\nu+1}\bar{u}\|_{L^2} \| \partial_x^k\bar{u}\|_{L^2}\\
& \leq \sum \| u^\alpha \bar{u}^\beta (\partial_x \bar{u})^{\gamma-\nu}\|_{L^\infty} \| [\partial_x^{j_1-1} (\partial_x^2\bar{u})] \cdots [\partial_x^{j_\nu-1} (\partial_x^2 \bar{u})]\|_{L^2} \| \partial_x^k\bar{u}\|_{L^2}
\end{align*}
Since $(j_1-1) + \cdots + (j_\nu-1) =k-\nu \leq k-2$, use that $H^{k-2}$ is an algebra.\\

\noindent\textbf{Problem 4a}.  (Taken from \cite{St2}).  Use that $\|T^\ast\|=\|T\|$ to show that $\|T^\ast T \|=\|T\|^2$.  This shows that, for $B$ self-adjoint, $\|B^2\|=\|B\|^2$, and we can thus deduce (by induction for $m=2^k$, then an interpolation-type argument for arbitrary $m$) that 
$$\|\underbrace{T^*T \cdots T^*T}_{m \text{ copies}}\|=\|T\|^{2m}$$
Now let $T=\sum_{j=-N}^{N}T_j$, and note that
$$\underbrace{T^*T \cdots T^*T}_{m \text{ copies}} = \sum_{\substack {j_1, \ldots, j_m \\ k_1, \cdots, k_m }} T^\ast_{j_1}T_{k_1} \cdots T^\ast_{j_m}T_{k_m}$$
Grouping terms as $(T_{j_1}^\ast T_{k_1})\cdots (T_{j_m}^\ast T_{k_m})$,
$$\|T^\ast_{j_1}T_{k_1} \cdots T^\ast_{j_m}T_{k_m}\| \leq \gamma(j_1-k_1)^2 \cdots \gamma(j_m-k_m)^2 $$
Grouping terms as $T_{j_1}^\ast (T_{k_1}T_{j_2}^\ast) \cdots (T_{k_{m-1}}T_{j_m}^\ast) T_{k_m}$,
$$\|T^\ast_{j_1}T_{k_1} \cdots T^\ast_{j_m}T_{k_m}\|\leq  \gamma(0)\gamma(k_1-j_2)^2 \cdots \gamma(k_{m-1}-j_m)^2 \gamma(0)$$
and taking the geometric mean
$$\|T^\ast_{j_1}T_{k_1} \cdots T^\ast_{j_m}T_{k_m}\| \leq \gamma(0) \gamma(j_1-k_1)\gamma(k_1-j_2) \cdots \gamma(k_{m-1}-j_m)\gamma(j_m-k_m)$$
Thus
$$\|T\|^{2m} \leq \sum_{\substack {j_1, \ldots, j_m \\ k_1, \cdots, k_m }} \gamma(0) \gamma(j_1-k_1)\gamma(k_1-j_2) \cdots \gamma(k_{m-1}-j_m)\gamma(j_m-k_m)\leq (2N+1)\gamma(0)A^{2m-1}$$
which gives
$$\|T\| \leq [(2N+1)\gamma(0)]^{\frac{1}{2m}}A^\frac{2m-1}{2m}$$
Letting $m\to \infty$, we get
$$\|T\| \leq A$$\\

\noindent\textbf{Problem 4b}.  (Taken from \cite{St2}).  By Plancherel, it suffices to prove $L^2$ boundedness of the operator defined by
$$Tf(x) = \int_\xi e^{ix\cdot \xi} a(x,\xi) f(\xi) \, d\xi$$
Let $\phi(x)$ be a smooth function supported in $Q_0 = \{ \; x\in \mathbb{R}^n \; \big| \; |x_j| \leq 1, j=1,\ldots , n \; \}$ such that, for each $x$,  
$$ \sum_{i\in \mathbb{Z}^n} \phi(x-i) = 1$$
To construct such a $\phi$, take a $\phi_0(x)$ such that $\phi_0(x)=1$ on $\frac{1}{2}Q_0$ with $\text{supp }\phi_0(x) \subset Q_0$. Then, for each $x$, $\sum_{i\in \mathbb{Z}^n} \phi_0(x-i)$ is $\geq 1$ since integer translates of $\frac{1}{2}Q_0$ cover $\mathbb{R}^n$, and is a finite sum with $\leq 3^n$ terms for each $x$.  Set 
$$\phi(x) = \frac{\phi_0(x)}{\sum_{i\in \mathbb{Z}^n} \phi_0(x-i)} $$
Let 
$$T_{ij}f(x) = \phi(x-i) \int_\xi e^{ix\cdot \xi} \phi(\xi-j) a(x,\xi) f(\xi) \, d\xi$$
The adjoint $T_{ij}^\ast$ is
$$T_{ij}^\ast f(\xi) = \phi(\xi-j) \int_x \phi(x-i) e^{-ix\cdot \xi} \overline{a(x,\xi)} f(x) \, dx$$
Thus,
$$T_{ij}T_{i'j'}^\ast f(x) =
 \int_y K(x,y) f(y) \, dy$$
where $$K(x,y) = \int_\xi e^{i(x-y)\cdot \xi} \phi(x-i)\phi(\xi-j)\phi(\xi-j')\phi(y-i')a(x,\xi)\overline{a(y,\xi)} \, d\xi$$
By integration by parts,
$$|K(x,y)| \leq 
\left\{
\begin{aligned}
&C_N (1+|x-y|)^{-N} \phi(x-i)\phi(y-i')& \text{if }|j-j'|\leq 1\\
&0 & \text{if }|j-j'|\geq 2
\end{aligned}
\right.
$$
and hence
$$|K(x,y)|\leq C_N (1+|i-i'|)^{-N}(1+|j-j'|)^{-N} \phi(x-i)\phi(y-i')$$
We thus have
$$\| T_{ij}T_{i'j'}^\ast f \|_{L^2} \leq C(1+|i-i'|)^{-N}(1+|j-j'|)^{-N}\|f\|_{L^2}$$
A similar calculation shows that
$$\| T_{i'j'}^\ast T_{ij}f\|_{L^2} \leq C(1+|i-i'|)^{-N}(1+|j-j'|)^{-N}\|f\|_{L^2}$$
We then apply Cotlar-Stein.\\

\noindent\textbf{Problem 5}.  (Taken from \cite{St2}).  \textit{Note: }Actually, to give the following formulas meaning as absolutely convergent integrals and to carry out the following computation rigorously, we need to truncate the symbols $a(x,\xi)$, $b(x,\xi)$ by replacing them with $\gamma(\epsilon x, \epsilon \xi)a(x,\xi)$ and $\gamma(\epsilon x, \epsilon \xi)b(x,\xi)$ where $\gamma(0,0)=1$ and $\gamma\in C_0^\infty(\mathbb{R}^n\times \mathbb{R}^n)$.  The estimate that we will obtain will be independent of $\epsilon$ and then we can pass to the limit $\epsilon\to 0$ at the end.  \\
First assume that $b(x,\xi)$ has compact $x$-support.  We have
$$T_bf(y) = \int_\xi e^{iy\cdot \xi} b(y,\xi) \, \hat{f}(\xi) \, d\xi$$
and therefore
$$\widehat{T_bf}(\eta) = \int_y e^{-iy\cdot \eta} \int_\xi e^{iy\cdot \xi} b(y,\xi) \, \hat{f}(\xi) \, d\xi\, dy = \int_\xi \hat{b}(\eta-\xi,\xi) \hat{f}(\xi)\, d\xi$$
Also
$$T_aT_bf(x) = \int_\eta e^{ix\cdot \eta} a(x,\eta) \widehat{T_bf}(\eta) \, d\eta$$
Substituting, we get
$$T_aT_bf(x) = \int_\xi e^{ix\cdot \xi}\left[ \underbrace{\int_\eta e^{ix\cdot (\eta-\xi)} a(x,\eta) \hat{b}(\eta-\xi,\xi) \, d\eta}_{\text{call this }c(x,\xi)} \right] \hat{f}(\xi) \, d\xi$$
We have
$$c(x,\xi) = \int_\eta e^{ix\cdot \eta} a(x,\xi+\eta)\hat{b}(\eta,\xi) \, d\eta$$
Now use the Taylor expansion for $a(x,\xi+\eta)$ around $a(x,\xi)$:
$$a(x,\xi+\eta) = \sum_{|\alpha|\leq N-1} \partial_\xi^\alpha a(x,\xi) \frac{\eta^\alpha}{\alpha !}+R_N(x,\xi,\eta)$$
and thus
\begin{align*}
c(x,\xi) &= \sum_{|\alpha|\leq N-1} \partial_\xi^\alpha a(x,\xi) \frac{(-i)^{|\alpha|}}{\alpha !} \int_\eta (i\eta)^\alpha \hat{b}(\eta,\xi) \, d\eta + \int e^{ix\cdot \eta} R_N(x,\xi,\eta) \hat{b}(\eta,\xi) \, d\eta \\
&\leq \sum_{|\alpha|\leq N-1}  \frac{(-i)^{|\alpha|}}{\alpha !} \partial_\xi^\alpha a(x,\xi)\partial_x^\alpha b(x,\xi) + \int e^{ix\cdot \eta} R_N(x,\xi,\eta) \hat{b}(\eta,\xi) \, d\eta
\end{align*}
For the remainder, we have the estimate
\begin{align*}
|R_N(x,\xi,\eta)| & \leq \frac{1}{(N-1)!} \sum_{|\alpha|\leq N} \sup_{0\leq t\leq 1} |\partial_\xi^\alpha a(x,\xi+t\eta)| |\eta|^N \\
& \leq C_N \sup_{0\leq t \leq 1} (1+|\xi+t\eta|)^{m_1-N}|\eta|^N \\
& \leq \left\{
\begin{aligned}
& c_N (1+|\xi|)^{m_1-N}|\eta|^N && \text{if }|\xi|\geq 2|\eta| \\
& c_N |\eta|^N && \text{if } |\xi|\leq 2|\eta| \text{ at least if }m_1-N \leq 0
\end{aligned}
\right.
\end{align*}
By integration by parts, and the assumption that $b(x,\xi)$ has compact $x$ support,
$$|\hat{b}(\eta,\xi)| \leq c_M (1+|x|)^{-M}(1+|\xi|)^{m_2}$$
and so
$$\left| \int_\eta e^{ix\cdot \eta} R_N(x,\xi,\eta)\hat{b}(\eta,\xi) \, d\eta \right| \leq \left\{ 
\begin{aligned}
&c_N(1+|\xi|)^{m_1-N}(1+|\xi|)^{m_2} && \text{if }|\xi|\geq 2|\eta| \\
& c_N (1+|\xi|)^{m_2} \int_{|\eta|\geq \frac{|\xi|}{2}} (1+|\eta|)^{-M} \,d\eta && \text{if }|\xi|\leq 2|\eta|
\end{aligned}
\right.
$$
To handle the non-compact case, we realize that it suffices to establish the formula in the neighborhood of an arbitary, but fixed, point $x_0$.  Let $\rho(x)=1$ for $|x-x_0|\leq 1$, and $\rho(x)=0$ for $|x-x_0|\geq 2$.  Let $b(x,\xi) = \rho(x) b(x,\xi) + (1-\rho(x))b(x,\xi)=b_1(x,\xi)+b_2(x,\xi)$.  We then obtain symbols $c_1(x,\xi)$ such that $T_{c_1} = T_{a}\circ T_{b_1}$ and $T_{c_2}=T_a\circ T_{b_2}$, where $c_1(x,\xi)$ has the correct expansion and
$$c_2(x,\xi) = \int_\eta \int_y e^{i(x-y)\cdot (\eta-\xi)} a(x,\eta) (1-\rho(y)) b(y,\xi) \, d\eta dy$$
which, we shall show belongs to $S^{-\infty}$, at least if $|x-x_0|\leq \frac{1}{2}$.  Suppose $|x-x_0|\leq \frac{1}{2}$.
\begin{align*}
c_2(x,\xi) &= \int_\eta \int_y e^{i(x-y)\cdot(\eta-\xi)} |x-y|^{-2N} \Delta_\eta^Na(x,\eta) (1-\rho(y)) b(y,\xi) \, d\eta \, dy\\
&= \int_\eta \int_y 
\begin{aligned}[t]
&e^{i(x-y)\cdot (\eta-\xi)} (1+|\eta-\xi|^2)^{-N} \Delta_\eta^Na(x,\eta) \\
&\cdot (I-\Delta_y)^N \left[ |x-y|^{-2N} (1-\rho(y)) b(y,\xi) \right] \, d\eta \, dy
\end{aligned}
\end{align*}
Since $|x-x_0|\leq \frac{1}{2}$, $|y-x_0|\geq \frac{1}{2}$, $|x-y|^{-2N}\leq c(1+|x-y|)^{-2N}$, 
$$|c_2(x,\xi)| \leq \int_\eta \int_y (1+|\eta-\xi|^2)^{-N} (1+|x-y|^2)^{-N} (1+|\eta|)^{m_1-N} \, d\eta \, dy \leq C_N(1+|\xi|)^{-N}$$\\

\noindent \textbf{Problem 6}. We are given symbols $a_j(x,\xi)$ such that
$|\partial_x^\alpha \partial_\xi^\beta a_j(x,\xi)| \leq C_{\alpha,\beta,j} \langle \xi \rangle^{m_j-|\beta|}$.
Let us treat only the 1D case to simplify notation.  Replace $C_{\alpha,\beta, j}$ with
$$C_{\alpha,\beta,j} \equiv \sup_{0 \leq \alpha' \leq \alpha , 0 \leq \beta' \leq \beta} C_{\alpha',\beta',j}$$
Let $\varphi(\xi) = 1 $ for $|\xi|\geq 2$ and $\varphi(\xi)=0$ for $|\xi|\leq 1$.
Fix a particular pair $(\alpha,\beta)$.  We shall explain how to obtain an increasing sequence $H_j$ tending to $\infty$ such that, if $R_j \geq H_j$ for large $j$, then $a(x,\xi)$ defined by
$$a(x,\xi) \equiv \sum_{j=0}^\infty \varphi\left(\frac{\xi}{R_j}\right) a_j(x,\xi)$$
satisfies
$$\left| \partial_x^\alpha\partial_\xi^\beta \left[ a(x,\xi) - \sum_{j=0}^k a(x,\xi) \right] \right|  \leq A \langle \xi \rangle^{m_{k}}$$
Indeed, with this definition, we have
$$\partial_x^\alpha\partial_\xi^\beta \left[ a(x,\xi) - \sum_{j=0}^ka_j(x,\xi) \right]  = \partial_x^\alpha \partial_\xi^\beta \left[ \underbrace{-\sum_{j=0}^k \left( 1- \varphi\left( \frac{\xi}{R_j} \right)\right) a_j(x,\xi)}_{\text{in }S^{-\infty}} + \sum_{j=k+1}^\infty \varphi \left( \frac{\xi}{R_j} \right) a_j(x,\xi)\right]$$
\begin{align*}
\partial_x^\alpha \partial_\xi^\beta \sum_{j=k+1}^\infty \varphi \left( \frac{\xi}{R_j} \right) a_j(x,\xi) &=  \sum_{j=k+1}^\infty \sum_{\sigma=0}^\beta \binom{\beta}{\sigma} \partial_\xi^\sigma \left[ \varphi\left(\frac{\xi}{R_j} \right) \right] \partial_x^\alpha \partial_\xi^{\beta-\sigma}a_j(x,\xi)  \\
&= \sum_{j=k+1}^\infty \sum_{\sigma =0}^\beta \binom{\beta}{\sigma} \left[ \left( \frac{\xi}{R_j} \right)^\sigma \partial_\xi^\sigma \varphi \left( \frac{\xi}{R_j} \right)\right] \xi^{-\sigma} \partial_x^\alpha \partial_\xi^{\beta-\sigma}a_j(x,\xi)
\end{align*}
Let 
$$B_\beta = \sum_{\sigma=0}^\beta \binom{\beta}{\sigma} \sup_\xi |\xi^\sigma \partial_\xi^\sigma \varphi (\xi)|$$
Then
\begin{align*}
\left| \partial_x^\alpha \partial_\xi^\beta \sum_{j=k+1}^\infty \varphi \left( \frac{\xi}{R_j} \right) a_j(x,\xi) \right| &\leq \sum_{j=k+1}^\infty B_\beta C_{\alpha,\beta,j} \langle \xi \rangle ^{m_j-\beta} \chi_{|\xi|\geq R_{j}}\\
& \leq \sum_{j=k+1}^\infty B_\beta C_{\alpha,\beta,j} R_j^{m_j-m_{j-1}} \langle \xi \rangle ^{m_k-\beta}
\end{align*}
Thus we see it suffices to take $H_j$ such that
$$C_{\alpha,\beta,j}H_j^{m_j-m_{j-1}} \leq 2^{-j}$$
Let $(\alpha_l,\beta_l)$ be an enumeration of the pairs in $\mathbb{Z}_{\geq 0}^2$.  For each $l$, obtain $H_{l,j}$ as above, with the additional requirement that $H_{l+1,j}\geq H_{l,j}$, $\forall \; l, j$.  Then take the diagonal, i.e. set $R_j=H_{j,j}$, and define $a(x,\xi)$ as above.

\section{Solutions to Lecture 2 problems}

\noindent \textbf{Problems 1b}. (Taken from \cite{St2}, p. 258-259)
\begin{align*}
\Psi_{[c]}f(x) &= \int_y \int_\xi e^{i\xi \cdot (x-y)} c(x,y,\xi) f(y) \, dy d\xi \\
&= \int_y \int_\eta e^{i\eta\cdot(x-y)} c(x,y,\eta) f(y) dy d\eta \\
&= \int_y \int_\eta e^{i\eta\cdot(x-y)} c(x,y,\eta) \left[ \int_\xi e^{iy\cdot \xi}\hat{f}(\xi) \, d\xi \right] dy d\eta \\
&= \int_\xi e^{ix\cdot \xi} \left[ \int_y \int_\eta c(x,y,\eta) e^{i(\eta-\xi)\cdot(x-y)} \, dy d\eta \right] \hat{f}(\xi) \, d\xi
\end{align*}
Thus, we should set
\begin{equation} \label{PE:10}
a(x,\xi) = \int_y \int_\eta c(x,y,\eta) e^{i(\eta-\xi)\cdot(x-y)} \, dy d\eta
\end{equation}
Assume first that $c(x,y,\xi)$ has compact $y$-support.  Then we have
$$a(x,\xi) = \int_\eta e^{ix\cdot \eta} \hat{c}(x,\eta,\eta+\xi) \, d\eta$$
where $\hat{c}(x,\eta,\eta+\xi)$ denotes the Fourier transform of $c(x,y,\eta+\xi)$ in the $y$-variable.  By Taylor's formula,
$$\hat{c}(x,\eta,\eta+\xi) = \sum_{|\alpha|\leq N-1} \frac{1}{\alpha!} \partial_\xi^\alpha \hat{c}(x,\eta,\xi) \eta^\alpha + R_N(x,\eta, \xi)$$
Plugging the first piece into \eqref{PE:10}, we get
$$\int_\eta e^{ix\cdot \eta}  \sum_{|\alpha|\leq N-1} \frac{1}{\alpha!} \partial_\xi^\alpha \hat{c}(x,\eta,\xi) \eta^\alpha \, d\eta = \sum_{|\alpha|\leq N-1} \frac{i^{|\alpha|}}{\alpha!} \partial_\xi^\alpha \partial_y^\alpha c(x,y,\xi) \big|_{y=x}$$
We next estimate the remainder:
\begin{align*}
|R_N(x,\eta,\xi)| &\leq \frac{1}{(N-1)!} \sum_{|\alpha|=N} \sup_{0\leq t \leq 1} |\partial_\xi^\alpha \hat{c}(x,\eta, \xi+t\eta)| |\eta|^N \\
& \leq 
\begin{cases}
C |\eta|^N (1+|\eta|)^{-M} (1+|\xi|)^{m-N} & \text{if } |\xi|\geq 2|\eta|\\
C |\eta|^N(1+|\eta|)^{-M} & \text{if } |\xi| \leq 2|\eta|
\end{cases}
\end{align*}
for $M$ large.  Plugging this into \eqref{PE:10} gives the desired estimate, completing the proof in the case where $c(x,y,\xi)$ has compact $y$-support.   For the case where $c(x,y,\xi)$ does not have compact $y$-support, use the method explained in the solution to Problem \chapterref{P:1.5}.\ref{P:1.5}.\\

\noindent \textbf{Problem 2}.  (Taken from \cite{St2}, p. 259).   We first compute $\Psi_a^\ast$.
$$\int_x \Psi_af(x) \overline{g(x)} \, dx = \int_\xi \hat{f}(\xi) \overline{ \int_x e^{-ix\cdot \xi} \overline{a(x,\xi)} g(x) \, dx} \, d\xi$$
and so
$$\widehat{\Psi_a^\ast g}(\xi) = \int_y e^{-iy\cdot \xi}\overline{a(y,\xi)}g(y)\, dy$$
and therefore
$$\Psi_a^\ast g(x) = \int_y \int_\xi e^{i(x-y)\cdot \xi} \overline{a(y,\xi)} g(y) \, dy \, d\xi$$
Set $c(x,y,\xi) = \overline{a(y,\xi)}$, and we get by (1b) that $\exists \; a^\ast(x,\xi)$ such that
$$a^\ast(x,\xi) - \sum_{|\alpha|\leq N-1} \frac{i^{-|\alpha|}}{\alpha!} \partial_\xi^\alpha \partial_x^\alpha \bar{a}(x,\xi) \in S^{m-N}$$\\

\noindent \textbf{Problem 3}.  (Taken from \cite{T2}).  \\
\textit{Step 1}.  Replace $a(x,\xi)$ by $(1+|\xi|)^{-m/2} a(x,\xi) (1+|\xi|)^{-m/2}$ so that we may assume w.l.o.g.\ that $m=0$.  \\
\textit{Step 2}.  Suppose $\text{Re }a(x,\xi)\geq C$.  Set 
$$b(x,\xi) = (\text{Re }a(x,\xi) - \tfrac{1}{2}C)^{1/2} \in S^0$$
Then 
$$\Psi_b^\ast \Psi_b = \tfrac{1}{2} (\Psi_a + \Psi_a^\ast) - \tfrac{1}{2}CI + E$$
where $E$ has symbol in $S^{-1}$, by the Kohn-Nirenberg calculus.  This gives
$$\| \Psi_b u\|_{L^2}^2  = \text{Re }\langle \Psi_a u, u \rangle - \tfrac{1}{2}C \| u \|_{L^2}^2 + \langle Eu, u \rangle$$
which gives 
$$ \text{Re }\langle \Psi_a u, u \rangle \geq \tfrac{1}{2}C\|u\|_{L^2}^2 - \langle E u, u \rangle $$
When then estimate 
$$\langle E u, u \rangle \leq \frac{1}{\epsilon}\|Eu\|_{L^2}^2 + \epsilon \|u\|_{L^2}^2 \leq \frac{C}{\epsilon} \|u\|_{H^{-1}}^2 + \epsilon \|u\|_{L^2}^2$$
Now we use an interpolation inequality to bound $\|u\|_{H^{-1}}^2$ in terms of a large constant times $\|u\|^2_{H^s}$ and a small constant times $\|u\|^2_{L^2}$.\\

\noindent \textbf{Problem 4a}.  
$$D_x^{1/2} e^{it\partial_x^2} u_0 = \int_\xi e^{ix\xi} |\xi|^{1/2} e^{-it\xi^2} \hat{u}_0(\xi) \, d\xi$$
Change variables to $\eta=-\xi^2$ (actually, we should consider $\xi \geq 0$ and $\xi\leq 0$ separately, but the proof is written assuming $\xi\geq 0$), to obtain
$$D_x^{1/2} e^{it\partial_x^2} u_0 = c\int_\eta e^{it\eta} e^{ix(-\eta)^{1/2}} \hat{u}_0((-\eta)^{1/2}) (-\eta)^{-1/4} \, d\eta$$
Apply the $L_t^2$ norm to both sides, and apply Plancherel in $t$.  Then change variables back to $\xi=(-\eta)^{1/2}$.\\

\noindent \textbf{Problem 4b}.  (The following solution is drawn from \cite{KPV91a}, where a more general result appears as Theorem 4.1.)  We will consider only the case $n=2$, but the proof for general $n\geq 2$ is similar.  Divide the frequency space $\xi=(\xi_1, \xi_2)$ into four overlapping conical regions:  $\xi_i \geq \frac{1}{2}|\xi|$, $\xi_i \leq -\frac{1}{2}|\xi|$, $i=1,2$.  
$$D_x^{1/2}u= \iint_{\xi_1, \xi_2} |\xi|^{1/2}e^{ix_1\xi_1}e^{ix_2\xi_2}e^{-it|\xi|^2} \hat{u}_0(\xi_1, \xi_2) \, d\xi_1 d\xi_2 $$
It suffices to deduce the bound for each region separately.  We shall consider only the case $\xi_1 \geq \frac{1}{2}|\xi|$ (the other three are similar).  Hence, we need to show that 
$$\| D^{1/2}_x w(x,t) \|_{L^2(Q\times (-\infty,+\infty))}\leq c\|u_0\|_{L^2}$$
where
$$w(x,t) = \iint_{\substack{\xi_1, \xi_2 \\ \xi_1 \geq \frac{1}{2}|\xi|}} e^{ix_1\xi_1}e^{ix_2\xi_2}e^{-it|\xi|^2} |\xi|^{1/2} \hat{u}_0(\xi_1, \xi_2) \, d\xi_1 d\xi_2 $$
Make the change of variables $(\xi_1,\xi_2) \to (\xi_1^2+\xi_2^2, \xi_2)$.  The Jacobian is $2\xi_1$, so letting $r=\xi_1^2+\xi_2^2$, we have
$$w(x,t) = \int_r e^{-itr} \int_{\xi_2} e^{ix_1\sqrt{r-\xi_2^2}}e^{ix_2\xi_2}\frac{r^{1/4}}{2\sqrt{r-\xi_2^2}} \hat{u}_0(\sqrt{r-\xi_2^2}, \xi_2) \, d\xi_2\, dr $$
Applying Plancherel in $t$,
$$\|w(x,t)\|_{L^2_t}^2 = \int_{r=0}^\infty \left| \int_{\xi_2} e^{ix_2\xi_2} e^{ix_1\sqrt{r-\xi_2^2}}\frac{r^{1/4}}{2\sqrt{r-\xi_2^2}} \hat{u}_0(\sqrt{r-\xi_2^2}, \xi_2) \, d\xi_2 \right|^2 dr $$
If $Q=(\sigma_1,\sigma_1+1)\times(\sigma_2,\sigma_2+1)$, then
\begin{align*}
& \|w(x,t)\|_{L^2(Q\times (-\infty,+\infty))}^2 \\
& = \int_{x_1=\sigma_1}^{\sigma_1+1} \int_{r=0}^\infty \int_{x_2=\sigma_2}^{\sigma_2+1} \left| \int_{\xi_2} e^{ix_2\xi_2} e^{ix_1\sqrt{r-\xi_2^2}}\frac{r^{1/4}}{2\sqrt{r-\xi_2^2}} \hat{u}_0(\sqrt{r-\xi_2^2}, \xi_2) \, d\xi_2 \right|^2 dx_2 \, dr dx_1\\
& \leq \int_{x_1=\sigma_1}^{\sigma_1+1} \int_{r=0}^\infty \int_{\xi_2} \left|  e^{ix_1\sqrt{r-\xi_2^2}}\frac{r^{1/4}}{2\sqrt{r-\xi_2^2}} \hat{u}_0(\sqrt{r-\xi_2^2}, \xi_2) \, \right|^2 d\xi_2 \, dr dx_1  \qquad (\text{by Plancherel in }x_2)\\
& \leq  \int_{r=0}^\infty \int_{\xi_2}  \frac{1}{|\xi_1|} \left| \hat{u}_0(\sqrt{r-\xi_2^2}, \xi_2) \, \right|^2 d\xi_2 \, dr  \qquad (\text{since }r^{1/2}\leq 2|\xi_1| \text{ in this cone})\\
& = \iint_{\xi_1, \xi_2} | \hat{u}_0(\xi_1, \xi_2)|^2 \, d\xi_1 \, d\xi_2 \qquad (\text{changing variables back})
\end{align*}\\

\noindent \textbf{Problem 4c}. First show 
\begin{equation} \label{PE:11}
\left\| \int_{t'=-\infty}^{+\infty} D_x^{1/2}e^{-it'\partial_x^2}f(x,t') \, dt' \right\|_{L_x^2} \leq C\|f\|_{L_x^1L_t^2}
\end{equation}
To prove this, multiply the expression by $\bar{u}_0(x)$, integrate in $x$, apply H\"{o}lder, and use Problem (4a).  Next show
\begin{equation} \label{PE:12}
 \left\| D_x \int_{t'=-\infty}^{+\infty} e^{i(t-t')\partial_x^2} f(x,t') \, dt' \right\|_{L_x^\infty L_t^2} \leq C\|f\|_{L_x^1L_t^2}
\end{equation}
To prove this, write
\begin{align*}
&\int_x \int_t \left[ D_x \int_{t'} e^{i(t-t')\partial_x^2} f(x,t') \, dt' \right] \bar{g}(x,t) \, dx dt \\
&= \int_x \left[ \int_{t'} D_x^{1/2}e^{-it'\partial_x^2}f(x,t') \, dt' \right] \left[ \overline{\int_t D_x^{1/2} e^{-it\partial_x^2}g(x,t) \, dt} \right] \, dx
\end{align*}
apply Cauchy-Schwarz, and \eqref{PE:11}.  By writing $\chi_{(0,T)}(t') = 1-2\chi_{(-\infty,0)}(t')+\text{sgn }(t-t')$, and using that $\widehat{\text{sgn }t'}(\tau) = \text{pv } \dfrac{1}{\tau}$, deduce the identity, (for $g(x,t) = \sum_{j=1}^N g_j(x)h_j(t)$, with $g_j\in C^\infty_0$, $h_j\in C^\infty_0$, i.e. the tensor product space)
\begin{align*}
&\int_0^t e^{i(t-t')\partial_x^2}g(x,t') \, dt' \\
&= 
\begin{aligned}[t]
&\lim_{\epsilon \to 0^+} \iint_{|\tau-\xi^2|>\epsilon} e^{i(x\xi+t\tau)} \frac{\hat{g}(\xi,\tau)}{\tau-\xi^2} \, d\xi d\tau  + \int_{-\infty}^{+\infty} e^{i(t-t')\partial_x^2} g(x,t') \, dt' \\
& - 2 \int_{-\infty}^0 e^{i(t-t')\partial_x^2} g(x,t') \, dt'
\end{aligned}
\end{align*}
(We assumed that $g$ was in the tensor product space so that we could transfer the principal value from the $\tau$ integral to the $\xi$ integral, i.e.\ pull the $\lim_{\epsilon \to 0^+}$ out of the $\xi$ integral by dominated convergence).
The above work handles the last two terms on the right-hand side of this expression.  To complete the proof, we show
$$\left\| \int_\tau \int_\xi e^{it\tau}e^{ix\xi} \frac{|\xi|\hat{f}(\xi,\tau)}{\tau-\xi^2} \, d\xi d\tau \right\|_{L_x^\infty L_t^2} \leq C\|f\|_{L_x^1L_t^2}$$
This is done by first applying Plancherel in $t$ to obtain
$$\left\| \int_\tau \int_\xi e^{it\tau}e^{ix\xi} \frac{|\xi|\hat{f}(\xi,\tau)}{\tau-\xi^2} \, d\xi d\tau \right\|_{ L_t^2}^2 \leq \int_\tau \left| \int_\xi e^{ix\xi} \frac{|\xi|\hat{f}(\xi,\tau)}{\tau-\xi^2} \, d\xi \right|^2 d\tau $$
Let 
$$K(x,\tau) = \int_\xi e^{ix\xi} \frac{\xi}{\tau-\xi^2} \, d\xi$$
We now explain how to show that $|K(x,\tau)| \leq C$ (some details omitted).  If $\tau\leq 0$, 
\begin{equation}\label{PE:300}
K(x,\tau) = \int e^{ix(-\tau)^{1/2}\xi} \frac{\xi}{1-\xi^2} \, d\xi
\end{equation}
Use 
$$\frac{2\xi}{1-\xi^2}=\frac{1}{1-\xi}-\frac{1}{1+\xi}$$
which shows that \eqref{PE:300} is a difference of two phase shifted sgn functions.  If $\tau\geq 0$, then
\begin{equation}\label{PE:301}
K(x,\tau) = \int e^{ix\tau^{1/2}\xi}\frac{\xi}{1+\xi^2} \, d\xi
\end{equation}
For $|\xi|\geq 1$ (use a smooth cut-off), use
\begin{equation}\label{PE:302}
\frac{\xi}{1+\xi^2} = \frac{1}{\xi}-\frac{1}{\xi(1+\xi^2)}
\end{equation}
Since the second piece is (away from $0$) in $L^1$, when \eqref{PE:302} is substituted into \eqref{PE:301}, it gives the difference of a sgn function and another bounded function.  We can now complete the proof.  By Minkowskii's integral inequality and Plancherel in $\tau$, (where $\sphat$ represents Fourier transform in $t$ variable only)
$$\left( \int_\tau \left| \int_y K(x-y,\tau) \hat{f}(y,\tau) \, dy \right|^2 \, d\tau \right)^{1/2} \leq C \int_y \left( \int_\tau |\hat{f}(y,\tau)|^2 \, d\tau \right)^{1/2} \, dy = C \|f\|_{L_x^1L_t^2}$$

We now proceed to the $n\geq 2$ case.  First show
\begin{equation} \label{PE:13}
\left\| \int_{t'} D_x^{1/2} e^{-it'\Delta}f(x,t') \, dt' \right\|_{L^2(\mathbb{R}^n)} \leq \sum_\alpha \|f\|_{L^2(Q_\alpha\times (-\infty, +\infty))} 
\end{equation}
To show this, take $u_0\in L^2(\mathbb{R}^n)$, and compute
\begin{align*}
& \int_{x}\int_{t'} D_x^{1/2}e^{-it'\Delta}f(x,t')\, dt' \; \overline{u_0(x)} \, dx \\
&= \int_{x}\int_{t'} f(x,t') \; \overline{D_x^{1/2}e^{it'\Delta}u_0(x)} \, dx \, dt' \\
& \leq \sum_\alpha \|f\|_{L^2(Q_\alpha\times (-\infty,+\infty))} \|D_x^{1/2} e^{it\Delta} u_0 \|_{L^2(Q_\alpha \times (-\infty, +\infty))} \\
& \leq \sum_\alpha \|f \|_{L^2(Q_\alpha\times (-\infty,+\infty))} \|u_0\|_{L^2}
\end{align*}
The next step is to show that if $Q$ is a dyadic cube,
\begin{equation} \label{PE:14}
\left\| \int_{t'} D_x e^{i(t-t')\Delta} f(x,t') \, dt' \right\|_{L^2(Q\times (-\infty, +\infty))} \leq \sum_\alpha \|f\|_{L^2(Q_\alpha \times (-\infty, +\infty))}
\end{equation}
To show this, let $g(x,t) \in L^2(Q\times (-\infty, +\infty))$ and compute:
\begin{align*}
&\int_x \int_t \left[\int_{t'} D_x e^{i(t-t')\Delta} f(x,t') \, dt'\right] \; \overline{g(x,t)} \, dx \, dt \\
& = \int_x \int_{t'} D_x^{1/2}e^{-it'\Delta}f(x,t') \, dt' \; \overline{ \int_{t} D_x^{1/2} e^{-it\Delta}g(x,t) \, dt} \, dx\\
& \leq \left\| \int_{t'} D_x^{1/2}e^{-it'\Delta}f(x,t') \, dt' \right\|_{L^2(\mathbb{R}^n)} \left\| \int_{t} D_x^{1/2} e^{-it\Delta}g(x,t) \, dt \right\|_{L^2(\mathbb{R}^n)} \\
&\leq \sum_\alpha \|f\|_{L^2(Q_\alpha\times (-\infty,+\infty))} \sum_\alpha \|g\|_{L^2(Q_\alpha\times (-\infty,+\infty))} \\
&\leq \sum_\alpha \|f\|_{L^2(Q_\alpha\times (-\infty,+\infty))} \|g\|_{L^2(Q \times (-\infty,+\infty))} 
\end{align*}
Take the sup over all $Q$ in \eqref{PE:14}.  The remainder of the proof is taken from \cite{KPV93a}.  Again, we appeal to the fact
\begin{align*}
&\int_0^t e^{i(t-t')\Delta}f(x,t') \, dt' \\
&= 
\begin{aligned}[t]
&\lim_{\epsilon \to 0^+} \iint_{|\tau-|\xi|^2|>\epsilon} e^{i(x\cdot \xi+t\tau)} \frac{\hat{f}(\xi,\tau)}{\tau-|\xi|^2} \, d\xi d\tau  + \int_{-\infty}^{+\infty} e^{i(t-t')\Delta} f(x,t') \, dt' \\
& - 2 \int_{-\infty}^0 e^{i(t-t')\Delta} f(x,t') \, dt'
\end{aligned}
\end{align*}
It remains to treat the term
$$\iint e^{it\tau} e^{ix\cdot \xi} \frac{\xi}{\tau-|\xi|^2} \hat{f}(\xi,\tau) \, d\xi \, d\tau$$
Let $f_\alpha = f\chi_{Q_\alpha}$, and set
$$u_\alpha(x,t) = \iint e^{ix\cdot \xi}e^{it\tau} \frac{\xi\hat{f}_\alpha(\xi,\tau)}{|\xi|^2-\tau} \, d\xi \, d\tau$$
so that
$$u(x,t) \sum_\alpha u_\alpha(x,t)$$
We shall need the following lemma
\begin{lemma}[\cite{KPV93a}]If $$\widehat{Th}(\xi) = \frac{\xi}{|\xi|^2-1} \hat{h}(\xi)$$ then $\forall \; r>0$,
$$\sup_{\alpha \in \mathbb{Z}^n} \left( \int_{rQ_\alpha} |T(g\chi_{rQ_\beta})|^2 \, dx \right)^{1/2} \leq c r \left( \int_{rQ_\beta} |g|^2 \, dx \right)^{1/2}$$
with $c$ independent of $r$.
\end{lemma}
For the proof, see \cite{KPV93a}.  By the triangle inequality  and Plancherel in $t$,
\begin{align}
&\sup_\alpha \left( \int_{Q_\alpha} \int_t |u(x,t)|^2 \, dt \, dx \right)^{1/2} \notag\\
& \leq \sup_\alpha \sum_\beta \left( \int_{x\in Q_\alpha} \int_\tau \left| \int_\xi e^{ix\cdot \xi} \frac{\xi\hat{f}_\beta(\xi,\tau)}{|\xi|^2 - \tau} \, d\xi \right|^2 \, d\tau dx \right)^{1/2} \notag\\
& \leq \sum_\beta \left( \int_\tau \left[ \sup_{Q_\alpha} \int_{x\in Q_\alpha}  \left| \int_\xi e^{ix\cdot \xi} \frac{\xi \hat{f}_\beta(\xi,\tau)}{|\xi|^2 - \tau} \, d\xi \right|^2 \, dx \right] d\tau \right)^{1/2} \label{PE:205}
\end{align}
Assume $\tau>0$, and we shall examine the piece 
$$\sup_{\alpha} \int_{x\in Q_\alpha}  \left| \int_\xi e^{ix\cdot \xi} \frac{\xi\hat{f}_\beta(\xi,\tau)}{|\xi|^2 - \tau} \, d\xi \right|^2 \, dx $$
Change variables as  $\tau^{1/2}\eta=\xi$ and then change variables as $y=x\tau^{1/2}$, and this becomes
$$\tau^{\frac{n}{2}-1} \sup_\alpha \int_{y\in \tau^{1/2}Q_\alpha} \left| \int_{\eta\in \mathbb{R}^n} e^{iy\cdot \eta} \frac{ \eta \hat{f}_\beta(\tau^{1/2}\eta, \tau)}{|\eta|^2-1} \, d\eta \right|^2 \, dy$$
Let $\hat{g}(\eta) = \hat{f}_\beta(\tau^{1/2}\eta, \tau)$, so that 
$g(x)=\tau^{-n/2}\hat{f}_\beta(x\tau^{-1/2},\tau)$ where $\sphat$ denotes the Fourier transform in the $t$ variable only.  Therefore, $g=g\chi_{\tau^{1/2}Q_\beta}$, and the above is
$$ \tau^{\frac{n}{2}-1} \sup_\alpha \int_{y \in \tau^{1/2}Q_\alpha} |T(g\chi_{\tau^{1/2}Q_\beta})|^2 \, dy $$
and by the lemma, this is bounded by
$$\tau^{\frac{n}{2}} \int_x |g|^2 \, dx$$
and by Plancherel this is
$$\tau^{\frac{n}{2}} \int_\eta |\hat{f}_\beta(\tau^{1/2}\eta,\tau)|^2 \, d\eta$$and changing variables back we have
$$\int_\xi |\hat{f}_\beta(\xi,\tau)|^2 \, d\xi$$
Plugging into \eqref{PE:205}, we have
$$\eqref{PE:205} \leq c \sum_\beta \left( \int_\tau \int_\xi |\hat{f}_\beta(\xi,\tau)|^2 \, d\xi\, d\tau\right)^{1/2}$$
and we complete the argument using Plancherel.  The case where $\tau < 0$ corresponds to the multiplier $\frac{\eta}{|\eta|^2+1}$, which is easier to handle since there is no singularity.\\

\noindent \textbf{Problem 5}.  Let $Q$ be the cube centered at the origin.  Since
$$\|g \|_{L^2(Q\times [0,T])} \leq c\|g\|_{L^2(\lambda_m(x)dx \times [0,T])}$$
we have, by the proof given in the lecture notes,
$$\|v\|_{L_T^\infty L_x^2}^2 + c_1 \|J^{1/2}v \|_{L^2(Q\times [0,T])}^2 \leq \|v_0\|_{L_x^2}^2 + 2 \int_0^T \left| \int_x F \cdot \bar{v} \, dx \right| \, dt$$
We estimate:
\begin{align*}
\int_0^T\left| \int_x F \cdot \bar{v} \, dx \right|\, dt &= \int_0^T \left| \int_x J^{-1/2}F \cdot \overline{J^{1/2}v} \, dx \right|\, dt \\
& \leq \tv J^{1/2}v \tv_T \; \tv J^{-1/2}F \tv_T' \\
& \leq \frac{c_1}{2} \tv J^{1/2}v \tv_T^2 + \frac{2}{c_1} \tv J^{-1/2}F \tv_T'^2\end{align*}
and then convert back to $u$, $f$ to obtain the estimate
$$\|u\|_{L_T^\infty L_x^2}^2 + c_1 \|J^{1/2}u \|_{L^2(Q\times [0,T])}^2 \leq C\|u_0\|_{L_x^2}^2+ \frac{c_1}{2}\tv J^{1/2} u \tv_T^2 + C \tv J^{-1/2} f \tv_T'^2$$
Let $x_0$ be the center of the cube $Q_\alpha$.  Since $u$ solves
$$\left\{
\begin{aligned}
&\partial_t u = i \Delta u +f \\
&u\big|_{t=0} = u_0 
\end{aligned}
\right.
$$
if we set $w(x)=u(x+x_0)$, $g(x)=f(x+x_0)$, $w_0(x)=u_0(x+x_0)$, then $w$ solves$$\left\{
\begin{aligned}
&\partial_t w = i \Delta w +g \\
&w\big|_{t=0} = w_0 
\end{aligned}
\right.
$$
Hence we have the estimate
$$\|w\|_{L_T^\infty L_x^2}^2 + c_1 \|J^{1/2}w \|_{L^2(Q\times [0,T])}^2 \leq C\|w_0\|_{L_x^2}^2+ \frac{c_1}{2}\tv J^{1/2} w \tv_T^2 + C \tv J^{-1/2} g \tv_T'^2$$
which, when converted back to $u$ is
$$\|u\|_{L_T^\infty L_x^2}^2 + c_1 \|J^{1/2}u \|_{L^2(Q_\alpha\times [0,T])}^2 \leq C\|u_0\|_{L_x^2}^2+\frac{c_1}{2}\tv J^{1/2} u \tv_T^2 + C \tv J^{-1/2} f \tv_T'^2$$
Take the sup over $\alpha$ to obtain the result.

\section{Solutions to Lecture 3 problems}

\textbf{Problem 1a}.  By dualizing, we see the given estimate is equivalent to
\begin{equation} \label{PE:20}
\left\| \int_t D_x^{-1/4} e^{-it\partial_x^2} g(x,t) \, dt \right\|_{L_x^2} \leq C\|g\|_{L_x^{4/3}L_t^1}
\end{equation}
To prove \eqref{PE:20}, write:
\begin{align*}
&\left\| \int_t D_x^{-1/4} e^{-it\partial_x^2} g(x,t) \, dt \right\|_{L_x^2}^2 \\
& =\int_x \int_t D_x^{-1/4} e^{-it\partial_x^2} g(x,t) \, dt \overline{ \int_s D_x^{-1/4} e^{-is\partial_x^2} g(x,s) \, ds} \, dx \\
& \leq \int_x \int_t g(x,t) \overline{ \int_s D_x^{-1/2} e^{i(t-s)\partial_x^2} g(x,s) \, ds } \, dx\, dt\\
& \leq \|g\|_{L_x^{4/3}L_t^1}\left\|  \int_s D_x^{-1/2} e^{i(t-s)\partial_x^2} g(x,s) \, ds \right\|_{L_x^4L_t^\infty} 
\end{align*}
so it suffices to show
\begin{equation} \label{PE:22}
\left\| \int_s D_x^{-1/2}e^{i(t-s)\partial_x^2}g(x,s) \, ds \right\|_{L_x^4L_t^\infty} \leq C\|g\|_{L_x^{4/3}L_t^1}
\end{equation}
Let
$$K(x,t) = \int_\xi e^{ix\xi} e^{it\xi^2} |\xi|^{-1/2} \, d\xi = (\text{sgn }x)|x|^{-1/2}\int_\xi e^{i\xi}e^{i\frac{t}{x^2}\xi^2} |\xi|^{-1/2} \, d\xi$$
We shall show, by the Van der Corput Lemma, that 
\begin{equation} \label{PE:21}
|K(x,t)| \leq c|x|^{-1/2}
\end{equation}
\begin{lemma}[Van der Corput] If $\phi\in C^2(\mathbb{R})$ is real valued, $\phi''\geq 1$ on $(a,b)$, then
$$ \left| \int_a^b e^{i\lambda \phi(\xi)} \psi(\xi) \, d\xi \right| \leq \frac{c}{|\lambda|^{1/2}} \left( |\psi(b)| + \int_a^b |\psi'(x)| \, dx \right)$$
\end{lemma}
For a proof see \cite{St2}, pp. 332-334.  We want to show
$$\int e^{i(\xi+t\xi^2)}|\xi|^{-1/2} \, d\xi$$
is bounded (independent of $t$).  First, put $\psi_0(\xi)+\psi_1(\xi)=1$, where $\psi_0(\xi)=1$ for $|\xi|\leq 1$ and $\text{supp }\psi_0 \subset [-2,2]$.  Clearly
$$\int e^{i(\xi+t\xi^2)}|\xi|^{-1/2} \psi_0(\xi)\, d\xi$$
is bounded, and it remains to show
\begin{equation} \label{PE:23}
\int e^{i(\xi+t\xi^2)}|\xi|^{-1/2} \psi_1(\xi)\, d\xi
\end{equation}
is bounded.  Let $\phi_t(\xi)=\xi+t\xi^2$.  Then $\phi_t'(\xi)=1+2t\xi$.  If $|\phi_t'(\xi)|\geq \frac{1}{2}$, then we can prove \eqref{PE:23} is bounded using integration by parts.  If $|\phi_t'(\xi)|\leq \frac{1}{2}$, then $2t\xi\sim -1$, and we write
$$\eqref{PE:23}= t^{1/2}\int e^{it(\frac{\xi}{t}+\xi^2)} \psi_1(\xi) |t\xi|^{-1/2} \, d\xi $$
and apply the Van der Corput lemma with $\phi(\xi)=\frac{\xi}{t}+\xi^2$.  We now complete the proof.  
\begin{align*}
&\int_s D_x^{-1/2} e^{i(t-s)\partial_x^2}g(x,s) \, ds \\
&= \int_s \int_\xi e^{ix\xi} |\xi|^{-1/2} e^{-i(t-s)\xi^2} \hat{g}(\xi,s) \, ds \\
&= \int_s \int_y K(x-y,s-t)g(y,s) \, dy \, ds
\end{align*}
By \eqref{PE:21},
$$\left| \int_\xi D_x^{-1/2}e^{i(t-s)\partial_x^2}g(x,s), ds \right| \leq c \int_y |x-y|^{-1/2} \left( \int_\xi |g(y,s)|\, ds \right) \, dy$$
Apply the $L_x^4$ norm to both sides and use the theorem on fractional integration, see \cite{St1}, pp. 119-121. \\

\noindent \textbf{Problem 1b}. We will need the following lemma.
\begin{lemma}[Prop 2.6 in \cite{KPV91}]  Let $\psi(\xi)\in C^\infty(\mathbb{R})$, $\text{supp }\psi(\xi) \subset [2^{k-1},2^{k+1}]$, $k \in \mathbb{N}$.  Then, for $|t|\leq 2$
\begin{equation}\label{PE:200}
\left| \int_{-\infty}^{+\infty} e^{it\xi^2}e^{ix\xi} \psi(\xi) \, d\xi \right| \leq cH_k(x)
\end{equation}
where
$$H_k(x) =
\left\{\begin{aligned}
&2^k && \text{if }|x|\leq 1 \\
&2^{k/2}|x|^{-1/2} && \text{if }1\leq |x|\leq c2^k \\
&(1+x^2)^{-1} && \text{if }|x|\geq c2^k
\end{aligned}
\right.
$$
where $c$ in \eqref{PE:200} is independent of $t$, $k$.
\end{lemma}
\begin{proof}
The bound of $2^k$ is immediate for all $x$, since $\text{supp }\psi \subset [2^{k-1},2^{k+1}]$.  In the case $|x|\leq 1$, we can do no better.  Suppose $|x|\geq 1$, and assume $0\leq t \leq 2$ to simplify the exposition.  Let $\phi_x(\xi) = t\xi^2 +x\xi$, and note that $\phi_x'(\xi) = 2t\xi + x$ and $\phi_x''(\xi) = 2t$.  Let
$$\Omega = \left\{\; \xi \in \text{supp }\psi \; \Big| \; |2t\xi + x| \leq \frac{|x|}{2} \; \right\}$$
and choose $\varphi(\xi) \in C^\infty$ supported in $\Omega$ such that $\varphi(\xi)=1$ when $|2t\xi + x | \leq \frac{|x|}{3}$.  If $\xi \in \Omega$, then 
\begin{equation} \label{PE:201}
\frac{|x|}{2} \leq 2t|\xi| \leq \frac{3|x|}{2}
\end{equation}
 and thus $|x| \leq 2^{k+4}$, so we are in the second region in the definition of $H_k(x)$.  Also, \eqref{PE:201} implies $2^{-k-2}|x| \leq \phi_x''(\xi) \leq 2^{-k+2}|x|$, and therefore, by the van der Corput lemma,
$$\left| \int e^{i\phi_x(\xi)} \varphi(\xi) \psi(\xi) \, d\xi \right| \leq c2^{k/2}|x|^{-1/2}$$
Now suppose $\xi \in \text{supp }(1-\varphi(\xi))$, so that $|\phi_x'(\xi)| \geq \frac{|x|}{3}$.  Note that $x$ may still lie in either the second region or the third region in the definition of $H_k(x)$.  Applying integration by parts twice,
$$\int e^{it\xi^2}e^{ix\xi} (1-\varphi(\xi))\psi(\xi) \, d\xi = -\int e^{i\phi_x(\xi)} \partial_\xi \frac{1}{\phi_x'(\xi)} \partial_\xi \left[ \frac{(1-\varphi(\xi))\psi(\xi)}{\phi_x'(\xi)} \right] \, d\xi$$
from which we can deduce the necessary bound.
\end{proof}
Now we turn to the solution of Problem 1b, using this lemma.  Let $\psi_k(\xi)$, $k=0,1,2, \ldots$ be smooth functions, $\leq 1$, such that,
$$\text{supp }\psi_k(\xi) \subset [2^{k-1},2^k], \; k=1,2,\ldots, \qquad  \text{supp }\psi_0(\xi) \subset [-2,2], \qquad \sum_{k=0}^{+\infty}\psi_k(\xi) = 1 \; \; \forall \; \xi\in\mathbb{R}^+  $$
We claim that to prove 
\begin{equation}\label{PE:210}
\left\| e^{it\partial_x^2}u_0 \right\|_{L^2_xL^\infty_{t\in [-1,1]}} \leq c\|u_0\|_{H^s}
\end{equation}
for $s>\frac{1}{2}$, it suffices to prove, for each $k=0,1,\ldots$ that
\begin{equation}\label{PE:211}
\left\| \int_\xi e^{ix\xi}e^{it\xi^2}\psi_k(|\xi|)v(\xi) \, d\xi\right\|_{L^2_xL^\infty_{t\in [-1,1]}} \leq c 2^{k/2}\|v\|_{L^2}
\end{equation}
Indeed, suppose \eqref{PE:211}.  Then (letting $\widehat{E_kv}(\xi) = \chi_{2^{k-1}\leq |\xi|\leq 2^{k+1}}(\xi)\hat{v}(\xi)$),
\begin{align*}
&\left\| \int_\xi e^{ix\xi} e^{it\xi^2} \hat{u_0}(\xi) \, d\xi \right\|_{L^2_xL^\infty_{t\in [-1,1]}} \\
& =  \left\| \sum_k \int_\xi e^{ix\xi} e^{it\xi^2} \psi_k(|\xi|) \widehat{E_ku_0}(\xi) d\xi \right\|_{L^2_xL^\infty_{t\in [-1,1]}} \\
& \leq \sum_k \left\| \int_\xi e^{ix\xi} e^{it\xi^2} \psi_k(|\xi|) \widehat{E_ku_0}(\xi) d\xi \right\|_{L^2_xL^\infty_{t\in [-1,1]}} \qquad \text{by triangle inequality}\\
& \leq \sum_k 2^{k/2} \|E_ku_0\|_{L^2} \qquad \text{by \eqref{PE:211}}\\
& \leq c_\epsilon \left( \sum_k 2^{k(1+\epsilon)} \|E_ku_0\|_{L^2}^2 \right)^{1/2}\qquad \text{where $\epsilon = 2s-1$, by Cauchy-Schwarz}\\
& \leq c_\epsilon \|u_0\|_{H^s}
\end{align*}
which is \eqref{PE:210}.  To prove \eqref{PE:211}, it suffices to prove, by duality, that (with $\widehat{F_kv}(\xi)=\psi_k(\xi)\hat{v}(\xi)$), 
\begin{equation}\label{PE:212}
\left\| \int_{t=-1}^{t=1} e^{it\partial_x^2} F_kg(x,t) \, dt \right\|_{L^2_x} \leq 2^{k/2}\|g\|_{L_x^2L_{t\in [-1,1]}^1}
\end{equation}
By a $T^\ast T$ argument, to prove \eqref{PE:212} it suffices to prove
\begin{equation}\label{PE:213}
\left\| \int_{s=-1}^1 e^{i(t-s)\partial_x^2} F_kg(x,s) \, ds \right\|_{L_x^2L_{t\in [-1,1]}^\infty} \leq c2^k \|g\|_{L_x^2L_{t\in [-1,1]}^1}
\end{equation}
Let 
$$K_k(x,t) = \int_\xi e^{ix\xi}e^{it\xi^2} \psi_k(\xi) \, d\xi $$
By the lemma, $|K_k(x,t)| \leq cH_k(x)$. We have
$$
\int_{s=-1}^1 e^{i(t-s)\partial_x^2} F_kg(x,s) \, ds = \int_{s=-1}^1 \int_y K_k(x-y,t-s)g(y,s) \, dy \, ds
$$
and hence
$$\left\| \int_{s=-1}^1 e^{i(t-s)\partial_x^2} F_kg(x,s) \, ds \right\|_{L_{t\in [-1,1]}^\infty} \leq \int_y H_k(x-y) \int_{s=-1}^1 |g(y,s)| \, ds \, dy$$
and therefore
$$\left\| \int_{s=-1}^1 e^{i(t-s)\partial_x^2} F_kg(x,s) \, ds \right\|_{L_x^2 L_{t\in [-1,1]}^\infty} \leq c \left( \int H_k(y) \, dy \right) \; \|g\|_{L_x^2L_{t\in [-1,1]}^1}$$
and one can verify that $\int H_k(y) \, dy \leq c2^k$.\\

\noindent \textbf{Problem 1c}.  By dualizing, we can show that the given estimate is equivalent to
\begin{equation} \label{PE:24}
\left\| \int_{t'} e^{it'\partial_x^2} f(x,t') \, dt' \right\|_{L_x^2} \leq c \|f\|_{L_t^{4/3}L_x^1}
\end{equation}
To prove \eqref{PE:24}, write
\begin{align*}
&\left\| \int_{t'} e^{it'\partial_x^2} f(x,t') \, dt' \right\|_{L_x^2}^2 \\
& =\int_x \int_t e^{it\partial_x^2} f(x,t) \, dt  \overline{ \int_{s} e^{is\partial_x^2} f(x,s) \, ds} \, dx \\
& = \int_x \int_t f(x,t) \overline{ \int_s e^{i(t-s)\partial_x^2} f(x,s) \, ds} \, dx\, dt\\
& \leq \|f\|_{L_t^{4/3}L_x^1} \left\| \int_s e^{i(t-s)\partial_x^2} f(x,s) \, ds \right\|_{L_t^4L_x^\infty}
\end{align*}
and thus it suffices to show that
$$\left\| \int_s e^{i(t-s)\partial_x^2} f(x,s) \, ds \right\|_{L_t^4L_x^\infty} \leq C\|f\|_{L_t^{4/3}L_x^1}$$
We have
$$\widehat{ \frac{1}{t^{1/2}} e^{i\frac{x^2}{t}}}(\xi) = e^{-it\xi^2}$$
so
$$\int_s e^{i(t-s)\partial_x^2} f(x,s) \, ds = \int_s \int_y \frac{1}{(t-s)^{1/2}} e^{i\frac{(x-y)^2}{t-s}} f(y,s) \, dy \, ds$$
Apply the theorem on fractional integration.\\

\noindent \textbf{Problem 3}.  We have
$$u\big|_{\gamma=0}(t) = 0 \qquad \partial_xu\big|_{\gamma=0}(t) = 0 \qquad \partial_x^2u\big|_{\gamma=0}(t) = 0$$
Applying $\partial_\gamma$ to the integral equation, we get
$$\partial_\gamma u(t) = e^{it\partial_x^2} u_0(x) + \int_0^t e^{i(t-t')\partial_x^2} \left[ \partial_\gamma u \partial_x u + u \partial_\gamma \partial_x u \right] \, dt'$$
which gives
$$\partial_\gamma u\big|_{\gamma=0}(t) = e^{it\partial_x^2}u_0(x) \qquad \partial_\gamma \partial_x u \big|_{\gamma=0}(t) = \partial_x e^{it\partial_x^2} u_0(x)$$
Applying $\partial_\gamma^2$ to the equation gives
$$\partial_\gamma^2 u (t) =\int_0^t e^{i(t-t')\partial_x^2} \left[ \partial_\gamma^2 u \partial_x u + 2 \partial_\gamma u  \partial_\gamma \partial_xu + u \partial_{\gamma}^2 \partial_x u \right] \, dt'$$
which gives
$$\partial_\gamma^2 u \big|_{\gamma=0} = 2 \int_0^t e^{i(t-t')\partial_x^2} \left[ e^{it'\partial_x^2}u_0 \partial_x e^{it'\partial_x^2}u_0 \right] \, dt'$$
Let $F: H^s \to H^s$ be the solution map.  Let $G: \mathbb{R} \to H^s$ be the map $G(\gamma)=\gamma u_0$.  Let $H: \mathbb{R} \to H^s$ be the composition map $H(\gamma) =F\circ G(\gamma)$.  Then $H'(\gamma)\in \mathcal{L}(\mathbb{R};H^s)$ is given by
$$H'(\gamma)(t) = \underbrace{DF(\gamma u_0)}_{\in \mathcal{L}(H^s;H^s)}(tu_0)$$
and $H''(\gamma)\in \mathcal{L}(\mathbb{R}\times \mathbb{R}; H^s)$ is given by
$$H''(\gamma)(t_1,t_2) = \underbrace{D^2F(\gamma u_0)}_{\in \mathcal{L}(H^s\times H^s ; H^s)}(t_1 u_0, t_2 u_0)$$
Setting $\gamma =0$, $t_1=1$, $t_2=1$, we get
$$\partial_\gamma^2 u \big|_{\gamma=0} = 2 \int_0^t e^{i(t-t')\partial_x^2} \left[ e^{it'\partial_x^2}u_0 \partial_x e^{it'\partial_x^2}u_0 \right] \, dt' = H''(0)(1,1) = D^2F(0)(u_0,u_0)$$
and by hypothesis, 
$$\|D^2F(0)(u_0,u_0)\|\leq C\|u_0\|^2$$
By writing $\chi_{(0,t)}(t')= \text{sgn }(t')+\text{sgn }(t-t')$, using $\widehat{\text{sgn }t'}(\tau) = \text{pv }\frac{1}{\tau}$, one can deduce the standard formula
$$\int_0^t e^{i(t-t')\partial_x^2}h(x,t') \, dt' = \int_\xi \int_\tau e^{ix\xi} \frac{e^{it\tau}-e^{it\xi^2}}{\tau-\xi^2}\hat{h}(\xi,\tau) \, d\tau, d\xi$$
If $h(x,t)=\partial_x [ S(t)u_0 ]^2$, then ($\sphat \;$ is the Fourier transform in $x$ only)
$$\hat{h}(\xi,t) = \xi \int_\xi e^{it(\xi-\xi_1)^2} \hat{u}_0(\xi-\xi_1) e^{it\xi_1^2}\hat{u}_0(\xi_1)\, d\xi_1$$
and so ($\sphat \;$ is the Fourier transform in both $x$ and $t$)
$$\hat{h}(\xi,\tau) = \xi \int_\xi \delta( (\xi-\xi_1)^2 + \xi_1^2-\tau) \hat{u}_0(\xi-\xi_1) \hat{u}_0(\xi_1) \, d\xi_1$$
and when plugged in we get
$$\int_0^t e^{i(t-t')\partial_x^2} \partial_x [S(t)u_0]^2 \, dt' = \int_\xi \int_{\xi_1} e^{i(x\xi+t\xi^2)}\xi \hat{u}_0(\xi_1) \hat{u}_0(\xi-\xi_1) \frac{ e^{it[\xi_1^2+(\xi-\xi_1)^2-\xi_2]}-1}{\xi_1^2+(\xi-\xi_1)^2-\xi_2} \, d\xi_1 \, d\xi$$
We now follow the suggested method in the lecture notes for showing this is a contradiction. Set $\hat{u}_0(\xi)= \alpha^{-1/2}\chi_{I_1}(\xi) + \alpha^{-1/2}N^{-s}\chi_{I_2}(\xi)$, where $I_1 = [\frac{1}{2}\alpha,\alpha]$, $I_2=[N,N+\alpha]$, where $\alpha << 1$ and $N>>1$. 
By Plancherel and the formula 
\begin{align}
 &\left\| \int_0^t e^{i(t-t')\partial_x^2} \left[ e^{it'\partial_x^2}u_0 \partial_x e^{it'\partial_x^2}u_0 \right] \, dt' \right\|_{H^s}^2 \notag\\
&=\int_{\xi=-\infty}^{+\infty} \left| \int_{\xi_1} (1+|\xi|)^s \xi \hat{u}_0(\xi_1)\hat{u}_0(\xi-\xi_1)\frac{e^{2it\xi_1(\xi-\xi_1)}-1}{2it\xi_1(\xi-\xi_1)} \, d\xi_1 \right|^2 \, d\xi\notag\\
&\geq \int_{\xi=N+\alpha}^{N+\frac{3\alpha}{2}} \left| q(\xi,\xi_1) \right|^2 \, d\xi \label{PE:30}
\end{align}
where 
$$q(\xi,\xi_1) = \int_{\xi_1} (1+|\xi|)^s \xi \hat{u}_0(\xi_1)\hat{u}_0(\xi-\xi_1)\frac{e^{2it\xi_1(\xi-\xi_1)}-1}{2it\xi_1(\xi-\xi_1)} \, d\xi_1$$
If $N+\alpha \leq \xi \leq N+\frac{3\alpha}{2}$, then
\begin{align*}
\frac{\alpha}{2} \leq \xi_1 \leq \alpha \quad &\Rightarrow \quad N \leq \xi-\xi_1 \leq N+\alpha\\
\frac{\alpha}{2} \leq \xi-\xi_1 \leq \alpha \quad &\Rightarrow \quad N \leq \xi_1 \leq N+\alpha
\end{align*}
In either case, $\xi_1(\xi-\xi_1)\sim N\alpha$, and we shall require that $N\alpha << 1$ so that 
$$\frac{e^{2it\xi_1(\xi-\xi_1)}-1}{2\xi_1(\xi-\xi_1)} \sim it$$
Also
\begin{equation}\label{PE:31}
\eqref{PE:30} \geq \int_{\xi=N+\alpha}^{N+\frac{3\alpha}{2}} \left| \int_{\xi_1=\frac{\alpha}{2}}^\alpha q(\xi,\xi_1) \, d\xi_1 + \int_{\xi_1=\xi-\alpha}^{\xi-\frac{\alpha}{2}} q(\xi,\xi_1) \, d\xi_1\right|^2 \, d\xi
\end{equation}
Now
$$\int_{\xi_1=\frac{\alpha}{2}}^\alpha q(\xi,\xi_1) \, d\xi_1 \sim \int_{\xi_1=\frac{\alpha}{2}}^\alpha N^s N \alpha^{-1/2} \alpha^{-1/2}N^{-s} i t \;d\xi_1 \sim itN$$
and similarly
$$\int_{\xi_1=\xi-\alpha}^{\xi-\frac{\alpha}{2}} q(\xi,\xi_1) \, d\xi_1 \sim itN$$
and thus
$$\eqref{PE:31} \geq \sim \alpha t^2 N^2$$
Take $\alpha$, $N$ such that $\alpha N << 1$ and $\alpha N^2 t^2 >> 1$.  We compute
$$\|u_0\|_{H^s}^2 \sim 1$$
thus obtaining a contradiction.\\

\noindent \textbf{Problem 4}.  Consider Case (1), i.e.\ $|\text{Im }\vec{b}_1(x)|\leq C\lambda_m(x)$.  We are considering the equation
$$\partial_t u = i\Delta u + \vec{b}_1(x) \cdot \nabla u + c_1(x) u + c_2(x) \bar{u} + f$$
The proof proceeds along lines similar to the proof of Theorem \ref{T:9}, except that we set $c_{M,R}(x,\xi) = \exp \left[ -M \theta_R(\xi) p(x,\xi) \right]$, where $M$ is to be chosen large in terms of $\text{Im }\vec{b}_1(x)$.  Let $v=\Psi_{c_{M,R}} u$.  Now $i[ \Psi_{c_{M,R}} \Delta - \Delta \Psi_{c_{M,R}}]$ has symbol
$$\leq -c_0 M \theta_R(\xi) \lambda_m(x) |\xi| c_{M,R}(x,\xi) + Mc_0^{-1}\theta_R(\xi) c_{M,R}(x,\xi)$$
Also,
\begin{align*}
&\Psi_{c_{M,R}}([ \text{Re }\vec{b}_1(x) + i \text{Im }\vec{b}_1(x)] \cdot \nabla) - ([\text{Re }\vec{b}_1(x) + i \text{Im }\vec{b}_1(x)] \cdot \nabla) \Psi_{c_{M,R}}\\
& =i[ \Psi_{c_{M,R}} (\text{Im }\vec{b}_1(x) \cdot \nabla) - (\text{Im }\vec{b}_1(x) \cdot \nabla) \Psi_{c_{M,R}}] + \text{skew-adjoint (mod order $0$)}
\end{align*}
This operator has symbol (mod order $0$)
$$-[\text{Im }\vec{b}_1(x) \cdot \xi] c_{M,R}(x,\xi) = -\theta_R(\xi)[\text{Im }\vec{b}_1(x) \cdot \xi] c_{M,R}(x,\xi) - (1-\theta_R(\xi))[\text{Im }\vec{b}_1(x) \cdot \xi] c_{M,R}(x,\xi)$$
 Examine the piece with symbol
\begin{align*}
&-\frac{c_0}{2}M\theta_R(\xi) \lambda_m(x) |\xi|c_{M,R}(x,\xi) -\theta_R(\xi)[\text{Im }\vec{b}_1(x) \cdot \xi] c_{M,R}(x,\xi) \\
&= \theta_R(\xi)c_{M,R}(x,\xi)|\xi|\left( -\frac{c_0}{2}M \lambda_m(x) - \vec{b}_1(x) \cdot \frac{\xi}{|\xi|}\right)\\
& \leq 0
\end{align*}
by taking $M$ sufficiently large.  Hence this term $\leq C\|v\|_{L^2}$ by the sharp G\r{a}rding inequality.

\section{Solutions to Lecture 4 problems}

\textbf{Problem 1}.  It was shown in the lecture notes that $\xi(t)$ satisfies
$$\lambda^{-2}|\xi_0|^2 \leq |\xi(t)|^2 \leq \lambda^2 |\xi_0|^2$$
The ellipticity condition
$$\lambda^{-1} |\xi|^2 \leq \sum_{j,k} a_{jk} \xi_j \xi_k \leq \lambda |\xi|^2$$
with $\xi= (0, \ldots, 0, \underset{j\text{th }}{1}, 0, \ldots 0)$, implies $\lambda^{-1} \leq a_{jj}(x) \leq \lambda$ (no sum) $\forall \; x$, $\forall \;j=1,\ldots, n$.  It also implies, with $\xi = (0, \ldots, 0, \underset{j\text{th}}{1}, 0, \ldots , 0, \underset{k\text{th}}{1}, 0,\ldots, 0)$, $j\neq k$, that $4\lambda^{-1} \leq a_{jk}+a_{kj}+a_{jj}+a_{kk}\leq 4\lambda$ (no sum),$\forall \; x$, $\forall \; j,k=1,\ldots, n$,  and hence
$$|a_{jk} + a_{kj}| \leq 4\lambda$$
We have
$$\dot{x}_j = \sum_k [ a_{jk}(x)+a_{kj}(x)] \xi_k$$
and thus
$$|\dot{x}_j| \leq c_\lambda |\xi_0|$$
The only way for a solution to fail to be globally defined is for it to ``blow-up'' by leaving every compact set  as $t\to t_0$ for some time $t_0$.\\

\noindent \textbf{Problem 2}.  The following solution is based on the more general results appearing as Lemma 1.3 and Lemma 3.2 in \cite{Do2}.  We shall assume that $a_{jk}$ is symmetric.  Let $\lambda_1^2 \equiv \lambda^{-2} \inf_{(x,\xi)\in K} |\xi|^2 >0$, and $\lambda_2^2 \equiv \lambda^2 \sup_{(x,\xi) \in K} |\xi|^2 < \infty$.  (Thus, $\lambda_1$, $\lambda_2$ are determined by $K$ and the ellipticity constant $\lambda$.)  By ellipticity, for each $(x_0,\xi_0)\in K$, we have that the flow $(X(t;x_0,\xi_0),\Xi(t;x_0,\xi_0))$ satisfies
$$\lambda_1^2 \leq |\Xi(t,x_0,\xi_0)|^2 \leq \lambda_2^2$$
Let $G\equiv \{ \; (x,\xi)\; | \;(x,\xi)\in \mathbb{R}^n\times [\lambda_1,\lambda_2] \; \}$, so that $(X(t;x_0,\xi_0),\Xi(t;x_0,\xi_0)) \in G$ for all $t$.
We shall construct $q(x,\xi)$ such that, for $(x,\xi)\in G$, we have 
\begin{equation} \label{PE:231}
H_h q \geq 1, \text{ and }|q| \leq C(1+|x|)
\end{equation}  
This $q$ will be constructed in two parts, $q_1$ and $q_2$, and we begin with the construction of $q_1$.  Let $\theta(x)=0$ for $x\leq M^2$, $\theta(x)=1$ for $x\geq (M+1)^2$, $\theta'(x)\geq 0$, $\theta(x)\geq 0$, where $M$ is to be chosen large (in terms of asymptotic flatness, etc.).  Let 
$$q_1(x,\xi) = \theta(|x|^2)H_h|x|^2$$
Then
$$H_h q_1 = \theta'(|x|^2) (H_h|x|^2)^2 + \theta(|x|^2) H_h^2|x|^2$$
We compute (with $a_{jk}=a_{jk}(x)$):
$$H_h|x|^2 = 4a_{jk}x_j \xi_k$$
and
$$H_h^2|x|^2 = 8 a_{lm}a_{lk}\xi_m\xi_k + 8a_{lm}(\partial_{x_l}a_{jk})x_j\xi_m\xi_k - 4 (\partial_{x_l}a_{mp})a_{jl}x_j\xi_m\xi_p= \text{I}+\text{II}+\text{III}$$
We have that $\text{I} =8 |A\xi|^2 \geq 8\lambda^{-2}|\xi|^2$, since
$$|A\xi| \geq A\xi\cdot \frac{\xi}{|\xi|}=\frac{1}{|\xi|}a_{jk}\xi_j\xi_k \geq \lambda^{-1}|\xi|$$
 and by the asymptotic flatness and ellipticity assumptions, that for $|x|\geq M$, with $M$ chosen sufficiently large (chosen in terms of $C_1,\lambda$),
$$|\text{II}| + |\text{III}| \leq \lambda^{-2}|\xi|^2$$
and hence we get, for $|x|\geq M$,
$$H_h^2|x|^2 \geq \lambda^{-2}|\xi|^2$$
Therefore, for $(x,\xi)\in G$, we have
$$H_h q_1 \geq \begin{cases} 0 & \text{for }|x|\leq M+1 \\ \lambda_1^2 & \text{for }|x|\geq M+1 \end{cases}$$
Moreover, if $(x,\xi)\in G$, then
$$H_h|x|^2 = 4a_{jk}x_j\xi_k \leq 4|a_{jk}||x||\xi| \leq 8\lambda |x||\xi| \leq 8\lambda \lambda_2|x|$$
(the fact the $|a_{jk}|\leq 2\lambda$ was explained in Problem 1), and thus $|q_1(x,\xi)| \leq C|x|$.

Now we construct $q_2$.  Let $G_2 \equiv \{ \; (x,\xi) \; \big| \; |x|\leq M+1, \lambda_1 \leq |\xi| \leq \lambda_2 \; \}$, so that $G_2$ is compact.  For each $(x_0,\xi_0)\in G_2$, there is a neighborhood $U_{(x_0,\xi_0)}$ and a time $t_{(x_0,\xi_0)}$ such that  $\forall \; (x,\xi) \in U_{(x_0,\xi_0)}$, the flow $(X(t,x,\xi),\Xi(t,x,\xi))$ with initial position $(x,\xi)$ satisfies $|X(t_{(x_0,\xi_0)},x,\xi)| \geq M+2$.  (Here, we used that the flow map is continuous, and the nontrapping assumption).  Since $G_2$ is compact, there is a finite cover $U_1, \ldots U_k$, and associated escape times $t_1, \ldots t_k$.  Let $\alpha_j$ be a partition of unity subordinate to this cover.  Let 
\begin{equation} \label{PE:232}
q_2(x,\xi) = \sum_{j=1}^k \int_{t=0}^{t_j} \alpha_j( X(x,\xi,-t), \Xi(x,\xi,-t) ) \, dt
\end{equation}
We have
\begin{align}
H_h q_2(x,\xi) &= \frac{d}{ds}\Big|_{s=0} q_2(X(x,\xi,s),\Xi(x,\xi,s)) \notag\\
&= \frac{d}{ds}\Big|_{s=0} \sum_{j=1}^k \int_{t=0}^{t_j} \alpha_j( X(x,\xi,s-t), \Xi(x,\xi,s-t)) \, dt \notag\\
&= \sum_{j=1}^k \int_{t=0}^{t_j} -\frac{d}{dt} \alpha_j(X(x,\xi,s-t), \Xi(x,\xi,s-t)) \, dt \Big|_{s=0} \notag\\
&= \sum_{j=1}^k\alpha_j(x,\xi) - \sum_{j=1}^k \alpha_j(X(x,\xi,-t_j),\Xi(x,\xi,-t_j)) \label{PE:230}
\end{align} 
If $(x,\xi)\in G_2$, then $\sum_{j=1}^k \alpha_j(x,\xi)=1$.  Also, if $(x,\xi)\in G_2$, we have $$\alpha_j(X(x,\xi,-t_j),\Xi(x,\xi,-t_j)) = 0 \qquad \forall j=1,\ldots k$$  Indeed, if $\neq 0$, then  $(X(x,\xi,-t_j),\Xi(x,\xi,-t_j))\in U_j$, and by flowing forward by $t_j$, we would arrive at $(x,\xi)$, a point in $G_2$, contradicting the definition of $t_j$, $U_j$.  Hence,  for $(x,\xi)\in G$,
$$H_hq_2 \geq \begin{cases} 1 & \text{for }|x|\leq M+1 \\ -1 & \text{for }|x| \geq M+1 \end{cases}$$
We see from \eqref{PE:232} that $|q_2(x,\xi)| \leq \sum_{j=1}^k t_j \leq C$.

Set $q=2\lambda_1^{-2}q_1+q_2$, so that $q$ satisfies \eqref{PE:231}.  Then, for all $(x_0,\xi_0)\in K$, we have
$$1 \leq (H_hq)(X(x_0,\xi_0,t),\Xi(x_0,\xi_0,t)) = \frac{d}{dt} q( X(x_0,\xi_0, t),\Xi(x_0,\xi_0,t))$$
which, when integrated, gives
$$t \leq q(X(x_0,\xi_0,t),\Xi(x_0,\xi_0,t))-q(x_0,\xi_0)$$
This gives, for $(x_0,\xi_0)\in K$, that
$$t+q(x_0,\xi_0) \leq q(X(x_0,\xi_0,t),\Xi(x_0,\xi_0,t)) \leq C(1+|X(x_0,\xi_0,t)|)$$
Letting $R=\inf_{(x_0,\xi_0)\in K} q(x_0,\xi_0)$, we have
$$t+R \leq C(1+|X(x_0,\xi_0,t)|) \qquad \forall \; (x_0,\xi_0)\in K$$

\noindent \textbf{Problem 3}.  We will be working with the system
$$\partial_t u =
\begin{aligned}[t]
&-\epsilon \Delta^2 u + i \Delta u + \vec{b}_1(x)\cdot \nabla u + \vec{b}_2(x)\cdot \nabla \bar{u} \\
&+ c_1(x) u + c_2(x) \bar{u} + f
\end{aligned}
$$
which, with $\vec{w}=\begin{bmatrix} u \\ \bar{u} \end{bmatrix}$, gives the system
$$\partial_t \vec{w} = -\epsilon \begin{bmatrix} \Delta^2 & 0 \\ 0 & \Delta^2 \end{bmatrix} \vec{w} + i \begin{bmatrix} \Delta & 0 \\ 0 &-\Delta \end{bmatrix} \vec{w} + \begin{bmatrix} \vec{b}_1\cdot \nabla & \vec{b}_2 \cdot \nabla \\ \overline{\vec{b}}_2 \cdot \nabla & \overline{\vec{b}}_1 \cdot \nabla \end{bmatrix}\vec{w} + \begin{bmatrix} c_1 & c_2 \\ \bar{c}_2 & \bar{c}_1 \end{bmatrix} \vec{w} + \begin{bmatrix} f \\ \bar{f} \end{bmatrix}$$
Put 
$$S= \begin{bmatrix} 0 & -\tfrac{1}{2} i (\vec{b}_2 \cdot \nabla) \Delta_R^{-1} \\ \tfrac{1}{2} i (\overline{\vec{b}}_2\cdot \nabla) \Delta_R^{-1} & 0 \end{bmatrix}$$
 $\Lambda = I - S$, and $\vec{z}=\Lambda \vec{w}$ ($\Lambda$ is of order $0$).  $R$ will be chosen large to make it invertible in the appropriate spaces.  We now calculate the equation solved by $\vec{z}$.  
$$\partial_t \vec{z} = \Lambda \left\{ -\epsilon \begin{bmatrix} \Delta^2 & 0 \\ 0 & \Delta^2 \end{bmatrix} \vec{w} + i \begin{bmatrix} \Delta & 0 \\ 0 &-\Delta \end{bmatrix} \vec{w} + \begin{bmatrix} \vec{b}_1\cdot \nabla & \vec{b}_2 \cdot \nabla \\ \overline{\vec{b}}_2 \cdot \nabla & \overline{\vec{b}}_1 \cdot \nabla \end{bmatrix}\vec{w} + \begin{bmatrix} c_1 & c_2 \\ \bar{c}_2 & \bar{c}_1 \end{bmatrix} \vec{w} + \begin{bmatrix} f \\ \bar{f} \end{bmatrix} \right\}$$
\noindent \textbf{Term I}.
$$ -\epsilon \Lambda \begin{bmatrix} \Delta^2 & 0 \\ 0 & \Delta^2 \end{bmatrix} \vec{w} = -\epsilon  \begin{bmatrix} \Delta^2 & 0 \\ 0 & \Delta^2 \end{bmatrix} \vec{z} + \epsilon R_1\vec{z}$$
where $R_1$ is order $2$. \\
\noindent \textbf{Term II}.  
\begin{align*}
& i\Lambda \begin{bmatrix} \Delta & 0 \\ 0 & -\Delta \end{bmatrix} \vec{w} = 
i \begin{bmatrix} \Delta & 0 \\ 0 & -\Delta \end{bmatrix}\vec{z} \\
&+ i \begin{bmatrix} 0 & -\tfrac{1}{2}i \Delta (\vec{b}_2 \cdot \nabla) \Delta_R^{-1} - \tfrac{1}{2} i (\vec{b}_2 \cdot \nabla ) \Delta_R^{-1}\Delta \\ -\tfrac{1}{2}i \Delta (\vec{b}_2 \cdot \nabla) \Delta_R^{-1} - \tfrac{1}{2} i (\vec{b}_2 \cdot \nabla ) \Delta_R^{-1}\Delta & 0 \end{bmatrix} \vec{w} \\
& = i \begin{bmatrix} \Delta & 0 \\ 0 & -\Delta \end{bmatrix}\vec{z} + \begin{bmatrix} 0 & -(\vec{b}_2 \cdot \nabla) \chi_{|\xi|\geq R} \\ -(\overline{\vec{b}}_2 \cdot \nabla) \chi_{|\xi|\geq R} & 0 \end{bmatrix}\vec{w} + R_2 \vec{z}
\end{align*}
where $R_2$ is of order $0$.\\
\noindent \textbf{Term III}.  
$$\Lambda \begin{bmatrix} \vec{b}_1 \cdot \nabla & 0 \\ 0 & \overline{\vec{b}}_2 \cdot \nabla \end{bmatrix} \vec{w} = \begin{bmatrix} \vec{b}_1 \cdot \nabla & 0 \\ 0 & \overline{\vec{b}}_2 \cdot \nabla \end{bmatrix}\vec{z} + R_3 \vec{z}$$
where $R_3$ is of order $0$.\\
\noindent \textbf{Term IV}.
$$\Lambda \begin{bmatrix} c_1 & c_2 \\ \bar{c}_2 &\bar{c}_1 \end{bmatrix} \vec{w} = R_4 \vec{z}$$
where $R_4$ is of order $0$.\\
\noindent \textbf{Term V}.
$$\Lambda \begin{bmatrix} f \\ \bar{f} \end{bmatrix} =: \vec{g}$$
Thus, the equation that $z$ solves is
$$\partial_t \vec{z} = -\epsilon \begin{bmatrix} \Delta^2 & 0 \\ 0 & \Delta^2 \end{bmatrix} \vec{z} + i \begin{bmatrix} \Delta & 0 \\ 0 & -\Delta \end{bmatrix} \vec{z} + \begin{bmatrix} \vec{b}_1 \cdot \Delta & 0 \\ 0 & \overline{\vec{b}}_1\cdot \Delta \end{bmatrix} \vec{z} + \epsilon R_1\vec{z} + R\vec{z} + \vec{g}$$
grouping order $0$ operators as $$R=R_2+R_3+R_4+ \begin{bmatrix} 0 & (\vec{b}_2 \cdot \nabla)\chi_{|\xi|\leq R} \\ \overline{\vec{b_2}}\cdot \nabla \chi_{|\xi|\leq R} & 0 \end{bmatrix}$$
Take $R$ large so that $\Lambda$ is invertible in the space $L^2_x$, i.e.\
$$\|\Lambda\vec{w}\|_{L_x^2} \sim \|\vec{w}\|_{L_x^2}$$
Also, since $S$ is of order $-1$, by taking $R$ large we have
$$\| J^{1/2} \Lambda \vec{w} \|_{L^2(\mathbb{R}^n\times [0,T]; \lambda_m(x)dxdt)} \sim \| J^{1/2} \vec{w} \|_{L^2(\mathbb{R}^n\times [0,T]; \lambda_m(x)dxdt)} $$
and
$$\tv J^{1/2} \Lambda \vec{w} \tv_T \sim \tv J^{1/2} \vec{w} \tv_T$$
Let me elaborate on the last statement.  We need to show that for $R$ large,
$$\tv J^{1/2}SJ^{-1/2}v \tv_T \leq \frac{1}{2} \tv v \tv_T$$
We have $J^{1/2}SJ^{-1/2}=J_R^{-1}(J^{3/2}SJ^{-1/2})$, where $\widehat{J_R^{-1}v} = \chi_{|\xi|\geq R}|\xi|^{-1}\hat{v}(\xi)$ and thus we have
$$\tv J^{1/2}SJ^{-1/2} \tv_T \leq \frac{1}{R} \tv J^{3/2}SJ^{-1/2}v \tv_T \leq \frac{1}{R} \tv v \tv_T$$
where, for the first inequality, we have used Cor \ref{C:1}, where the bound depends on $S^0$ seminorms of $J_R^{-1}$, which are clearly $\leq \frac{1}{R}$.
Problem \chapterref{P:4}.\ref{P:4} gives (in Case (1))
$$\sup_{0<t<T} \|\vec{z}(t) \|_{L^2} + \| J^{1/2}\vec{z} \|_{L^2(\mathbb{R}^n\times [0,T]; \lambda_m(x)dxdt)} \leq C_T \left\{ 
\begin{aligned} 
& \|\vec{z}_0 \|_{L^2} + \|\vec{g}\|_{L_T^1L_x^2} \\ 
& \|\vec{z}_0 \|_{L^2} + \|J^{-1/2} \vec{g} \|_{L^2(\mathbb{R}^n\times [0,T]; \lambda_m^{-1}(x) \, dxdt)}
\end{aligned}\right. $$
and, in case (2),
$$\sup_{0<t<T} \|\vec{z}(t) \|_{L^2} + \tv J^{1/2}\vec{z} \tv_T \leq C_T \left\{ 
\begin{aligned} 
& \|\vec{z}_0 \|_{L^2} + \|\vec{g}\|_{L_T^1L_x^2} \\ 
& \|\vec{z}_0 \|_{L^2} + \tv J^{-1/2} \vec{g} \tv_T'
\end{aligned}\right. $$
We must also check
\begin{gather}
\| \vec{g}\|_{L_T^1 L_x^2} \leq \|\vec{f}\|_{L_T^1L_x^2} \label{PE:401}\\
\|J^{-1/2}\vec{g} \|_{L^2(\mathbb{R}^n\times [0,T]; \lambda_m^{-1}(x)dxdt)} \leq \|J^{-1/2}\vec{f} \|_{L^2(\mathbb{R}^n\times [0,T]; \lambda_m^{-1}(x)dxdt)} \label{PE:402}\\
\tv J^{-1/2}\vec{g} \tv_T' \leq \tv J^{-1/2}\vec{f} \tv_T' \label{PE:400}
\end{gather}
The proof of \eqref{PE:400} is
$$\tv J^{-1/2} \Lambda^{-1} \vec{f} \tv_T' = \tv (J^{-1/2}\Lambda^{-1} J^{1/2})(J^{-1/2}f) \tv_T' \leq C\tv J^{-1/2}f \tv_T'$$
\eqref{PE:402} follows similarly from Theorem \ref{T:7}, and \eqref{PE:401} is just Theorem \ref{T:1} directly.\\

\noindent \textbf{Problem 4}.  We have thus now proved Theorem \ref{T:S3}, and the goal is to use this to prove:
\begin{theorem}[Extension of Theorem \ref{T:S1}]  Let $d\geq 3$.  Then $\forall \; u_0 \in H^s(\mathbb{R}^n)$, $s\geq s_0=n+6+\frac{1}{2}$, we have local well-posedness in $H^s$ of
$$\partial_t u = i\Delta u + P(u,\bar{u},\nabla_x u, \nabla_x \bar{u} )$$
\end{theorem}
We shall restrict to the equation
\begin{equation} \label{PE:40}
\partial_t u = i \Delta u + \partial_x u \cdot \partial_x u \cdot \partial_x u
\end{equation}
where $\partial_x=\partial_{x_k}$ for some $k$.  Take $s_0=n+6+\frac{1}{2}$.  Consider the space $Z_T^a = \{ \; v: \mathbb{R}^n \times [0,T], \lambda_j^T(v) \leq a, j=1,2,3,4 \; \}$, $T\leq 1$, where
\begin{align*}
\lambda_1^T(v) &= \sup_{0\leq t \leq T} \|v(t)\|_{H^{s_0}}\\
\lambda_2^T(v) &= \sum_{|\beta|=s_0+\frac{1}{2}} \tv \partial_x^\beta v \tv_T \\
\lambda_3^T(v) &= \sum_{|\beta|\leq 3} \|\partial_x^\beta v \|_{l^2(L^\infty(Q_\mu\times [0,T]))} \\
\lambda_4^T(v) &= \| \partial_t \nabla v \|_{l^2(L^\infty(Q_\mu\times [0,T]))}
\end{align*}
If $v\big|_{t=0}=u_0$, then set \begin{equation*}
u=\Gamma v = e^{it\Delta} u_0 + \int_0^t e^{i(t-t')\Delta} \partial_xv \cdot \partial_x v \cdot \partial_x v (t') \, dt'
\end{equation*}
 so that $u$ solves the linear inhomogeneous problem
\begin{equation}\label{PE:403}
\left\{ \begin{aligned}
\partial_t u &= i\Delta u + \partial_x v \cdot \partial_x v \cdot \partial_x v\\
u\big|_{t=0} &= u_0
\end{aligned}
\right.
\end{equation}
We shall prove that $\Gamma$ is a contraction on $Z_a^T$ so that $\exists \; u$ such that $\Gamma u = u$, which will be a solution to our nonlinear problem.  Apply $J^{1/2}\partial_\beta$ to \eqref{PE:403}, $|\beta|=s_0-\frac{1}{2}$.  This gives
\begin{align*}
\partial_t (J^{1/2} \partial_x^\beta u) &= 
\begin{aligned}[t]
&i\Delta (J^{1/2}\partial_x^\beta u) + J^{1/2} [ \partial_x v \cdot \partial_x v \cdot \partial_x \partial_x^\beta v ] \\
&+ \sum_{\substack{ \beta_1+\beta_2+\beta_3=\beta \\ |\beta_i|\leq s_0-\frac{3}{2}}} C_{\beta_i} J^{1/2} [ \partial_x (\partial_x^{\beta_1}v) \cdot  \partial_x (\partial_x^{\beta_2}v) \cdot \partial_x (\partial_x^{\beta_3}v)] 
\end{aligned}\\
&= 
\begin{aligned}[t]
&i\Delta (J^{1/2}\partial_x^\beta u) + J^{1/2} [ \partial_x u_0 \cdot \partial_x u_0 \cdot \partial_x \partial_x^\beta v ] + J^{1/2} [ (\partial_x v \cdot \partial_x v - \partial_x u_0 \cdot \partial_x u_0) \partial_x \partial_x^\beta v ]\\
&+ \sum_{\substack{ \beta_1+\beta_2+\beta_3=\beta \\ |\beta_i|\leq s_0-\frac{3}{2}}} C_{\beta_i} J^{1/2} [ \partial_x (\partial_x^{\beta_1}v) \cdot  \partial_x (\partial_x^{\beta_2}v) \cdot \partial_x (\partial_x^{\beta_3}v)] 
\end{aligned}
\end{align*}
where, in the last step, we have introduced a first order linear term $J^{1/2} [ \partial_x u_0 \cdot \partial_x u_0 \cdot \partial_x \partial_x^\beta v ]$.  Apply Theorem \ref{T:S3} (2), which gives the bound
\begin{align*}
\| J^{1/2} \partial_x^\beta u \|_{L_x^2} + \tv J \partial_x^\beta v \tv_T &\leq 
\begin{aligned}[t]
&\|J^{1/2} \partial_x^\beta u_0 \|_{L^2} \\
&+ C \sum_{\beta_i} \| J^{1/2} [ \partial_x (\partial_x^{\beta_1}v)] \cdot J^{1/2} [ \partial_x (\partial_x^{\beta_2}v)] \cdot J^{1/2} [ \partial_x (\partial_x^{\beta_3}v) ] \|_{L_T^1 L_x^2} \\
&+ \tv (\partial_x v \cdot \partial_x v - \partial_x u_0 \cdot \partial_x u_0) \partial_x \partial_x^\beta v \tv_T'
\end{aligned}
\end{align*}
The new part of the proof is to estimate $\tv (\partial_x v \cdot \partial_x v - \partial_x u_0 \cdot \partial_x u_0) \partial_x \partial_x^\beta v \tv_T'$.  The trick is to use that
$$\int_0^t \partial_t( \partial_x v \cdot \partial_x v) \, ds = \partial_x v \cdot \partial_x v - \partial_x u_0 \cdot \partial_x u_0$$
We thus get
\begin{align*}
&\tv (\partial_x v \cdot \partial_x v - \partial_x u_0 \cdot \partial_x u_0) \partial_x \partial_x^\beta v \tv_T' \\
&\leq \left( \sum_\mu \left\| \int_0^t \partial_t (\partial_x v \cdot \partial_x v) \, ds \right\|_{L^\infty(Q_\mu\times [0,T])} \right) \sup_\mu \| \partial_x \partial_x^\beta v \|_{L^2(Q_\mu\times [0,T])} \\
& \leq T \left( \sum_\mu \left\| \partial_t \partial_x v \right\|_{L^\infty(Q_\mu\times [0,T])}^2 \right)^{1/2} \left( \sum_\mu \left\| \partial_x v \right\|_{L^\infty(Q_\mu \times [0,T])}^2 \right)^{1/2} \sup_\mu \| \partial_x \partial_x^\beta v \|_{L^2(Q_\mu \times [0,T])} \\
&\leq T \lambda_4^T(v)\lambda_3^T(v) \lambda_4^T(v)
\end{align*}
After summing over all $\beta$ such that $|\beta|=s_0-\frac{1}{2}$, we have
$$\lambda_1^T(u) + \lambda_2^T(u) \leq C\|u_0\|_{H^{s_0}} + T \lambda_4^T(v) \lambda_3^T(v) \lambda^T_2(v) + T\lambda_1^T(v)^3$$
Estimate 
$$\lambda_3^T(u) \leq \lambda_1^T(u)+ T\lambda_1^T(v)^3$$
using the stategy of Lemma \ref{L:MX}.  For $\lambda_4^T(u)$, use the equation directly, i.e. write
$$\partial_t \nabla u = i \Delta \nabla u + \partial_x \nabla v \cdot \partial_x v \cdot \partial_x v$$
and conclude
$$\lambda_4^T(u) \leq \lambda_3^T(u) + \lambda_3^T(v)\lambda_1^T(v)^2$$
by Sobolev imbedding.  This shows that for appropriately chosen $a$ and $T$, $\Gamma: Z_a^T \to Z_a^T$.  Similarly, show $\Gamma$ is a contraction on $Z_a^T$ to complete the proof.

\section{Solutions to Lecture 5 problems}

\textbf{Problem 1}.  We shall use Littlewood-Paley theory:  Let $\eta\in C^N(\mathbb{R}^n\backslash \{ 0 \})$ and  $0<\gamma <1$ be such that
\begin{equation}\label{PE:61}
\left| \frac{d^j}{d\xi^j} \eta(\xi) \right| \leq \begin{cases} c_j |\xi|^{\gamma-j} & \text{if }|\xi|\leq 1 \\ c_j |\xi|^{-\gamma -j} & \text{if }|\xi|\leq 1 \end{cases} \qquad 0\leq j \leq N
\end{equation}
and also
\begin{equation} \label{PE:62}
\tilde{c}_1 \leq \sum |\eta(2^{-k}\xi)|^2 \leq c_1 \quad \xi \neq 0
\end{equation}
Let $\widehat{Q_k f}(\xi) = \eta(2^{-k}\xi)\hat{f}(\xi)$.  The expression
$$ \left( \sum |Q_k f |^2 \right)^{1/2}$$ 
is called the discrete Littlewood-Paley function.  We have
\begin{theorem}[Littlewood-Paley]  \label{PT:LP}Suppose $1< p < \infty$.  Then
$$\| f\|_p \simeq \left\| \left( \sum |Q_k f |^2 \right)^{1/2} \right\|_p$$
If $f_k$ is a sequence of functions, then
$$ \left\| \sum Q_k f_k \right\|_p \leq C \left\| \left( \sum |f_k|^2 \right)^{1/2} \right\|_p$$
\end{theorem}
For a proof, see \cite{Duo}.  We shall now apply this to the problem.  Select $\eta(\xi) \geq 0$, $\text{supp }\eta \in \{ \; \xi \in \mathbb{R}^n \; | \; \frac{1}{2} \leq |\xi| \leq 2 \; \}$, such that $\sum_{k=-\infty}^{+\infty} \eta(2^{-k}\xi) = 1$, $\xi \neq 0$.  Define $\widehat{Q_k f}(\xi) = \eta(2^{-k}\xi) \hat{f}(\xi)$.  Let $\tilde{\eta}(\xi)=1$ on $\{ \; \xi \in \mathbb{R}^n \; | \; \frac{1}{4} \leq  |\xi| \leq 4 \; \}$, $\tilde{\eta}(\xi)\geq 0$, $\text{supp }\tilde{\eta} \subset \{ \xi \in \mathbb{R}^n \; | \; \frac{1}{8} \leq |\xi| \leq 8 \; \}$, and define $\widehat{\tilde{Q}_kf}(\xi)= \tilde{\eta}(2^{-k}\xi) \hat{f}(\xi)$.  Let $p(\xi) = \sum_{j\leq -3} \eta(2^{-j}\xi)$, so that $p(\xi)=1$ for $|\xi| \leq \frac{1}{16}$, $\text{supp }p \subset \{ \; \xi \in \mathbb{R}^n \; | \; |\xi| < \frac{1}{4} \}$, and define $\widehat{P_kf}(\xi) = p(2^{-k}\xi)\hat{f}(\xi)$.  Note that $P_kf = \sum_{j\leq k-3} Q_j f$.  Let $\tilde{p}(\xi)\geq 0$, $\tilde{p}(\xi)= 1$ on $\{ \; \xi \in \mathbb{R}^n \; | \; |\xi| \leq 100 \; \}$, and $\text{supp }\tilde{p} \subset \{ \; \xi \in \mathbb{R}^n \; | \; |\xi| \leq 200 \; \}$, and define $\tilde{P}_k$ accordingly.  Note that $Q_k$ and $\tilde{Q}_k$ fall into the Littlewood-Paley framework outlined above.  We record some elementary facts.
\begin{fact}  \label{PF:101}
$Q_k f P_k g = \tilde{Q}_k(Q_kf P_k g )$.
\end{fact} 
\begin{fact} \label{PF:102}For $|j|\leq 2$, $Q_k f Q_{k-j} g = \tilde{P}_k (Q_k f Q_{k-j} g)$.\end{fact} 
\begin{fact} \label{PF:103}Let $\varphi \in S(\mathbb{R}^n)$, $\varphi_k = 2^{-kn}\varphi(2^{-k}x)$.  Then $\sup_k |(\varphi_k \ast f)(x)| \leq CMf(x)$, where $M$ is the Hardy-Littlewood maximal function.
\end{fact}

We next record some important consequences of the boundedness of the Hardy-Littlewood maximal function and Littlewood-Paley theory.
\begin{lemma} \label{PL:101} If $1<p\leq \infty$, then
\begin{align*}
\| \sup_k |Q_k f | \|_p &\leq C_p \|f\|_p \\
\| \sup_k |\tilde{Q}_k f | \|_p &\leq C_p \|f\|_p \\
\| \sup_k |P_k f | \|_p &\leq C_p \|f\|_p
\end{align*}
\end{lemma}
\begin{proof}These three statements follow from Fact \ref{PF:103}.
\end{proof}
\begin{lemma}  \label{PL:102} If $1<p<\infty$, then
\begin{align*}
\left\| \sum_k \tilde{Q}_k ( Q_k f Q_k g ) \right\|_p &\leq C \|f \|_{p_1} \|g\|_{p_2} \quad \text{ for } \left\{ \begin{gathered} 1<p_1 \leq \infty \\ 1< p_2 \leq \infty \end{gathered} \right. \\
\left\| \sum_k \tilde{Q}_k ( Q_k f P_k g ) \right\|_p &\leq C \|f \|_{p_1} \|g\|_{p_2} \quad \text{ for } \left\{\begin{gathered} 1<p_1 < \infty \\ 1< p_2 \leq \infty \end{gathered}\right.
\end{align*}
where $\frac{1}{p_1}+\frac{1}{p_2}=1$.
\end{lemma}
\begin{proof}  We carry out the proof of the first inequality (the second is very similar).  Since the roles of $f$ and $g$ are symmetric, we may assume that $1<p_1<\infty$ and $1<p_2 \leq \infty$.  Then
\begin{align*}
\left\| \sum_k \tilde{Q}_k (Q_kf Q_k g) \right\|_p &\leq c \left \| \left( \sum_k |Q_k f Q_k g |^2 \right)^{1/2} \right\|_p \quad \text{by Theorem \ref{PT:LP}}\\
& \leq \left\| c \left( \sum_k |Q_kf|^2 \right)^{1/2} \sup_k |Q_kg| \right\|_p \\
& \leq \left\| c \left( \sum_k |Q_kf|^2 \right)^{1/2} \right\|_{p_1} \left\| \sup_k |Q_kg| \right\|_{p_2} \\
& \leq \|f\|_{p_1} \|g\|_{p_2} \quad \text{by Theorem \ref{PT:LP} for $f$ and Lemma \ref{PL:101} for $g$}
\end{align*}
\end{proof}

\begin{lemma}  \label{PL:103} If $1<p<\infty$, $1< p_1 \leq  \infty$, $1< p_2 \leq \infty$, then
$$\left\| \sum_k P_k (Q_k f Q_k g ) \right\|_p \leq \|f\|_{p_1} \|g\|_{p_2}$$
\end{lemma}
\begin{proof} We shall only carry out the proof in the case $1<p_1<\infty$, $1<p_2< \infty$.  The case $p_1=\infty$ or $p_2=\infty$ is a little more fancy (see the appendix to \cite{KPV93}).
Test against $h\in L^{p'}$, $\|h\|_{L^{p'}}=1$.
\begin{align*}
\sum_k \int P_k(Q_kf Q_k g) h &= \sum_k \int Q_k f Q_k g P_k h\\
&\leq \int \left( \sum_k |Q_k f|^2 \right)^{1/2} \sup_k |P_k h | \left( \sum_k | Q_k g |^2 \right)^{1/2} \\
& \leq  \left\| \left( \sum_k |Q_k f|^2 \right)^{1/2} \right\|_{p_1} \| \sup_k |P_k h | \|_{p'} \left\| \left( \sum_k |Q_k g |^2 \right)^{1/2}\right\|_{p_2}\\
& \leq \|f\|_{p_1} \|h\|_{p'} \|g\|_{p_2}\quad \text{by Theorem \ref{PT:LP} for $f$ and $g$, Lemma \ref{PL:101} for $h$}
\end{align*}
\end{proof}
Let $\eta^1(\xi) = |\xi|^{-s}\eta(\xi)$, $\eta^3(\xi) = |\xi|^sp(\xi)$, $\tilde{\eta}^3(\xi) = |\xi|^s\tilde{p}(\xi)$, $\eta^4(\xi) = |\xi|^s\eta(\xi)$, $\tilde{\eta}^4(\xi) = |\xi|^s\tilde{\eta}(\xi)$.  Accordingly define the operators $Q^1_k$, $Q^3_k$, $\tilde{Q}^3_k$, $Q^4_k$, $\tilde{Q}^4_k$, and note that they fall into the Littlewood-Paley framework outlined above.
Expand
\begin{align*}
J^s(fg) &= \sum_k J^s( \tilde{Q}_k(Q_kf P_k g)) + \sum_k J^s(\tilde{Q}_k(P_kfQ_k g)) + \sum_{k, |j|\leq 2} J^s (\tilde{P}_k (Q_kfQ_{k-j}g)) \\
&= \text{I}_1+\text{II}_1+\text{III}_1
\end{align*}
We shall show that $\text{I}_1$ and $\text{III}_1$ pose no difficulty.  In fact,
$$\text{I}_1 = \sum_k \tilde{Q}^4_k( Q^1_k J^s f \cdot P_k g)$$
and thus $\| \text{I}_1 \|_{L^2}$ is appropriately bounded by Lemma \ref{PL:102}.  Note that here is the reason that the $L^\infty$ norm can only be put on $g$, and not on $J^s f$.  Also,
$$\text{III}_1 = \sum_k \tilde{Q}^3_k(Q^1_kJ^sf \cdot Q_{k-j}g)$$
and thus, by Lemma \ref{PL:102} it is appropriately bounded.  Expand
\begin{align*}
f J^s g &= \sum_k \tilde{Q}_k (Q_kf J^s P_k g ) + \sum_k \tilde{Q}_k (P_k f J^s Q_k g ) + \sum_{k, |j|\leq 2} \tilde{P}_k(Q_k f J^s Q_{k-j} g)\\
&= \text{I}_2+\text{II}_2 + \text{III}_2
\end{align*}
We shall show that $\text{I}_2$ and $\text{III}_2$ also pose no difficulty. In fact,
$$\text{I}_2 = \sum_k \tilde{Q}_k(Q^1_kJ^sf \cdot Q^3_kg)$$
and therefore, $\|\text{I}_2 \|_{L^2}$ is appropriately bounded by Lemma \ref{PL:102}.  Also,
$$\text{III}_2 = \sum_{k, |j|\leq 2} 2^{-js} \tilde{P}_k( Q_k^1  J^s f \cdot Q_{k-j}^4 g)$$
and therefore, $\|\text{III}_2\|_{L^2}$ is appropriately bounded by Lemma \ref{PL:103}.  It remains to control $\text{II}_1-\text{II}_2$.  
$$J^s(\tilde{Q}_k(P_kfQ_k g))  = \iint_{\xi,\zeta} e^{ix(\xi+\zeta)} |\xi+\zeta|^s \tilde{\eta} (2^{-k}(\xi+\zeta)) p(2^{-k}\xi) \eta(2^{-k}\zeta) \hat{f}(\xi)\hat{g}(\zeta) \, d\zeta \, d\xi$$
$$\tilde{Q}_k (P_k f J^s Q_k g ) = \iint_{\xi,\zeta} e^{ix(\xi+\zeta)} |\zeta|^s \tilde{\eta} (2^{-k}(\xi+\zeta)) p(2^{-k}\xi) \eta(2^{-k}\zeta) \hat{f}(\xi)\hat{g}(\zeta) \, d\zeta \, d\xi$$
and thus
\begin{equation} \label{PE:60}
J^s(\tilde{Q}_k(P_kfQ_k g)) -\tilde{Q}_k (P_k f J^s Q_k g ) = \iint_{\xi,\zeta} e^{ix(\xi+\zeta)} m(2^{-k}\xi,2^{-k}\zeta) |\xi|^s \hat{f}(\xi) \hat{g}(\zeta) \, d\zeta \, d\xi
\end{equation}
where
$$m(\xi,\zeta) = \underbrace{\frac{|\xi+\zeta|^s -|\zeta|^s}{|\xi|}}_{\equiv \tau(\xi,\zeta)} \tilde{\eta}(\xi+\zeta) \frac{|\xi|}{|\xi|^s}p(\xi) \eta(\zeta)$$
Now
$$\frac{|\xi+\zeta|^s - |\zeta|^s}{|\xi|} = s \int_0^1 \frac{|t\xi+\zeta|^{s-2}(t\xi+\zeta)\cdot \xi}{|\xi|} \, dt$$
Since $|\xi|\leq \frac{1}{4}$ and $\frac{1}{2}\leq |\zeta|\leq 2$, we have $|t\xi+\zeta|\geq \frac{1}{4}$, and thus $\tau(\xi,\zeta)\in C_0^\infty(\mathbb{R}^n\times \mathbb{R}^n)$ [add cutoffs in $\xi$ and $\zeta$ to the definition of $\tau(\xi,\zeta)$], and therefore there is a function $r(\mu,\nu) \in \mathcal{S}(\mathbb{R}^n\times \mathbb{R}^n)$ such that
$$\tau(\xi,\zeta) = \iint_{\mu,\nu} e^{i\mu\xi}e^{i\nu\zeta} r(\mu,\nu) \, d\mu d\nu$$
Inserting into \eqref{PE:60},
\begin{align*}
&J^s(\tilde{Q}_k(P_kfQ_k g)) -\tilde{Q}_k (P_k f J^s Q_k g ) \\
&= \iint_{\mu,\nu} r(\mu,\nu) \iint_{\xi,\zeta} e^{ix(\xi+\zeta)} \tilde{\eta}(2^{-k}(\xi+\zeta))\eta_\mu(2^{-k}\xi) |\xi|^s \hat{f}(\xi) \eta_\nu(2^{-k}\zeta) \hat{g}(\zeta) \, d\zeta \, d\xi \; d\mu \, d\nu
\end{align*}
where
$$\eta_\mu(\xi) = e^{i\mu\xi} p(\xi) \frac{|\xi|}{|\xi|^s} \qquad \eta_\nu(\zeta) = e^{i\nu \zeta} \eta(\zeta)$$
One can check that $\eta_\mu$, $\eta_\nu$, for fixed $\mu$, $\nu$, verify \eqref{PE:61}, \eqref{PE:62} (with constants independent of $\mu$, $\nu$).  Letting $Q_{\mu,k}$, $Q_{\nu,k}$ be the associated operators, we have
$$ J^s(\tilde{Q}_k(P_kfQ_k g)) -\tilde{Q}_k (P_k f J^s Q_k g ) = \iint_{\mu,\nu} r(\mu,\nu) \tilde{Q}_k (Q_{\mu,k} J^s f \cdot Q_{\nu,k} g) \, d\mu \, d\nu$$
and hence
$$\text{II}_1-\text{II}_2 = \iint_{\mu,\nu} r(\mu,\nu) \left[ \sum_k \tilde{Q}_k (Q_{\mu,k}J^sf \cdot Q_{\nu,k}g) \right] \, d\mu d\nu$$
By Fact 5, we have
$$\| \text{II}_1-\text{II}_2 \|_{L^2} \leq \iint_{\mu,\nu} |r(\mu,\nu)| \|J^sf\|_{L^2}\|g\|_{L^\infty} \, d\mu \, d\nu \leq C \|J^sf\|_{L^2}\|g\|_{L^\infty}$$
 
\noindent \textbf{Problem 2}. \underline{Assumptions:} We need $(\text{H}_{1,l})-(\text{H}_{5,l})$ as stated in the lecture, except that in $\text{H}_{2,l}$ we need to replace $N(n)$ by $\max (N(n),s)$. The constants are $\gamma$ (ellipticity), $C$, $C_1$, $B$ (Doi's Lemma).\\
\underline{Statement:}  Let $T_0>0$ be as given by Theorem \ref{T:1B}.  $\forall \; s > 0$, if we make the above assumptions, then $\exists \; A>0, \, K>0$ with dependencies
\begin{align*}
T_0 &= T_0(C,C_1,B)\\
A&= A(\gamma, C_1, B, s) \\
K &= K(\gamma, C, C_1, B, s)
\end{align*}
such that
\begin{equation}\label{PE:220}
\sup_{[0,T]}\| u(t)\|_{H^s}^2 + \|J^{s+\frac{1}{2}}u\|_{L^2(\mathbb{R}^n\times [0,T] ; \lambda_2(x) \, dxdt)}^2 \leq A e^{T_0K}\|u_0\|_{H^s}^2 + 
C\left\{\begin{aligned}
&\|f\|^2_{L_T^1H_x^s} \\ 
&\|J^{s-\frac{1}{2}}f\|_{L^2(\mathbb{R}^n\times [0,T]; \lambda_2^{-1}(x)dxdt)}^2 
\end{aligned}
\right.
\end{equation}
We shall only prove the result for $s=2m$, $m\in \mathbb{N}$, and the proof will be by induction on $m$.  By Theorem \ref{T:1B}, we have the bound \eqref{PE:220} for $m=0$. Apply $\Delta^m$ to the equation:
$$\partial_t \Delta^m u = -\epsilon \Delta^{2+m}u + \Delta^m (ia_{kl}\partial_{x_k}\partial_{x_l}u) + \Delta^m( b_{1,k}\partial_{x_k}u) + \Delta^m(b_{2,k}\partial_{x_k} \bar{u}) + \Delta^m(c_1u) + \Delta^m(c_2\bar{u}) + \Delta^mf$$
We now rearrange terms to put the equation in the form (with $w = \Delta^m u$)
$$\partial_t w = 
\begin{aligned}[t]
&-\epsilon \Delta^2 w + i  a_{lk}\partial_{x_l}\partial_{x_k}w  +  \Psi_{\tilde{b}_1}w + b_{2,k} \partial_{x_k} \bar{w}\\
&+ \Psi_{\tilde{c}_1}w + \Psi_{\tilde{c}_2} \bar{w} + \Delta^m f + f_1
\end{aligned}$$
where $f_1$ will contain terms that we have already bounded. Expand, and group terms as
\begin{align*}
i\Delta^m [ a_{kl}\partial_{x_k}\partial_{x_l}u] &= 
\begin{aligned}[t]
& ia_{kl}\partial_{x_k}\partial_{x_l} \Delta^m u\\
& + i \frac{(\partial_xa_{kl}) \partial_x^{2m-1}\partial_{x_k}\partial_{x_l}}{I+(-\Delta)^m} [I+(-\Delta)^m] u\\
& + \frac{i(\partial_x^2a_{kl}) \partial_x^{2m-2}\partial_{x_k}\partial_{x_l} + \partial_x^3a_{kl} \partial_x^{2m-3}\partial_{x_k}\partial_{x_l}}{I+(-\Delta)^m} [I+(-\Delta)^m] u\\
& \sum_{|r|\leq 2m-4} i(\partial_x^{2m-r}a_{kl}) \partial_x^r\partial_{x_k}\partial_{x_l}u
\end{aligned}\\
&= \text{I}+\text{II}+\text{III}+\text{IV}
\end{align*}
where $\partial_x^r$ represents a term of the type 
$\partial_x^{\sigma_1}\cdots \partial_x^{\sigma_n}$ with $|\sigma|=r$ with coefficient depending on $m$.  (II, III, and IV are actually sums of terms of the specified form).
Note that term II is a first order $\Psi$DO applied to $\Delta^mu$ and thus should be included in $\Psi_{\tilde{b}}$.  Moreover, the symbol $\tilde{b}$ satisfies the conditions
$$|\partial_t \text{Im}\vec{b}_1(x,t)| \leq \frac{C|\xi|}{\langle x \rangle^2}$$and
$$|\text{Im }\vec{b}_1(x,0)| \leq \frac{C_1|\xi|}{\langle x \rangle^2}$$
of $(\text{H}_{4,l})$ since $\partial_{x_j}a_{lk}$ satisfies the conditions
$$|\partial_t \partial_{x_j}a_{lk}(x,t)| \leq \frac{C}{\langle x \rangle^2}$$
and 
$$|\partial_{x_j}a_{lk}(x,0)|\leq \frac{C_1}{\langle x \rangle^2}$$
from $(\text{H}_{3,l})$.  Term III is a zeroth order $\Psi$DO applied to $\Delta^m u$, and thus should be included in $\Psi_{c_1}$.  Term IV is put into $f_1$.  A similar analysis can be applied to $\Delta^m(b_{1,k}\partial_{x_k}u)$ and $\Delta^m(b_{2,k}\partial_{x_k}\bar{u})$.  Let $w=\Delta^mu$.  We then proceed as in the lecture, first converting to a system, introducing $\vec{w}= \begin{bmatrix} w \\ \bar{w} \end{bmatrix}$, then diagonalizing by introducing the order $-1$ operator $S$ and changing variables to $\vec{z}=\Lambda \vec{w} = (I-S)\vec{w}$, and finally constructing the gauged system in the variable $\vec{\alpha}$.  Note that $M$ in the symbol, (and hence $R$) will depend on $s=2m$, since multiplicative factors of $s=2m$ appeared in the first order terms $\tilde{b}_1$.  However, the time interval has been restricted to $[0,T_0]$ independent of $s$.  As in the lecture notes, we get the bound (with $G$ corresponding to $\Delta^mf$ and $G_1$ corresponding to $f_1$), 
\begin{align*}
&\| \vec{\alpha}(t) \|_{L_x^2}^2 + \epsilon \int_0^t \|\Delta \vec{\alpha} \|_{L^2}^2 + \int_0^t \| J^{1/2} \vec{\alpha} \|_{L^2(dx/\langle x \rangle^2 )}^2\\ 
&\leq  \| \vec{\alpha}(0)\|_{L_x^2}^2 + \int_0^t \left[ K\|\vec{\alpha}(s)\|_{L_x^2}^2 + (\|G(s)\|_{L_x^2}+\|G_1(s)\|_{L_x^2})\|\vec{\alpha}(s)\|_{L^2_x}\right] \, ds\\
&\leq \|\vec{\alpha}(0)\|_{L_x^2}^2 + K\int_0^t \|\vec{\alpha}(s)\|_{L_x^2}^2 \, ds + (\|G\|_{L_{T_0}^1L_x^2}+\|G_1\|_{L_{T_0}^1L_x^2})\sup_{[0,T_0]}\|\vec{\alpha}\|_{L_x^2} 
\end{align*}
Thus, with $f(t) = \int_0^t \|\vec{\alpha}(s) \|_{L_x^2}^2 \, ds$, we have $f'(t) \leq c_1 + c_2 f(t)$.  Gronwall's inequality implies $f'(t) \leq c_1e^{c_2t}$, and therefore we get
$$\sup_{[0,T_0]} \|\vec{\alpha} \|_{L_x^2}^2 \leq e^{KT}(\| \vec{\alpha}(0)\|_{L_x^2}^2 + \sup_{[0,T_0]}\|\vec{\alpha}\|_{L_x^2}(\|G\|_{L_{T_0}L_x^2}+\|G_1\|_{L_{T_0}L_x^2}))$$
which converts to the bound 
$$\sup_{[0,T]} \| \Delta^m u(t) \|_{L^2_x}^2 \leq Ae^{TK}\|\Delta^m u(0)\|_{L_x^2}^2 + C\|f\|_{L_{T_0}^1L_x^2}^2 + C\|f_1\|_{L_{T_0}^1L_x^2}^2$$
Since $f_1$ consists of derivatives of total order $\leq m-2$, we can bound this term in terms of $\|u\|_{L_{T_0}^\infty L_x^2}^2$ and $\|\Delta^{m-1}u\|_{L_{T_0}^\infty L_x^2}^2$, which have already been estimated, giving the bound
$$\sup_{[0,T]} \| \Delta^m u(t) \|_{L^2_x}^2 \leq Ae^{TK}\| u(0)\|_{H_x^{2m}}^2 + C\|f\|_{L_{T_0}^1L_x^2}^2 $$

\noindent \textbf{Problem 4}. Suppose $h$ is as stated in the problem, with ellipticity
$$\lambda^{-1}|\xi|^2 \leq h(x,\xi) \leq \lambda |\xi|^2$$
and asymptotic flatness assumptions:
\begin{align*}
&\sup_{|\xi|=1} |\partial_{\xi_j}h(x,\xi)|\leq C_1 \\
& \sup_{|\xi|=1}|\partial_{x_k}\partial_{\xi_j}h(x,\xi)|\leq \frac{C_1}{\langle x \rangle^2} \\
& \sup_{|\xi|=1} |\partial_{x_k}h(x,\xi)|\leq \frac{C_1}{\langle x \rangle^2} \\
& \sup_{|\xi|=1} |\partial_{\xi_k}\partial_{\xi_j} h(x,\xi)|\leq C_1
\end{align*}
The first step is to show that it suffices to construct a smooth function $q(x,\xi)$ such that 
$$H_{\theta h}q(x,\xi) \geq B_1|\xi| - B_2, \quad \forall \; x,\xi$$
$$|\partial_x^\alpha\partial_\xi^\beta q(x,\xi)| \leq C_{\alpha,\beta}\langle x \rangle \langle \xi \rangle^{-|\beta|} \qquad \text{for }|\alpha|+|\beta| \leq N(n)$$
Proceeding as in the notes, we define $f$, $\phi$, $\phi_0$, $\phi_\pm$, $\Psi_0$, $\Psi_{\pm}$, and define
$$p=\frac{q}{\langle x \rangle}\Psi_0 + [ f(|q|) + 2\epsilon] (\Psi_+-\Psi_-) \in S^0$$
We have, for $|\xi| \geq 2$,
$$H_{\theta h} \left( \frac{q}{\langle x \rangle} \right) = \frac{H_{\theta h} q}{\langle x \rangle} - \frac{q(\partial_{\xi_j}h) x_j}{\langle x \rangle^3}$$
Now
$$H_{\theta h}q \geq B_1 |\xi| - B_2$$
$$|\partial_{\xi_j}h| \leq \sup_{|\eta|=1} |\partial_{\xi_j}h(x,\eta)| |\xi| \leq C_1|\xi|$$
and using that, on $\text{supp }\Psi_0$, we have $|q|\leq \epsilon \langle x \rangle$, we have
$$H_{\theta h} \left( \frac{q}{\langle x \rangle}\right) \geq \frac{B_1|\xi| -B_2}{\langle x \rangle} - \frac{C_1\epsilon \langle x \rangle |\xi| |x|}{\langle x \rangle^3} \geq \frac{\tilde{B}_1|\xi|}{\langle x \rangle} - \tilde{B}_2$$
for $\epsilon$ small in terms of $C_1$.  Just as in the notes, we obtain (after a few steps)
$$H_{\theta h}p \geq \frac{\tilde{B}_1|\xi|}{\langle x \rangle} - \tilde{B}_2$$
We now discuss the modifications necessary in the construction of $q$.  As in the notes, set
$$ q_1 = \langle \xi \rangle^{-1} \psi(|x|^2)H_h(|x|^2)$$
(where $\psi(t)=0$ for $t\leq M^2$, $\psi(t)=1$ for $t\geq (M+1)^2$, and $\psi'(t) \geq 0$)
and we set out to bound
$$H_h(q_1)= \langle \xi \rangle^{-1} \psi'(|x|^2)[H_h(|x|^2)]^2 + \langle \xi \rangle^{-1}\psi(|x|^2)H_h(H_h(|x|^2))\geq \langle \xi \rangle^{-1}\psi(|x|^2)H_h(H_h(|x|^2))$$
from below.  Now
$$H_h(|x|^2)=2\partial_{\xi_j}hx_j$$
$$H_hH_h(|x|^2) = 2\partial_{\xi_k}h\partial_{x_k}\partial_{\xi_j}hx_j + 2\partial_{\xi_k}h\partial_{\xi_k}h - 2 \partial_{x_k}h (\partial_{\xi_k}\partial_{\xi_j}h) x_j$$
Now $2\partial_{\xi_k}h\partial_{\xi_k}h = 2|\nabla_\xi h|^2$, and 
$$|\nabla_\xi h| \geq \nabla_{\xi}h\cdot \frac{\xi}{|\xi|} = \partial_r [h(x,r\xi')] = 2h(x,\xi')r = \frac{2h(x,\xi)}{|\xi|} \geq 2\lambda^{-1} |\xi|$$
by homogeneity and ellipticity.  From the asymptotic flatness bounds
$$2\partial_{x_k}h \cdot \partial_{\xi_k}\partial_{\xi_j}h \cdot x_j \leq (\sup_{|\eta|=1} |\partial_{x_k}h(x,\eta)|)|\xi|^2 ( \sup_{|\eta|=1} |\partial_{\xi_k}\partial_{\xi_j}h(x,\eta)|)\langle x \rangle \leq \frac{C_1|\xi|^2 \langle x \rangle}{\langle x \rangle^2}$$
and similarly,
$$2 \partial_{\xi_k}h \cdot \partial_{x_k}\partial_{\xi_j}h \cdot x_j \leq \frac{C_1|\xi|^2}{\langle x \rangle}$$
and thus:
$$H_hH_h(|x|^2) \geq 8\lambda^{-2}|\xi|^2 - \frac{C_1|\xi|^2}{\langle x \rangle}\geq \lambda^{-2}|\xi|^2$$
for $\langle x \rangle$ large (which dictates the choice of $M$), giving that
$$H_h(q_1) \geq \lambda^{-2} \psi(|x|^2) |\xi|$$
We need, as before, the lemma
\begin{lemma}
Let $K\subset \mathbb{R}^n\times \mathbb{R}^n\backslash \{ 0\}$ be a compact set.  Then, given $\mu>0$, $\exists \; s_0=s_0(K,\mu)$ such that $|X(s,x_0,\xi)|\geq \mu$, $\forall |s| > |s_0|$, $\forall (x_0, \xi_0)\in K$.  
\end{lemma}
With $\phi_1(x)=1$ on $|x|\leq M+1$, $\phi_1(x)=0$ for $|x|\geq M+2$, set
$$q_2(x,\xi) = -\int_0^{+\infty} \phi_1(X(s,x,\xi)) \langle \Xi(s,x,\xi) \rangle \, ds$$
The homogeneity of the flow properties still hold, since $h(x,\xi)$ is homogeneous of degree $2$, and therefore
$$q_2(x,\xi) = -|\xi|^{-1}\int_0^{+\infty} \phi_1(X(s,x,\frac{\xi}{|\xi|})) \langle |\xi|\Xi(s,x,\frac{\xi}{|\xi|}) \rangle \, ds$$
With $\phi_2(\xi)=0$ for $|\xi|\leq 1$, $\phi_2(\xi)=1$ for $|\xi|\geq 2$, 
\begin{align*}
q_3(x,\xi) &= \phi_1(x)\phi_2(\xi) q_2(x,\xi)\\
&= \frac{\phi_1(x)\phi_2(\xi)}{|\xi|} \int_0^{+\infty} \phi_1(X(s,x,\frac{\xi}{|\xi|})) \langle |\xi| \Xi(s; x, \frac{\xi}{|\xi|}) \rangle \, ds
\end{align*}
Now we have to compute
\begin{align*}
H_hq_2(x,\xi) &= -\int_0^{+\infty} H_h [ \phi_1(X(s,x,\xi)) \langle \Xi(s,x,\xi) \rangle \, ds \\
& = -\int_0^{+\infty} \frac{d}{ds} [ \phi_1(X(s,x,\xi)) \langle \Xi(s,x,\xi) \rangle ]\, ds\\
&= \phi_1(x)\langle \xi \rangle
\end{align*}
Also let $K= \{ (x,\xi) \; | \; |x|\leq M+1, |\xi|=1 \}$.  Then we get $s_0$ as in the lemma (for $\mu = M+2$), and so
$$|q_3(x,\xi)| \leq \lambda s_0$$
Complete the proof as in the notes by setting $q=Nq_1+q_3$.\\

\noindent \textbf{Problem 5}.  If $n=1$, then $(\text{H}_{1,l})$ implies $\lambda^{-2}\xi_0^2 \leq \xi^2 \leq \lambda^2 \xi_0^2$, and since $\xi(s)$ is continuous, it never changes sign.  Thus
$$\xi_0 >0 \quad \Rightarrow \quad \xi(s)>0 \text{ and } \lambda^{-1}\xi_0 < \xi < \lambda \xi_0$$
$$\xi_0 < 0 \quad \Rightarrow \quad  \xi(s)<0 \text{ and } \lambda \xi_0 < \xi < \lambda^{-1} \xi_0$$
Also, $(\text{H}_{1,l})$ is $\lambda^{-1}\xi^2 \leq a(x) \xi^2 \leq \lambda \xi^2$, and since $\xi\neq 0$, $\lambda^{-1} \leq a(x) \leq \lambda$.  The first flow equation is $\dot{x}=2a(x)\xi$.  In the case $\xi_0>0$, we get $2\lambda^{-2}\xi_0 < \dot{x} < 2\lambda^2\xi_0$, so $2\lambda^{-2}\xi_0s +x_0 < x(s)<2\lambda^2\xi_0 s + x_0$.  If $s\to +\infty$, then $x(s)\to +\infty$, and if $s\to -\infty$, then $x(s)\to -\infty$.  The case $\xi_0<0$ is similar.

\def\cprime{$'$}
\providecommand{\bysame}{\leavevmode\hbox to3em{\hrulefill}\thinspace}
\providecommand{\MR}{\relax\ifhmode\unskip\space\fi MR }
\providecommand{\MRhref}[2]{%
  \href{http://www.ams.org/mathscinet-getitem?mr=#1}{#2}
}
\providecommand{\href}[2]{#2}

\end{document}

%% file: pcmslmod.tex
\makeatletter

\newif\iffirstlecture\firstlecturefalse

\newcommand{\lectureseries}{\firstlecturetrue
              \secdef\@lectureseries\@slectureseries} 

\newcommand{\@lectureseries}[2][default]{\chapter*{#2}%
              \gdef\thelectureseries{#1}} 

\newcommand{\@slectureseries}[1]{\chapter*{#1}}

\renewcommand{\auth}{\secdef\@auth\@sauth}

\newcommand{\@auth}[2][default]{\vspace{-1pc}{\raggedleft
        \Large\bf\noindent
        #2\endgraf
        \vspace*{2pc}
        }
        \def\@author{#1}%
}

\newcommand{\@sauth}[1]{\vspace{-1pc}{\raggedleft
        \Large\bf\noindent
        #1\endgraf
        \vspace*{2pc}
        }
        \def\@author{#1}%
}

\def\lecture#1{\global\Lecturetrue\global\Monographfalse
\iffirstlecture\else\chapter*{}\fi\firstlecturefalse
  \global\let\sectionmark\@gobble 
  \addtocounter{lecture}1\relax
  \refstepcounter{chapter}%
\gdef\thelecturename{#1\unskip}
  {\Large\bfseries
   \raggedleft
   \@xp\uppercase\@xp{\thelecturelabel} {\LARGE\thelecturenum}\\
   \vspace*{3pt}%
   \thelecturename
   \endgraf}%
  \let\@secnumber=\thelecturenum
  \@xp\lecturemark\@xp{\thelecturename}%
  \addcontentsline{toc}{chapter}{%
    \thelecturelabel\ \thelecturenum.\ \thelecturename}%
  \vspace{34\p@}\noindent}
  
\def\lecture{\global\Lecturetrue\global\Monographfalse
\iffirstlecture\else\chapter*{}\fi%
  \global\let\sectionmark\@gobble 
\secdef\@lecture\@slecture}

\def\@lecture[#1]#2{%
  \addtocounter{lecture}1\relax
  \refstepcounter{chapter}%
\gdef\thelecturename{#1\unskip}\firstlecturefalse
  {\Large\bfseries
   \raggedleft
   \@xp\uppercase\@xp{\thelecturelabel} {\LARGE\thelecturenum}\\
   \vspace*{3pt}%
    #2\unskip
   \endgraf}%
  \let\@secnumber=\thelecturenum
  \@xp\lecturemark\@xp{\thelecturename}%
  \addcontentsline{toc}{chapter}{%
    \thelecturelabel\ \thelecturenum.\ #2}%
  \vspace{34\p@}\noindent}
  
\def\slecturerunhead#1#2#3{%
    \let\@tempa\chaptername
    \uppercasenonmath{\@tempa}%
    \def\@tempb{#3\unskip}%
    \uppercasenonmath{\@tempb}%
    {\normalfont\@tempb}
    }
\def\slecturemark{
    \@secmark\markright\slecturerunhead\chaptername}%

\def\@slecture#1{%
\iffirstlecture
\gdef\thelecturename{#1\unskip}\firstlecturefalse
  {\Large\bfseries
\noindent\thelecturename
   \endgraf}%
  \let\@secnumber=\thelecturenum
  \@xp\slecturemark\@xp{\thelecturename}%
  \addcontentsline{toc}{chapter}{%
    \thelecturename}%
 \vspace{-6\p@}\noindent
\else
\gdef\thelecturename{#1\unskip}\firstlecturefalse
  {\Large\bfseries
   \raggedleft
   \@xp\uppercase\@xp{\thelecturename}
   \endgraf}%
  \let\@secnumber=\thelecturenum
  \@xp\slecturemark\@xp{\thelecturename}%
  \addcontentsline{toc}{chapter}{%
    \thelecturename}%
  \vspace{34\p@}\noindent
\fi}

\ifLecture
  \def\chapterrunhead#1#2#3{%
    \let\@tempa\@author
    \uppercasenonmath{\@tempa}%
    \uppercasenonmath{\thelectureseries}%
    \textmd{\@tempa, \thelectureseries}
    }
  \def\lecturerunhead#1#2#3{%
    \let\@tempa\chaptername
    \uppercasenonmath{\@tempa}%
    \def\@tempb{#3\unskip}%
    \uppercasenonmath{\@tempb}%
    \textmd{\@tempa\ #2. \@tempb}
    }
\else
  \let\chapterrunhead\partrunhead
\fi

\newif\ifBibliographyIsASection\BibliographyIsASectionfalse

  \def\bibliomark{
    \@secmark\markright\bibliorunhead\chaptername}%

  \def\bibliorunhead#1#2#3{%
    \let\@tempa\chaptername
    \uppercasenonmath{\@tempa}%
    \def\@tempb{#3\unskip}%
    \uppercasenonmath{\@tempb}%
    \textmd{\@tempb}
    }

\def\thebibliography#1{%
  \ifBibliographyIsASection
    \section*\refname
    \if@backmatter
      \markboth{\refname}{\refname}%
    \fi
  \else
\chapter*{}
  {\Large\bfseries
   \raggedleft
   \@xp\uppercase\@xp{\bibname} \\
   \endgraf}%
  \let\@secnumber=\thelecturenum
  \@xp\bibliomark\@xp{\bibname}%
  \addcontentsline{toc}{chapter}{%
    \bibname}%
  \vspace{34\p@}\noindent
  \fi
  \normalsize\labelsep .5em\relax
  \list{\@arabic\c@enumi.}{\settowidth\labelwidth{\@biblabel{#1}}%
  \leftmargin\labelwidth
  \advance\leftmargin\labelsep
	\usecounter{enumi}}\sloppy
  \clubpenalty9999 \widowpenalty\clubpenalty  \sfcode`\.\@m}

  \def\indexmark{
    \@secmark\markright\indexrunhead\chaptername}%

  \def\indexrunhead#1#2#3{%
    \let\@tempa\chaptername
    \uppercasenonmath{\@tempa}%
    \def\@tempb{#3\unskip}%
    \uppercasenonmath{\@tempb}%
    \textmd{\@tempb}
    }

\ifLecture
\def\theindex{\cleardoublepage
\@restonecoltrue\if@twocolumn\@restonecolfalse\fi
\columnseprule \z@ \columnsep 35pt
\def\indexchap{\@startsection
		{chapter}{1}{\z@}{8pc}{34pt}%
		{\raggedleft
		\Large\bfseries}}%
 \twocolumn[\indexchap[{\indexname}]{\@xp\uppercase\@xp{\indexname}}]
  \@xp\indexmark\@xp{\indexname}%
	\thispagestyle{plain}\let\item\@idxitem\parindent\z@
	 \footnotesize\parskip\z@ plus .3pt\relax\let\item\@idxitem}
\fi

\def\@makefntext{\noindent\@makefnmark}

\def\setaddress{%
  {\let\@makefnmark\relax  \let\@thefnmark\relax
        \nobreak
        \addressnum@=\z@
        \loop\ifnum\addressnum@<\addresscount@\advance\addressnum@\@ne
           \footnote{$^{\hbox{\tiny\number\addressnum@}}$%
           \csname @address\number\addressnum@\endcsname
           \csname @curraddr\number\addressnum@\endcsname
           \csname @email\number\addressnum@\endcsname}\repeat
  \ifx\@empty\@date\else \@footnotetext{\@setdate}\fi
  \ifx\@empty\@subjclass\else \@footnotetext{\@setsubjclass}\fi
  \ifx\@empty\@keywords\else \@footnotetext{\@setkeywords}\fi
  \ifx\@empty\thankses\else \@footnotetext{%
    \def\par{\let\par\@par}\@setthanks}\fi
    }%
  \@setcopyright
}

\def\@tmpevenhead{\relax}

\def\cleardoublepage{\clearpage\if@twoside \ifodd\c@page\else
    \let\@tmpevenhead\@evenhead \let\@evenhead\relax\hbox{}\eject 
    \let\@evenhead\@tmpevenhead\if@twocolumn\hbox{}\newpage\fi\fi\fi}

\def\@setcopyright{%
  \let\copyrightyear\currentyear             
  \insert\copyins{\hsize\textwidth
    \parfillskip\z@ \leftskip\z@\@plus.9\textwidth
    \fontsize{6}{7\p@}\normalfont\upshape
    \everypar{}%
    \vskip-\skip\copyins \nointerlineskip
    \noindent\vrule\@width\z@\@height\skip\copyins
    \copyright\copyrightyear\ American Mathematical Society\par
    \kern\z@}%
}

\renewcommand{\@auth}[2][default]{{\raggedleft
        \begingroup
  \fontsize{\@xivpt}{18}\bfseries
  #2\par \endgroup
        \vspace*{2pc}
        }
        \def\@author{#1}%
}

\renewcommand{\@sauth}[1]{{\raggedleft
        \begingroup
  \fontsize{\@xivpt}{18}\bfseries
  #1\par \endgroup
        \vspace*{2pc}
        }
        \def\@author{#1}%
}

\def\@lecture[#1]#2{%
  \addtocounter{lecture}1\relax
  \refstepcounter{chapter}%
\gdef\thelecturename{#1\unskip}\firstlecturefalse
  {\Large\bfseries
   \raggedleft
   \@xp\uppercase\@xp{\thelecturelabel} {\LARGE\thelecturenum}\\
   \vspace*{3pt}%
    #2\unskip
   \endgraf}%
  \let\@secnumber=\thelecturenum
  \@xp\lecturemark\@xp{\thelecturename}%
  \addcontentsline{toc}{chapter}{%
    \thelecturelabel\ \thelecturenum.\ #2}%
  \vspace{10\p@}\noindent}
  
\def\@slecture#1{%
\iffirstlecture
\gdef\thelecturename{#1\unskip}\firstlecturefalse
  {\Large\bfseries
\noindent\thelecturename
   \endgraf}%
  \let\@secnumber=\thelecturenum
  \@xp\slecturemark\@xp{\thelecturename}%
  \addcontentsline{toc}{chapter}{%
    \thelecturename}%
 \vspace{-6\p@}\noindent
\else
\gdef\thelecturename{#1\unskip}\firstlecturefalse
  {\Large\bfseries
   \raggedleft
   \@xp\uppercase\@xp{\thelecturename}
   \endgraf}%
  \let\@secnumber=\thelecturenum
  \@xp\slecturemark\@xp{\thelecturename}%
  \addcontentsline{toc}{chapter}{%
    \thelecturename}%
  \vspace{10\p@}\noindent
\fi}

\makeatother